\documentclass{article}

\newcommand{\Db}{D^h_k}
\newcommand{\Dbn}{D^{-h}_k}
\newcommand{\Ltw}{L^2(\reall^2)}
\newcommand{\intt}{\int_{\reall^2}}

\newcommand{\Ve}{V_\eps}

\usepackage{graphicx} 
\usepackage{amsmath,mathtools,amsthm,fancyhdr,amssymb,bbm, enumerate,float, mathrsfs}
\usepackage{tcolorbox}
\usepackage{appendix}
\tcbuselibrary{breakable}
\usepackage{listings}
\usepackage{color, comment}
\usepackage{accents}
\usepackage{hyperref}
\usepackage{indentfirst}
\hypersetup{
colorlinks,
linkcolor=blue,
anchorcolor=blue,
citecolor=blue
}

\mathtoolsset{showonlyrefs,showmanualtags} 

\newcommand{\reall}{\mathbb{R}}
\newcommand{\Sp}{\mathbb{S}}

\newcommand{\Z}{\mathbb{Z}}
\renewcommand{\P}{\mathbb{P}}

\newcommand{\scl}[2]{\left\langle #1\,,\, #2\right\rangle}
\newcommand{\bkt}[1]{\langle #1 \rangle}
\newcommand{\om}{\omega}

\newcommand{\vp}{\varphi}
\newcommand{\dr}[1]{\left(#1\right)}

\newcommand{\dn}[1]{\left\|#1\right\|}
\newcommand{\da}[1]{\left|#1\right|}
\newcommand{\Rmnum}[1]{\uppercase\expandafter{\romannumeral#1}}
\newcommand{\eps}{\varepsilon}
\newcommand{\pa}{\partial}

\renewcommand{\div}{\mathrm{div}\,}
\newcommand{\supp}{\mathrm{supp}}
\newcommand{\PV}{\mathrm{P.V.}}

\newtheorem{proposition}{Proposition}[section]
\newtheorem{lemma}{Lemma}[section]

\newtheorem{theorem}{Theorem}[section]

\numberwithin{equation}{section}

\newcommand{\Ac}{\mathcal{A}}

\newcommand{\Cc}{\mathcal{C}}

\newcommand{\Fc}{\mathcal{F}}

\newcommand{\Hc}{\mathcal{H}}
\newcommand{\Lc}{\mathcal{L}}
\newcommand{\Oc}{\mathcal{O}}

\newcommand{\Tc}{\mathcal{T}}

\newcommand{\Ls}{\mathscr{L}}

\title{Forward Self-Similar Solutions to the 2D Hypodissipative Navier-Stokes Equations}
\author{Thomas Y. Hou\footnote{Applied and Computational Mathematics, Caltech, Pasadena, CA. Email: \href{hou@cms.caltech.edu}{hou@cms.caltech.edu}},\; 
Peicong Song\footnote{Applied and Computational Mathematics, Caltech, Pasadena, CA. Email: \href{psong2@caltech.edu}{psong2@caltech.edu} }}

\date{}

\begin{document}

\maketitle
\begin{abstract}
    We study the forward self-similar solutions to the $2$D hypodissipative Navier-Stokes equation with fractional diffusion $(-\Delta)^\alpha$ for $\frac{1}{2}<\alpha<1$. We first show that for arbitrarily large $(1-2\alpha)$-homogeneous initial data which are locally Lipschitz, there exists at least one weak solution whose profile differs from the self-similar profile of the fractional heat equation by an element of $H^\alpha(\reall^2)$. Moreover, when $\alpha\in(\frac{2}{3},1)$ we show that any such weak solution is actually smooth, hence a strong solution, and satisfies certain far field decay estimates. Finally, we provide numerical evidence for the nonuniqueness of the related $2$D Navier-Stokes equation with time-dependent viscosity.
\end{abstract}

\section{Introduction}
We consider the incompressible fractional Navier-Stokes equations in two dimensions:
\begin{equation}\label{fractional NS}
    \begin{cases}
       \pa_t u + u\cdot\nabla u + \nabla p + \Lambda^{2\alpha}u = 0, \quad (x,t)\in \reall^2\times (0,+\infty)\\
       \div u = 0,
    \end{cases}\tag{NS\textsubscript{$\alpha$}}
\end{equation}
where $\alpha\in(\frac{1}{2},1)$ and $\Lambda^{2\alpha}:= (-\Delta)^{\alpha}$ is the fractional Laplacian operator. Such a system is called hypodissipative for $\alpha<1$, where the diffusion is weaker than the usual Laplacian. We consider the Cauchy problem for \eqref{fractional NS} with some divergence-free initial velocity field:
\begin{equation}
    u(x,0) = u_0(x).
\end{equation}
Existence and uniqueness of solutions to such systems have been established in homogeneous Besov spaces in \cite{Wu2006}.
Similar to the Navier-Stokes equation, the system \eqref{fractional NS} admits the following scaling invariance: if $(u,p)$ is a solution to \eqref{fractional NS} with initial data $u_0$, then for all $\lambda>0$, $(u_\lambda,p_{\lambda})$ is also a solution with initial data $u_{0,\lambda}$, where
\begin{align}\label{eq: scaling invariance}
    u_{\lambda}(x,t) :=\lambda^{2\alpha-1}u(\lambda x,\lambda^{2\alpha}t),\quad p_{\lambda}(x,t):=\lambda^{4\alpha-2}p(\lambda x,\lambda^{2\alpha}t),\quad u_{0,\lambda}(x):=\lambda^{2\alpha-1}u_0(\lambda x). 
\end{align}
A (forward) self-similar solution is defined as a solution $(u,p)$ that is invariant under the scaling transformation above, i.e., $u(x,t)=u_\lambda(x,t)$ and $p(x,t) = p_\lambda(x,t)$ for any $\lambda>0$. Such solutions necessarily take the form:
\begin{equation}\label{self similar change of variables}
    u(x,t) = t^{\frac{1}{2\alpha}-1}U\dr{\frac{x}{t^{\frac{1}{2\alpha}}}},\quad p(x,t)= t^{\frac{2}{2\alpha}-2}P\dr{\frac{x}{t^{\frac{1}{2\alpha}}}},\quad y:=\frac{x}{t^{\frac{1}{2\alpha}}},
\end{equation}
where $U,P$ are called the self-similar profiles and satisfy the following self-similar equation
\begin{align}\label{self similar eq for U}
    \Lambda^{2\alpha}U - \frac{2\alpha-1}{2\alpha}U - \frac{1}{2\alpha}y\cdot\nabla U + U\cdot\nabla U +\nabla P=0\quad \div U = 0.
\end{align} 
Moreover, self-similar solutions arise from
$(1-2\alpha)$-homogeneous initial data, namely
\begin{equation}
    u_0(x) = |x|^{1-2\alpha}\bar u_0\dr{\frac{x}{|x|}},\quad  \div u_0 = 0.
\end{equation}
We note that the initial condition $u_0$ is encoded into the far field behavior of profile $U$:
\begin{align}
    |y|^{2\alpha-1}U(y)= \bar u_0\dr{\frac{y}{|y|}}+o(1)\quad \text{as}\quad y\to +\infty.
\end{align}
Our goal is to establish the existence of self-similar solutions to \eqref{fractional NS} and derive pointwise estimates of the profile $U$. Then, we provide numerical evidence of nonuniqueness for a modified $2$D Navier-Stokes equation \eqref{NS with time dependent visc} which shares the same scaling invariance property as \eqref{fractional NS}.

\subsection{Main result}
In an energy estimate point of view, we expect existence of weak solutions (i.e., the solutions that satisfy \eqref{fractional NS} in the sense of distribution) in the class $H^\alpha(\reall^2)$. 
As for the far field asymptotics, we expect that the leading order of the far field behavior of the profile $U$ is determined by the linear part of the equation, because roughly speaking the nonlinear part decays faster at infinity. Thus, we denote
\begin{equation}
    U_0(x) := e^{-\Lambda^{2\alpha}}u_0, 
\end{equation}
which is the self-similar profile of the fractional heat equation solving
\begin{equation}\label{equation that U_0 solves}
    \Lambda^{2\alpha}U_0 - \frac{2\alpha-1}{2\alpha}U_0 - \frac{1}{2\alpha}y\cdot\nabla U_0 =0.
\end{equation}
Moreover, the regularity of the initial $u_0$ will also affect the decay of $U$, as one can see from the Duhamel expression. Now, let us introduce the main result of this work. 

\begin{theorem}\label{theorem: Main}
    Let $\alpha\in(\frac{1}{2},1)$ and $u_0(x)\in C^{0,1}_{loc}(\reall^2\backslash\{0\})$ be a divergence free, $(1-2\alpha)$-homogeneous data. Then, there exists a self-similar solution $u$ to \eqref{fractional NS} with initial data $u_0$ and profile function $U(x) = u(x,1)$, such that $U-U_0$ belongs to $H^\alpha(\reall^2)$. Moreover, if $\alpha\in(\frac{2}{3},1)$, then any solution profile $U$ such that $U-U_0\in H^\alpha(\reall^2)$ is smooth, and satisfies the following pointwise estimates:\\ 
    (\romannumeral 1)(Low regularity case) When $u_0\in C^{0,1}_{loc}(\reall^2\backslash\{0\})$,
    \begin{align}\label{theorem_main: decay esti of V_low reg 1}
      &|\nabla^k U(x)|\leq C(1+|x|)^{1-2\alpha-\bar k},\quad \bar k = \min(k,1),\\
     &|\nabla^{k}(U(x)-U_0(x))|\leq C(1+|x|)^{1-4\alpha}\log(2+|x|),\quad\forall\; k\in \Z_{\geq 0}.
    \end{align}
    When $u_0\in C^{1,\beta}_{loc}(\reall^2\backslash\{0\})$ for some $\beta\in(0,1]$,
    \begin{align}\label{theorem_main: decay esti of V_low reg 2}
     &|\nabla^k U(x)|\leq C(1+|x|)^{1-2\alpha-\bar k},\quad \bar k = \min(k,1+\beta),\\
     &|\nabla^{k}(U(x)-U_0(x))|\leq C(1+|x|)^{1-4\alpha},\quad\forall\; k\in \Z_{\geq 0}.
    \end{align}
    (\romannumeral 2)(High regularity case)
    When $u_0\in C^{\infty}(\reall^2\backslash\{0\})$,
    \begin{align}\label{theorem_main: decay esti of V_high reg}
     &|\nabla^k U(x)|\leq C(1+|x|)^{1-2\alpha- k},\\
     &|\nabla^{k}(U(x)-U_0(x))|\leq C(1+|x|)^{1-4\alpha-k},\quad\forall\; k\in \Z_{\geq 0}.
    \end{align}
    The constants above depend only on $ u_0,k,\alpha$. 
\end{theorem}

\noindent\textit{Remarks.} The threshold $\alpha=\frac{1}{2}$ is the value above which the background velocity field $U_0$ is decaying and below which it is growing. In $2$D, existence of self-similar solutions holds all the way to this threshold, while in $3$D the threshold for existence is $\alpha=\frac{5}{8}$, determined by the $L^2$-integrability of the source term $U_0\cdot\nabla U_0$ (see \cite{LAI2019981}). The smoothness threshold $\alpha = \frac{2}{3}$ arises from the competing effect of diffusion and nonlinearity of the system.
Formally, if one starts with $U\in H^\alpha$ for $\alpha\in (\frac{1}{2},1)$, due to the Sobolev embedding of multiplication $H^\alpha(\reall^2)\times H^\alpha(\reall^2)\hookrightarrow H^{2\alpha-1}(\reall^2)$, the nonlinear term $\div(U\otimes U)$ belongs to $H^{2\alpha-2}(\reall^2)$. The elliptic operator $(-\Delta)^{\alpha}$ gains regularity by $2\alpha$, so we deduce that $U\in H^{4\alpha-2}(\reall^2)$ and there is a genuine gain if $4\alpha-2>\alpha$, i.e., $\alpha>\frac{2}{3}$. Indeed, by the previous result \cite[Theorem 2.1]{Dong2009OptimalLS}, together with Sobolev embedding one can show that any $H^{\alpha}(\reall^n)$-solution to \eqref{self similar eq for U} is actually smooth for $\alpha>\frac{n+2}{6}$ in general dimension $n\geq 2$. Theorem \ref{theorem: Main} can be viewed as a refinement of this result in $2$D, in the sense that it also gives pointwise decay estimates of the profile and its derivatives. The decay rate $(1+|x|)^{1-4\alpha}$ in the theorem is sharp in the sense that it is of the same order as the source term $U_0\cdot\nabla U_0$. Finally, we remark that for the critical case $\alpha=\frac{2}{3}$, it might be possible to adopt the techniques in \cite{LAI2024JFA} to establish the smoothness and decay estimates of $U$. We do not pursue further in this direction here.

Once the existence of forward self-similar solutions has been established, the natural followup question is the uniqueness of such solutions. To understand such question, we conduct numerical investigation to the so-called $2$D Navier-Stokes equation with time-dependent viscosity:
\begin{equation}\label{NS with time dependent visc}
    \begin{cases}
       \pa_t u + u\cdot\nabla u + \nabla p + t^{\frac{1}{\alpha}-1}(-\Delta)u = 0, \quad (x,t)\in \reall^2\times (0,+\infty),\quad \alpha\in(\frac{1}{2},1)\\
       \div u = 0.
\end{cases}\tag{NS\textsubscript{t,$\alpha$}}
\end{equation}
The Laplacian operator in \eqref{NS with time dependent visc} is local, hence more numerically feasible compared to its nonlocal counterpart in \eqref{fractional NS}. This is the main reason why we study this modified equation instead of the original one.
On the other hand, we note that for any fixed $\alpha$, \eqref{NS with time dependent visc} shares the same scaling invariance property \eqref{eq: scaling invariance} with \eqref{fractional NS}. Thus, we expect that the (non)uniqueness of \eqref{NS with time dependent visc} may shed light on that of \eqref{fractional NS}. Specifically, we have the following numerical observations for \eqref{NS with time dependent visc}: 
\begin{itemize}
    \item For large initial data, unstable eigenmodes appear in the spectrum of the linearized operator around the corresponding self-similar profile.
    \item For large initial data, bifurcation of self-similar profiles emerges. In particular, if the initial data is $k$-fold symmetric ($k\in \Z_{\geq 1}$), the bifurcation breaks such symmetry.
\end{itemize}
We refer the readers for more details to Section \ref{Sec 6} and Appendix \ref{Appendix B}, which are independent from the rest part of this work.

\subsection{Comparison with existing literature}
\subsubsection{Forward self-similar solutions}
Small forward self-similar solutions are first constructed via perturbation theory \cite{Cannone1993to1994,Cannone1996,Giga1989,barraza1996selfsimilar,KochTataru2001}. Later, Jia and Šverák constructed large self-similar solutions for the $3$D Navier-Stokes equation with $-1$-homogeneous initial data and described their far field asymptotic behaviors via local-in-space regularity estimates near the initial time \cite{JiaSverak_Invention}. Following this groundbreaking result, alternative constructions and extending results of the same equation are available in \cite{Tsai2014CMP,Tsai_halfspace_AnalPDE16,Tsai_discrete_AnnHP17,Tsai_rotationCPDE17,BradshawTsai_Besov2018,AlbrittonARMA2019,CHAE_discrete_2018,BradshawTsai_LocalEnergy_AnalPDE2019}. \\ \indent
In a series of works \cite{LAI2019981,Lai_TAMS21,LAI2024JFA}, the authors constructed forward self-similar solutions to the hypodissipative Navier-Stokes equations in $3$D using the Leray-Schauder fixed point theorem and direct energy estimates, including the critical case. The current work for the $2$D case adopts the same general framework, with some different techniques that yield more refined results. In particular, the pointwise asymptotic estimates hold for \textit{all} solutions, not just for the ones constructed. This information about the self-similar profiles is important, for example if one wants to further investigate the nonuniqueness of solutions.\\ \indent
Recently, the forward self-similar solutions to the $2$D Navier-Stokes equations are constructed in \cite{albritton2026forwardselfsimilarsolutions2d,gui2026forwardselfsimilarsolutionstwodimensional}, with two different approaches. Compared with their results, the nonlocal nature of the fractional Laplacian in our case requires additional techniques in the estimates. Besides, weaker diffusion poses technical issues, for example in the control of the drift and nonlinear terms. Indeed, when the diffusion is not strong enough to control the nonlinear term (in $2$D the threshold is $\alpha=\frac{2}{3}$ while in $3$D it is $\alpha=\frac{5}{6}$ for $(-\Delta)^\alpha$), the regularity of the solution is not fully understood. On the other hand, in the fractional case we have coercivity in the full $H^\alpha$-level, while in the $2$D Navier-Stokes case the coercivity only holds in the sense of $\dot H^1$. To recover control of the $L^2$ quantity, one can apply an additional $L^p$-estimate for $p\in(1,2)$, as introduced in \cite{gui2026forwardselfsimilarsolutionstwodimensional}. \\
\indent Finally, we remark that there are recent constructions of self-similar algebraic spiral solutions for the 2-D Euler equations \cite{ShaoWeiZhang2025,choi2025asymmetricselfsimilarspiralsolutions}, originated from the seminal works of Elling \cite{ELLING2013JDE,EllingCMP2016,EllingBBMS2016}.

\subsubsection{Connection with nonuniqueness}
Recently, forward self-similar solutions have played an important role in the investigation of the nonuniqueness phenomenon in fluid mechanics. For the $3$D Navier-Stokes equation, one of the major problems is the nonuniqueness of the Leray-Hopf solutions. In \cite{JiaSverak_Invention}, Jia and Šverák conjectured that nonuniqueness of solutions might occur for large initial data, and they described a bifurcation scenario in \cite{JiaSverak2015JFA} where multiple self-similar solutions emerge from the same initial data. The mechanism responsible for this behavior is the crossing of isolated eigenvalues of the linearized operator through the imaginary axis, resulting in instability. Later, Guillod and Šverák investigated this bifurcation numerically in \cite{Guillod2023}, showing that there is indeed crossing of eigenvalues when increasing the magnitude of the initial data. However, it is hard to elevate this idea to a rigorous computer-assisted proof, since the verification of such crossing behavior requires keeping track of the eigenvalue dynamically. \\
\indent Later in \cite{abc_forcing}, Albritton, Bru\'e  and Colombo simplified this mechanism, showing that the mere existence of an unstable mode is sufficient to construct multiple solutions with the same initial data. Based on this idea, they proved nonuniqueness of the Leray-Hopf solutions of the $3$D Navier-Stokes equation with forcing, by lifting Vishik's unstable vortex for the $2$D Euler equation in \cite{vishik2018I,vishik2018II} to the $3$D axisymmetric setting. See also \cite{Instability_after_Vishik} for an illustration of Vishik's work. The same problem \textit{without} forcing is very different though, since there is no freedom to choose the self-similar profile anymore, and it was unclear whether there exists a self-similar solution that admits an unstable mode, hence leading to nonuniqueness. Most recently, in \cite{hou2025nonuniquenesslerayhopfsolutionsunforced} T. Hou, Y. Wang, and C. Yang answered this question affirmatively by using computer-assistance, and rigorously proved the nonuniqueness of the Leray-Hopf solutions to the $3$D Navier-Stokes without forcing. \\
\indent
In $2$D, one of the major problems is the uniqueness of solutions to the Euler equation beyond the Yudovich class. Vishik's work established the nonuniqueness in the forced case. However, to the best of our knowledge, the unforced case remains open. One fundamental difference from the Navier-Stokes case is that the self-similar solutions to the Euler equation, if exist, must be singular, due to the absence of viscosity. This poses serious issues both theoretically and numerically. The introduction of forcing is one way to circumvent it, since the presence of forcing makes smooth self-similar profiles feasible. Another possible approach to investigate this problem is to add certain viscous regularization term to the equation and then try to recover the original $2$D Euler scenario through some limiting process. One interesting result is \cite{albritton2025vanishingviscositynonuniquesolutions}, where the nonuniqueness of vanishing viscosity solutions of the forced $2$D Euler equation is established. We note that the Euler equation has one more degree of scaling freedom than the Navier-Stokes equation. Thus, one might consider regularization that reflects this degree of freedom, and one natural choice is the family of fractional Laplacians. This is the motivation for our work. Numerical evidence in \cite{albritton2026forwardselfsimilarsolutions2d} already indicates the possibility of nonuniqueness in $2$D viscous systems. It is worth noting that the nonuniqueness of the Leray-Hopf solution to the forced $2$D hypodissipative Navier-Stokes equations has been proved in \cite{hypoNS_force23}, using again the Vishik's unstable vortex. Now based on the information about the self-similar profiles obtained in the current work together with the preliminary numerical results, one might be able to prove the nonuniqueness of the unforced problem via computer assistance. We plan to address this in future work. Furthermore, in the framework of convex integration, nonuniqueness of the Leray-Hopf solutions for the fractional Navier-Stokes equations on the $3$D torus $(\reall/\mathbb{Z})^3$ was established in \cite{ConvexIntegration2018,ConvexIntegration2019}, provided that the exponent of the fractional Laplacian is small enough (specifically, $\alpha<\frac{1}{3}$ for $(-\Delta)^\alpha$).\\
\indent Finally, we remark that (backward) self-similar solutions are important mathematical objects in the study of finite time singularities, for example of the $3$D Euler equations \cite{Elgindi_Annals,Elgindi2019stability,ChenHou_numerics,ChenHou_analysis}. For the case of $3$D Navier-Stokes, one of the celebrated millennium prize problems, finite energy self-similar blowup solution has been ruled out in \cite{Necas_Acta1996,Tsai_LeraySol1998}. Nevertheless, some numerical results indicate singularity formations in a nearly self-similar way \cite{Hou_3DNS23}.  

\subsection{Strategy of the proof}
We decompose the profile $U:= U_0 + V$. Then, the remainder $V$ solves
\begin{align}\label{self similar eq for V}
    \Lambda^{2\alpha}V - \frac{2\alpha-1}{2\alpha}V - \frac{1}{2\alpha}y\cdot\nabla V + \nabla P= -U_0\cdot\nabla U_0 -V\cdot\nabla U_0 -U_0\cdot\nabla V -V\cdot\nabla V,\quad \div V=0.
\end{align}
The fractional heat profile $U_0$ can be expressed explicitly via the convolution of the heat kernel, so its pointwise information is clear. To show the existence of $V\in H^\alpha$, we perform a priori energy estimates and apply Leray-Schauder fixed point theorem. At a formal level, testing \eqref{self similar eq for V} with $V$, through integration by parts, the diffusion term and the scaling term yields
\[
   \intt\dr{ \Lambda^{2\alpha}V - \frac{2\alpha-1}{2\alpha}V - \frac{1}{2\alpha}y\cdot\nabla V}\cdot V \;dy= \|\Lambda^\alpha V\|^2_{\Ltw}+\frac{1-\alpha}{\alpha}\|V\|^2_{\Ltw}. 
\]
Since $\alpha<1$, $\frac{1-\alpha}{\alpha}>0$, and we have coercivity in the $H^\alpha$-level. For the source term $U_0\cdot \nabla U_0$, basic regularity of $u_0$ is required to make $U_0\cdot\nabla U_0$ decay fast enough so that it belongs to $L^2(\reall^2)$. For the term $V\cdot \nabla U_0$ on the right-hand side, since $\|\nabla U_0\|_{L^\infty}$ is not small, it cannot be absorbed directly into the coercive $L^2$-part. Inspired by \cite{gui2026forwardselfsimilarsolutionstwodimensional}, we therefore introduce a second decomposition of the velocity field via a Bogovskiĭ correction. More precisely, we first truncate $U_0$ by a large-scale cutoff, and then restore the divergence-free condition by removing the resulting divergence defect with the Bogovskiĭ operator. In this way, we obtain a new background velocity field whose gradient is arbitrarily small, so that the corresponding transport term becomes perturbative and can be absorbed into the coercive part of the energy estimate. Then, one can go back to the energy estimate of $V$ up to some error depending only on $U_0$. Roughly speaking, the aforementioned observations give us the uniform boundedness of the $H^\alpha$-norm of $V$, a key ingredient in the Leray-Schauder fixed point theorem. To make the argument rigorous, as in \cite{LAI2019981} we add a vanishing viscous term $\nu\Delta$ to get information of $\nabla V$, so that the drift term $y\cdot\nabla V$ makes sense. In addition, we approximate the whole space with an increasing sequence of concentric balls to maintain compactness. Finally, taking $\nu\to 0^+$ and the radius of the domain to infinity yields a desired global self-similar solution to \eqref{self similar eq for V}. \\
\indent After we show the existence of $H^\alpha$ weak solutions to \eqref{self similar eq for V}, we perform a posteriori energy estimates to improve their regularities. 
The nonlinear term is the main obstacle for the improvement of regularity, and the threshold $\alpha=\frac{2}{3}$ is seen clearly from the Sobolev embedding of multiplication $H^\alpha(\reall^2)\times H^\alpha(\reall^2)\hookrightarrow H^{2\alpha-1}(\reall^2)$ for $\alpha\in (\frac{1}{2},1)$. Indeed, when $\alpha>\frac{2}{3}$, the nonlinear term is controlled by the diffusion and we can improve the regularity of $V$ all the way to the class $\cap_{k\in\mathbb{N}}H^k(\reall^2)$.
To overcome the regularity issue at the beginning (for example, to make the drift term well-defined), we mollify $V$ to get $\Ve$. Necessarily, one needs to deal with the commutators emerging from the mollification of \eqref{self similar eq for V}. Through quantitative estimates of the convergence rate of mollification, we show that the commutators are actually of order $\eps^\kappa$ to some positive power $\kappa$ when measuring with suitable norms, hence negligible when we take $\eps$ small. The mollification argument does not rely on the way that $V$ is constructed, thus the results apply to \textit{all} weak solutions that are $H^\alpha$.
Next, to prepare for the decay estimates, we derive preliminary pointwise information of $V$ via weighted energy estimates, following the idea in \cite{Lai_TAMS21}. One technical difficulty is that, to close the weighted energy estimate on the $H^\alpha$-level, one needs to design the test function carefully so that it justifies the pairing with the drift term while preserving the $H^\alpha$-coercivity. The two things sound contradictory, because it is exactly the growing behavior of $y$ that produces a good sign and leads to coercivity. An informative toy computation is 
\[
    -\intt (y|y|^{-s}\cdot \nabla V)\cdot V = \frac{1}{2}\intt \div(y|y|^{-s})|V|^2\;dy = \frac{2-s}{2}\intt |V|^2\;dy,
\]
so that coercivity is lost when $s>0$ is not small. A direct cutoff, or a multiplication of the test function by $\frac{1}{(1+\delta^2|y|^2)^q}$, will not work, since it will destroy the coercivity at the region of cutoff or in the far field. One solution to this issue is a modification function consisting of three parts: a growing part, a long slow decaying part, and a fast decay tail. In the slow decaying region, the overall drift is ``almost" $y$ and the coercivity is preserved. In the tail, the modification is small (since it has been decaying in a long distance), hence negligible. With preliminary pointwise estimates in hand, one can use the Duhamel expression of $V$, treating all the nonlinear interactions as the source term, to bootstrap the decay. Ideally, the decay rate of $V$ should be the same as the source term $U_0\cdot\nabla U_0$. However, due to criticality of integration, one gets logarithm loss from direct estimates. This loss is removed by a multiplier trick $\int fg = \int (-\Delta)^s f (-\Delta)^{-s}g$ introduced in \cite{LAI2024JFA,BradshawPhelps2023_RemoveLogLoss}, which breaks the criticality.    
\\ [3pt]
\textit{Organization of the paper.} In Section \ref{Sec 2}, we prove the important pointwise estimates for the fractional heat profile $U_0$, as well as other preliminary results which are useful later. In Section \ref{Sec 3}, we prove the existence of an $H^\alpha$ weak solution to \eqref{self similar eq for V}. In Section \ref{Sec 4}, we first prove that any $H^\alpha$ weak solution can be improved to $H^{k}$ for any $k\geq 1$. Next, we prove the weighted $H^\alpha$ estimate, and then the weighted $H^{1+\alpha}$ estimate. In Section \ref{Sec 5}, we prove the pointwise estimates of $V$. A few technical lemmas are deferred to Appendix \ref{Appendix A}. Finally, numerical investigation for nonuniqueness is in Section \ref{Sec 6} and Appendix \ref{Appendix B}.
\\[3pt]
\textit{Notations.} The subscript $\sigma$ stands for ``divergence-free", for example $H^s_{\sigma}:=\{f\in H^s:\div f = 0\}$. We denote constants generically by $C$, the specific values of which can change from line to line. Moreover, the dependence of a constant is indicated in the bracket. For example, $C(Q)$ denotes a constant that depends (only) on $Q$. Denote $\bkt{y}:=(1+|y|^2)^{\frac{1}{2}}$. $A\lesssim B$ means that there exists a universal constant $C>0$ such that $A\leq CB$, where $A$ and $B$ are nonnegative quantities.

\section{Preliminaries}\label{Sec 2}
On $\reall^n$, the fractional Laplacian can be defined via the Fourier transform:
\[
    (-\Delta)^\alpha f:= \Fc^{-1}\dr{|\xi|^{2\alpha}\hat f(\xi)}.
\]
Alternatively, it can also be defined in physical variables for suitable function $f$ as
\begin{align}
    (-\Delta)^\alpha f(x) &= c_{n,\alpha}\,\PV\int_{\reall^n}\frac{f(x)-f(y)}{|x-y|^{n+2\alpha}}\;dy \\
    &=c_{n,\alpha}\lim_{\eps\to 0^+}\int_{\{y\in\reall^2:\;|y|>\eps\}}\frac{f(x)-f(y)}{|x-y|^{n+2\alpha}}\;dy,\quad\alpha\in(0,1),
\end{align}
where $c_{n,\alpha}$ is some constant depending only on $\alpha$ and the dimension. First of all, we derive some important pointwise estimates of the fractional heat profile.
\begin{lemma}[Estimates of $U_0$]\label{lemma: estimates on U_0}
    Let $u_0(x) = |x|^{1-2\alpha}\bar u_0\dr{\frac{x}{|x|}}$ be a divergence free velocity field on $\reall^2\backslash\{0\}$ and denote
    \[
        \|\bar u_0\|_{C^{m,\beta}(\Sp^1)} = M<+\infty,\quad m\in\Z_{\geq 0}, \quad 0<\beta\leq 1.
    \]
    Then, $U_0:= e^{-\Lambda^{2\alpha}}u_0$ is smooth, divergence free, and satisfies
    \begin{align}\label{lemma: estimates of U_0_integer orders}
        |\nabla^k U_0(x)|\leq \frac{C(M,\alpha,\beta,k)}{(1+|x|)^{2\alpha-1+\bar k}},\quad \bar k:= \min(k,\beta+m).
    \end{align}
    Moreover, for any $\gamma\in(0,\beta)$, we have
    \begin{align}\label{lemma: estimates of U_0_fraction orders}
        |\nabla^{k}(-\Delta)^{\frac{\gamma}{2}}U_0(x)|\leq \frac{C(\gamma,M,\alpha,\beta,k)}{(1+|x|)^{2\alpha-1+\gamma+\bar k}},\quad \bar k:= \min(k,m+\beta-\gamma).
    \end{align}
\end{lemma}
\begin{proof}
   Let $H_{\alpha}(x,t)$ be the fundamental solution of the fractional heat operator in $2$D, i.e., for $t>0$,
   \[
   (\partial_t + (-\Delta)^\alpha)H_\alpha(x,t)=0,\quad H_\alpha(x,t)\to \delta_0(x),\quad \text{as} \quad t\to 0^+.
   \]
   Note that $H_\alpha(x,t)$ can be expressed explicitly as 
   $$H_\alpha(x,t) = \Fc^{-1}_\xi (e^{-t|\xi|^{2\alpha}})(x)\quad\Longleftrightarrow \quad H_\alpha(x,t) = \frac{1}{4\pi^2}\intt e^{ix\cdot\xi}e^{-t|\xi|^{2\alpha}}\;d\xi.$$
   Some useful properties of $H_\alpha$ are:
   \begin{equation}
       \intt H_{\alpha}(x,t)\;dx = 1,\quad H_\alpha(x,t) = t^{-\frac{1}{\alpha}}H(t^{-\frac{1}{2\alpha}}x,1),
   \end{equation}
   and the decay estimates
   \begin{align}\label{lemma: estimates of U_0_estimates of H_alpha}
       |\nabla^{k}H_{\alpha}(x,1)|\leq C(\alpha,k)\bkt{x}^{-2-2\alpha-k},\quad\forall\; k\in \Z_{\geq 0}.
   \end{align}
   Now, we can write
   \[
       U_0(x) = \intt H_\alpha(x-y,1)u_0(y)\;dy,
   \]
   which is clearly smooth and divergence free on $\reall^2$. 
    For any $R\geq 2$, we denote the annulus $A_R:=\{x\in\reall^2:\frac{R}{2}\leq |x|\leq R\}$. Then, by scaling one sees that
   \begin{equation}\label{lemma: U_0 estimates_u_0 esti integer case}
       \|\nabla^j u_0\|_{L^\infty(A_R)}\leq C(\alpha,m)M R^{1-2\alpha-j},\quad 1\leq j\leq m, 
   \end{equation}
   and
   \begin{equation}\label{lemma: U_0 estimates_u_0 esti fraction case}
       \|\nabla^m u_0\|_{\dot C^{0,\beta}(A_R)}\leq C(\alpha,m,\beta)M R^{1-2\alpha-m-\beta}.
   \end{equation}
   It is also helpful to note that for any $k\geq 1$ and $|\gamma|=k$,
   \[
      \intt\pa^\gamma H_\alpha(z,1)P(z)\;dz=0,
   \]
   for any polynomial $P(z)$ with degree equal to or less than $k-1$. From now on, we denote constants $C(M,\alpha,\beta,k,\gamma)$ by $C$ for brevity. \\
   \textit{Proof of \eqref{lemma: estimates of U_0_integer orders}}: In the following, we assume $|x|>2$ and $R:= |x|$. First, we split the integration as
   \begin{equation}
       \nabla^{k} U_0(x) = \int_{|z|\leq R/2}\nabla^k H_\alpha(z,1)u_0(x-z)\;dz+\int_{|z|> R/2}\nabla^k H_\alpha(z,1)u_0(x-z)\;dz:= I_{near}+I_{far}.
   \end{equation}
   The estimate of the far field part is straightforward. By \eqref{lemma: estimates of U_0_estimates of H_alpha},  we have
   \begin{align}
       |I_{far}| &\leq  \int_{|z|>R/2,\;|x-z|<R/4}\da{\nabla^k H_\alpha(z,1)u_0(x-z)}\;dz+\int_{|z|>R/2,\;|x-z|\geq R/4}\da{\nabla^k H_\alpha(z,1)u_0(x-z)}\;dz\\
       &\leq C R^{-2-2\alpha-k}\int_{|x-z|<R/4}|x-z|^{1-2\alpha}\;dz+CR^{1-2\alpha}\int_{|z|>R/2}\bkt{z}^{-2-2\alpha-k}\;dz\\
       &\leq C R^{1-4\alpha-k}.
   \end{align}
   For the near field part, it divides into three cases for different $k$'s. When $k=0$, we note that
   \begin{align}
       |I_{near}| \leq C R^{1-2\alpha}\int_{|z|\leq R/2}\bkt{z}^{-2-2\alpha}\;dz\leq CR^{1-2\alpha},\quad k=0.
   \end{align}
   When $1\leq k\leq m$, we consider the Taylor expansion of $u_0(x-z)$ at $z=0$:
   \begin{equation}
       T_{k-1}(z):= \sum_{|l|\leq k-1}\frac{(-z)^l}{l!}(\pa^l u_0)(x).
   \end{equation}
   Then, since $\intt \nabla^k H_\alpha(z)T_{k-1}(z)\;dz=0$, we have
   \begin{align}
       I_{near} = \int_{|z|<R/2}\nabla^k H_\alpha(z,1)(u_0(x-z)-T_{k-1}(z))\;dz-\int_{|z|\geq R/2}\nabla^kH_\alpha(z,1)T_{k-1}(z)\;dz.
   \end{align}
   By Taylor's theorem and \eqref{lemma: U_0 estimates_u_0 esti integer case}, for $|z|<R/2$ we have
   \[
   |u_0(x-z) - T_{k-1}(z)|\leq C |z|^k\|\nabla^ku_0\|_{L^\infty(A_R)}\leq C|z|^kR^{1-2\alpha-k}.
   \]
   It then follows that
   \begin{align}
       \da{\int_{|z|<R/2}\nabla^k H_\alpha(z,1)(u_0(x-z)-T_{k-1}(z))\;dz}\leq CR^{1-2\alpha-k}\int_{|z|<R/2}\bkt{z}^{-2-2\alpha}\;dz\leq CR^{1-2\alpha-k}.
   \end{align}
   For the other term, through integration by parts for $k$ times, what remains are boundary integrals on $|z|=R/2$. Then, due to the estimates \eqref{lemma: U_0 estimates_u_0 esti integer case}, we get
   \begin{equation}
       \da{\int_{|z|\geq R/2}\nabla^kH_\alpha(z,1)T_{k-1}(z)\;dz} \leq C\int_{z = R/2}R^{-1-4\alpha-k}\;d\sigma \leq C R^{-4\alpha-k}.
   \end{equation}
   Therefore, we obtain the estimate for the near field part,
   \begin{equation}
       |I_{near}| \leq CR^{1-2\alpha-k},\quad 1\leq k\leq m. 
   \end{equation}
   When $k\geq m+1$, we consider the Taylor expansion
      \begin{equation}
       T_{m}(z):= \sum_{|l|\leq m}\frac{(-z)^l}{l!}(\pa^l u_0)(x).
   \end{equation}
   By Taylor's theorem and \eqref{lemma: U_0 estimates_u_0 esti fraction case}, for $|z|<R/2$ we have
   \[
   |u_0(x-z) - T_{m}(z)|\leq C |z|^{m+\beta}\|\nabla^mu_0\|_{\dot C^{0,\beta}(A_R)}\leq C|z|^{m+\beta} R^{1-2\alpha-m-\beta}.
   \]
   Then, following the same argument as in the $1\leq k\leq m$ case, we obtain
   \begin{equation}
       |I_{near}|\leq R^{1-2\alpha-m-\beta},\quad k\geq m+1.
   \end{equation}
   Finally, since $1-4\alpha-k < 1-2\alpha-\min(k,m+\beta)$, collecting all these estimates yields \eqref{lemma: estimates of U_0_integer orders}.\\
   \textit{Proof of \eqref{lemma: estimates of U_0_fraction orders}}: First of all, we note that (e.g., via Fourier transform)
   \begin{equation}
       (-\Delta)^{\frac{\gamma}{2}}U_0(x) = 
       \intt H_\alpha(x-y,1)(-\Delta)^{\frac{\gamma}{2}}u_0(y)\;dy.
   \end{equation}
   By scaling analysis, we have $(-\Delta)^{\frac{\gamma}{2}}u_0(x) = |x|^{1-2\alpha-\gamma}\tilde u_{0,\gamma}(\frac{x}{|x|})$ for some function $\tilde u_{0,\gamma}$ on $\Sp^1$. We claim that
   \begin{equation}\label{lemma: estimates of U_0_technical result}
       \|\tilde u_{0,\gamma}\|_{C^{m,\beta-\gamma}(\Sp^1)}\leq C(m,\alpha,\beta,\gamma)\|\bar u_0\|_{C^{m,\beta}(\Sp^1)},\quad 0<\gamma<\beta.
   \end{equation}
   We remark that since we have $\gamma<3-2\alpha$, it holds that $(-\Delta)^{\frac{\gamma}{2}}u_0\in L^1_{loc}(\reall^2)$ so that $(-\Delta)^{\frac{\gamma}{2}}U_0$ is a well-defined smooth function. Assuming \eqref{lemma: estimates of U_0_technical result} holds, then the proof of \eqref{lemma: estimates of U_0_fraction orders} is exactly the same as that of \eqref{lemma: estimates of U_0_integer orders}. We defer the proof of \eqref{lemma: estimates of U_0_technical result} to the appendix (Lemma \ref{appendix: a technical holder estimate on S1}).
\end{proof}
The commutator between the operator $\Lambda^s$ and a function $\vp$ is defined in the usual way as 
\[
    [\Lambda^s,\vp]f:= \Lambda^s(\vp f) - \vp\Lambda^s f.
\]
Here is a commutator estimate that is useful in the energy estimates below.
\begin{lemma}[Commutator estimates]\label{lemma: commutator estimates}
    Let $\alpha\in (0,1)$, $f\in L^2(\reall^2)$ and $\vp\in \dot C^{\beta}(\reall^2)\cap\dot W^{1,\infty}(\reall^2)$ where $\beta\in [0,\alpha)$. Then, there exists some constant $C$ such that
    \begin{equation}\label{lemma: commutator estimate_ineq}
        \|[\Lambda^\alpha,\vp]f\|_{L^2(\reall^2)}\leq C(\|\vp\|_{\dot C^{\beta}(\reall^2)}+\|\vp\|_{\dot W^{1,\infty}(\reall^2)})\|f\|_{\Ltw}.
    \end{equation}
\end{lemma}
\begin{proof}
    By the equivalent definition of the fractional Laplacian operator, we have (note that the principal value is not needed here as $\alpha<1$)
    \begin{equation}
        ([\Lambda^\alpha,\vp]f)(x) = c_\alpha\intt\frac{\vp(x)-\vp(y)}{|x-y|^{2+\alpha}}f(y)\;dy.
    \end{equation}
    Then, by the assumption of $\vp$, we have
    \begin{align}
        \da{([\Lambda^\alpha,\vp]f)(x)} \leq c_\alpha\|\vp\|_{\dot W^{1,\infty}(\reall^2)}\int_{B_1(x)}\frac{|f(y)|}{|x-y|^{1+\alpha}}\;dy+c_\alpha\|\vp\|_{\dot C^{\beta}(\reall^2)}\int_{\reall^2\backslash B_1(x)}\frac{|f(y)|}{|x-y|^{2+\alpha-\beta}}\;dy.
    \end{align}
    By Young's convolution inequality and the fact that $\frac{1}{2} < \alpha < 1$ and $0 < \beta < \alpha$, we obtain the desired estimate \eqref{lemma: commutator estimate_ineq}.
\end{proof}

Next, we introduce the standard mollifier. Let $\rho\in C^\infty_0(\reall^n)$ be a nonnegative function such that $\int \rho \;dy= 1$. For any $\eps>0$, we define 
\[
    \rho_\eps(y):=\frac{1}{\eps^n}\rho\dr{\frac{y}{\eps}}.
\]
For any other function $f$, unless otherwise indicated from the context, we denote its mollification as
\[
    f_\eps := \rho_\epsilon * f.
\]
Here are some useful convergence properties of the mollification:
\begin{lemma}[Convergence of mollification]\label{lemma: mollification}
    Let $f\in H^s(\reall^n)$ for some $s\in\reall$, then we have
    \begin{equation}
        \lim_{\eps\to 0^+}\|f_\eps - f\|_{H^s(\reall^n)}=0.
    \end{equation}
    Moreover, if $\|f\|_{H^{s+\delta}(\reall^2)}$ is finite for some $\delta\in (0,1]$, then we can quantify the convergence rate as
    \begin{equation}
        \|f_\eps - f\|_{H^s(\reall^n)} \lesssim \eps^\delta\|f\|_{H^{s+\delta}(\reall^n)}.
    \end{equation}
\end{lemma}
\begin{proof}
    First, we note that $\widehat{\rho_\eps}(\xi) = \hat\rho(\eps\xi)$, and by construction
    \[
        \hat\rho(0) =\int\rho = 1,\quad \|\hat\rho\|_{L^\infty}\leq 1.
    \]
    Thus, we get
    \begin{align}
        \|f_\eps - f\|^2_{H^s} = \int |1-\hat\rho(\eps\xi)|^2|\hat f(\xi)|^2\bkt{\xi}^{2s}\;d\xi \to 0,\quad\text{as}\quad\eps\to0^+,
    \end{align}
    by the Lebesgue dominated convergence theorem.
    In addition, if $f\in H^{s+\delta}(\reall^n)$, we have
    \begin{align}
        \|f_\eps - f\|^2_{H^s} = \int |1-\hat\rho(\eps\xi)|^2|\hat f(\xi)|^2\bkt{\xi}^{2s}\;d\xi= \int \eps^{2\delta}|\xi|^{2\delta}\frac{|1-\hat\rho(\eps\xi)|^2}{|\eps\xi|^{2\delta}}|f(\xi)|^2\bkt{\xi}^{2s}\;d\xi\lesssim \eps^{2\delta}\|f\|^2_{H^{s+\delta}(\reall^n)},
    \end{align}
    where we have used 
    \[
        \sup_{\xi\neq 0}\frac{|1-\hat\rho(\xi)|}{|\xi|^{\delta}} <+\infty,\quad \delta\in [0,1].
    \]
    We remark that the range of $\delta$ can be enlarged once we choose $\rho$ such that $\hat\rho$ has a higher vanishing order at the origin (e.g., choose $\rho$ to be radially symmetric). However, we do not need this stronger property.
\end{proof}
Since we will conduct weighted energy estimates, it is helpful to consider the weighted Lebesgue spaces on which all Calderón–Zygmund operators, in particular the Riesz transforms, are bounded. Specifically, let $\om\in L^1_{loc}(\reall^n)$ be positive almost everywhere. We define the weighted $L^2$ space as
\[ 
    \|f\|^2_{L^2(\omega)}:= \int_{\reall^n}|f(y)|^2\omega(y)\;dy.
\]
In addition, we say that the weight function $\omega$ belongs to the $\Ac_2$ class, if it satisfies
\begin{equation}
    \|\omega\|_{\Ac_2} := \sup_{Q}\frac{1}{|Q|^2}\dr{\int_{Q}\om(y)\;dy}\dr{\int_Q \om^{-1}(y)\;dy}<+\infty,
\end{equation}
where the supremum above is taken over all cubes $Q\subset\reall^n$ and $|Q|$ is the volume of the cube under the Lebesgue measure. Then, we have the following result:
\begin{theorem}[{\cite[Theorem 1.3]{Hytonen2012SharpWeightedBound}}]\label{theorem: A2 conjecture}
    Let $T\in \Ls(L^2(\reall^n))$ be any Calderón–Zygmund operator for $n\in \Z_{\geq 1}$. Then, it holds that
    \begin{align}
        \|Tf\|_{L^2(\om)}\leq C(T)\|\om\|_{\Ac_2}\|f\|_{L^2(\om)}.
    \end{align}
\end{theorem}
This result will be used later to control the pressure term in weighted energy estimates.

\section{Existence of solutions in $H^\alpha(\reall^2)$}\label{Sec 3}
In this section, we establish the existence of $H^\alpha(\reall^2)$ solutions to \eqref{self similar eq for V} via the Leray-Schauder fixed point theorem. From now on, we confine our discussion to the scenario where $\bar u_0\in C^{0,1}(\Sp^1)$. The main result in this section is:
\begin{theorem}\label{theorem: existence of solution V in H alpha}
    There exists $V\in H^\alpha_\sigma(\reall^2)$ that solves \eqref{self similar eq for V} in the sense of distribution.
\end{theorem}
In order to get information of the first derivative and facilitate our argument, we add a perturbative viscosity term to \eqref{self similar eq for V} and restrict ourselves to bounded domains $B_R(0)$ first. Then, the solutions of the original equation on the whole domain are obtained once we set the viscosity to zero and $R\to\infty$. Namely, for any $\nu>0$, we consider 
\begin{align}\label{self similar eq for V plus viscosity}
\begin{split}
        \nu(-\Delta)V+\Lambda^{2\alpha}V - \frac{2\alpha-1}{2\alpha}V - \frac{1}{2\alpha}y\cdot\nabla V + \nabla P&= -U_0\cdot\nabla U_0 -V\cdot\nabla U_0 -U_0\cdot\nabla V -V\cdot\nabla V\\
        &:=F(V),  \qquad \div V=0.
\end{split}
\end{align}
For $R>0$, we define the function space
\begin{equation}
    \Hc_R := \{u:\; u\in H^1_{0,\sigma}(B_R)\cap H^{\alpha}_{0,\sigma}(B_R),\; u(x)\equiv0 \text{ for }x\in \reall^2\backslash B_R\}
\end{equation}
together with the inner product 
\[
   \scl{u}{v}_{\Hc_R}:=\nu\int_{\reall^2} \nabla u\cdot\nabla v\;dy+\int_{\reall^2} \Lambda^\alpha u\cdot\Lambda^\alpha v\;dy,\quad u,v\in \Hc_R.
\]
We remark that functions in $\Hc_R$ belong to $H^1(\reall^2)\cap H^\alpha(\reall^2)$, since the zero extension operator from $H^s_0(B_R)$ to $H^s(\reall^2)$ is bounded for $s\geq 0$. Moreover, for any $u\in \Hc_R$, by the (fractional) Poincar\'e inequality, we have
\begin{equation}\label{H_0 alpha poincare ineq}
    \|u\|_{\Ltw}\leq C(R)\|\Lambda^\alpha u\|_{\Ltw}.
\end{equation}
Now, we look for weak solutions $V_{R,\nu}$ to \eqref{self similar eq for V plus viscosity} in the space $\Hc_R$, i.e., an element in $\Hc_R$ that satisfies
\begin{equation}\label{weak formulation for V_R,nu}
    \scl{V_{R,\nu}}{\vp}_{\Hc_R} = \int_{\reall^2} \dr{\frac{2\alpha-1}{2\alpha}V_{R,\nu}+\frac{1}{2\alpha}y\cdot\nabla V_{R,\nu}+F(V_{R,\nu})}\cdot\vp\;dy,\quad\forall\,\vp\in \Hc_R ,
\end{equation}
which is well defined for functions in $\Hc_R$ due to Sobolev embeddings.
\begin{proposition}[Existence of solutions on $B_R$]
The system \eqref{self similar eq for V plus viscosity} admits a solution $V_{R,\nu}\in \Hc_R$ (in the sense of \eqref{weak formulation for V_R,nu}).
\end{proposition}
\begin{proof}
    This is an application of the Leray-Schauder fixed point theorem. First of all, for any $v,\vp\in\Hc_R$, by the Sobolev embedding $H^\alpha(\reall^2)\hookrightarrow L^4(\reall^2)$ for $\alpha>\frac{1}{2}$, it holds that
    \begin{align}
         \int_{\reall^2}(v\cdot\nabla v)\cdot\vp\;dx=-\int_{\reall^2}(v\otimes v)\cdot\nabla\vp \lesssim \|v\|^2_{L^4}\|\nabla\vp\|_{L^2} \lesssim \|v\|^2_{\Hc_R}\|\vp\|_{\Hc_R}.
    \end{align}
    This yields by the regularity of $U_0$ that
    \begin{align}
        \int_{\reall^2} \dr{\frac{2\alpha-1}{2\alpha}v+\frac{1}{2\alpha}y\cdot\nabla v+F(v)}\cdot\vp\;dy \leq C(U_0,R)(1+\|v\|_{\Hc_R}+\|v\|^2_{\Hc_R})\|\vp\|_{\Hc_R}.
    \end{align}
    Then, by the Riesz representation theorem, there exists some $\Tc(v)\in \Hc_R$, such that
    \[
       \scl{\Tc(v)}{\vp}_{\Hc_R} = \int_{\reall^2} \dr{\frac{2\alpha-1}{2\alpha}v+\frac{1}{2\alpha}y\cdot\nabla v+F(v)}\cdot\vp\;dy,\quad \forall\,\vp\in \Hc_R.
    \]
    Now \eqref{weak formulation for V_R,nu} translates into a fixed point problem
    \[
       V_{R,\nu} = \Tc(V_{R,\nu}).
    \]
    To apply the Leray-Schauder fixed point theorem, a few verifications are in order.\\
    \textbf{Step 1: Continuity and compactness of $\Tc:\Hc_R\to\Hc_R$:} For any $u,v\in \Hc_R$, based on the definition of $\Tc$ we have
    \begin{align}
        \|\Tc(u) - \Tc(v)\|_{\Hc_R} &= \sup_{\vp\in\Hc_R\backslash\{0\}} \frac{1}{\|\vp\|_{\Hc_R}}\scl{\Tc(u) - \Tc(v)}{\vp}_{\Hc_R}\\
        &=\sup_{\vp\in\Hc_R\backslash\{0\}} \frac{1}{\|\vp\|_{\Hc_R}}\int_{\reall^2}\dr{\frac{2\alpha-1}{2\alpha}(u-v)+\frac{1}{2\alpha}y\cdot\nabla(u-v)+F(u)-F(v)}\cdot\vp\;dy \\
        &\leq C(U_0,R)\dr{1+\|u\|_{\Hc_R}+\|v\|_{\Hc_R}}\|u-v\|_{\Hc_R}.
    \end{align}
    Thus, $\Tc$ is continuous. Now, suppose that $\{v_n\}$ is a bounded sequence in $\Hc_R$. Then, by the compact embedding $H^1(B_R)\to L^p(B_R)$ for any $p\geq 2$, up to a subsequence we have
    \begin{align}
        &v_n \rightharpoonup v \quad \text{ in }\Hc_R,\\
        &v_n \to v \quad\text{ in } L^p(B_R),\quad \forall\, p\geq 2.
    \end{align}
   
  Similar estimates (integration by parts whenever we encounter the term $\nabla(v_n-v)$) show that
    \begin{equation}
        \|\Tc(v_n)-\Tc(v)\|_{\Hc_R}\leq C(U_0,R) (1+\|v_n\|_{\Hc_R}+\|v\|_{\Hc_R})\|v_n-v\|_{L^4}\to 0,\quad n\to+\infty. 
    \end{equation}
    Therefore, we conclude that $\Tc$ is compact.\\
    \textbf{Step 2: A priori estimate:} Assume for $\lambda\in[0,1]$, $V_\lambda\in \Hc_R$ solves
    \begin{equation}\label{Leray-Schauder lambda formation}
        V_\lambda = \lambda \Tc(V_\lambda).
    \end{equation}
    We need to derive uniform estimates on $\|V_\lambda\|_{\Hc_R}$ which is independent of $\lambda$. Now, testing \eqref{Leray-Schauder lambda formation} with $V_\lambda$ and performing integration by parts yields
    \begin{align}\label{energy estimate for V lambda}
        \int_{\reall^2}\nu|\nabla V_\lambda|^2+|\Lambda^\alpha V_\lambda|^2\;dy+\frac{\lambda(1-\alpha)}{\alpha}\int_{\reall^2}|V_\lambda|^2\;dy = -\lambda\int_{\reall^2} (U_0\cdot\nabla U_0)\cdot V_\lambda + (V_\lambda\cdot\nabla U_0)\cdot V_\lambda\;dy,
    \end{align}
    where we have used that
    \[
       \int_{\reall^2} \dr{(V_\lambda+U_0)\cdot\nabla V_\lambda }\cdot V_\lambda \;dy= 0.
    \]
    By Lemma \ref{lemma: estimates on U_0}, we know that  $U_0\cdot\nabla U_0$ belongs to $L^2(\reall^2)$ when $\bar u_0\in C^{0,1}(\Sp^1)$. Then, H\"older's inequality gives us the control of the first term on the right hand side of \eqref{energy estimate for V lambda}. For the second term, it can be controlled if $\|\nabla U_0\|_{L^\infty}$ is sufficiently small. However, we cannot expect this to be valid in general. To address this issue, we make a second decomposition of the velocity field, inspired by the work \cite{gui2026forwardselfsimilarsolutionstwodimensional}. Let $\eta(y)$ be a smooth cutoff function such that $\eta(y)=0$ for $|y|\leq 1$, and $\eta(y)\equiv 1$ for $|y|\geq 2$. Fix some sufficiently large $R_0>0$, we denote $\eta_{R_0}(y):=\eta\dr{\frac{y}{R_0}}$. Since $-\nabla\eta_{R_0}\cdot U_0\in C^\infty_0(\reall^2)$, by the Bogovskiĭ formula (see, for example, \cite[Theorem \Rmnum{3}.3.3]{galdi2011introduction}), there exists some $\tilde U\in C^\infty_0(\reall^2)$ such that
    \begin{equation}
        \div \tilde U = -\nabla\eta_{R_0}\cdot U_0,\quad \supp(\tilde U) \subset B_{2R_0}\backslash B_{R_0}.
    \end{equation}
    Moreover, due to the scaling and the property of the Bogovskiĭ operator, we have 
    \begin{equation}
        \|\nabla \tilde U\|_{L^\infty}\leq C\|\nabla\eta_{R_0}\cdot U_0\|_{L^\infty},
    \end{equation}
    where $C$ is some universal constant independent of $R$. By the far field asymptotics of $U_0$, we know that $\|\nabla\eta_{R_0}\cdot U_0\|_{L^\infty}=\Oc(R_0^{-2\alpha})$ as $R_0\to\infty$. Therefore, for any small constant $\delta>0$, we can always find $R_0 = R_0(\delta)$, such that
    \[
        \|\nabla \tilde U\|_{L^\infty} + \|\nabla(\eta_{R_0}U_0)\|_{L^\infty} <\delta.
    \]
    Now, we define
    \[ 
        U_1 := \eta_{R_0}U_0 + \tilde U.
    \]
    By construction, we know that $\div U_1 = 0$ and $\|\nabla U_1\|_{L^\infty}<\delta$.
    Then, we decompose
    \begin{equation}
        V_\lambda = U_1-U_0+V_\lambda'.
    \end{equation}
    Since $U_1-U_0 \in \Hc_R$ (without loss of generality, we can assume $R>R_0$), we have $V_\lambda'\in\Hc_R$.
    Inserting the decomposition into \eqref{Leray-Schauder lambda formation} and testing it with $V'_\lambda$, we obtain
    \begin{align}\label{energy estimates for V prime lambda}
    \begin{split}
         &\int_{\reall^2}\nu|\nabla V'_\lambda|^2+|\Lambda^\alpha V'_\lambda|^2\;dy+\frac{\lambda(1-\alpha)}{\alpha}\int_{\reall^2}|V'_\lambda|^2\;dy \\
         &= -\lambda\int_{\reall^2} (U_1\cdot\nabla U_1)\cdot V'_\lambda + (V_\lambda'\cdot\nabla U_1)\cdot V'_\lambda\;dy-\intt N(U_0,\tilde U,\lambda)\cdot V'_\lambda \;dy,
    \end{split}
    \end{align}
    where
    \begin{equation}
        N(U_0,\tilde U,\lambda) := -\nu\Delta (U_1-U_0)+\Lambda^{2\alpha}U_1+(\lambda-1)\Lambda^{2\alpha}U_0-\lambda\dr{\frac{2\alpha-1}{2\alpha}U_1+\frac{1}{2\alpha}y\cdot\nabla U_1}. 
    \end{equation}
    Since $U_1$ has the same decay properties as $U_0$, we know that $U_1\cdot\nabla U_1\in L^2(\reall^2)$, and it follows that
    \begin{equation}
        -\lambda\int_{\reall^2} (U_1\cdot\nabla U_1)\cdot V'_\lambda + (V_\lambda'\cdot\nabla U_1)\cdot V'_\lambda\leq 2\delta \lambda\int_{\reall^2}|V'_\lambda|^2\;dy+C(U_0,\delta).
    \end{equation}
    Now it remains to estimate the term involving $N(U_0,\tilde U,\lambda)$. Note by \eqref{equation that U_0 solves} that
    \begin{equation}
        \frac{2\alpha-1}{2\alpha}U_1+\frac{1}{2\alpha}y\cdot\nabla U_1 = \frac{1}{2\alpha}y\cdot\nabla\eta_{R_0} U_0+\eta_{R_0}\Lambda^{2\alpha}U_0+ \frac{2\alpha-1}{2\alpha}\tilde U+\frac{1}{2\alpha}y\cdot\nabla \tilde U.
    \end{equation}
    Therefore, we get
    \begin{align}
        \lambda\int_{\reall^2}\dr{\frac{2\alpha-1}{2\alpha}U_1+\frac{1}{2\alpha}y\cdot\nabla U_1}\cdot V'_\lambda\;dy&\leq \delta\lambda\|V'_\lambda\|^2_{L^2(\reall^2)}+C(\delta)\dr{\|y\cdot\nabla\eta_{R_0} U_0\|^2_{L^2(\reall^2)}+\|\eta_{R_0}\Lambda^{2\alpha}U_0\|^2_{\reall^2}}\\
        &\quad+C(\delta)\dr{\|\tilde U\|^2_{\Ltw}+\|y\cdot \nabla\tilde U\|^2_{\Ltw}}\\
        &\leq \delta\lambda\|V'_\lambda\|^2_{L^2(\reall^2)}+C(U_0,R_0,\delta),
    \end{align}
    where we have used the interpolation estimate (as $1<2\alpha<2$)
    \[
       \|\eta_{R_0}\Lambda^{2\alpha}U_0\|^2_{\Ltw} \lesssim\|\Lambda^{2\alpha}U_0\|^2_{\Ltw}\lesssim\|\nabla U_0\|^2_{\Ltw}+\|\nabla^2 U_0\|^2_{\Ltw}<+\infty,
    \]
    since we have $\nabla^{k}U_0\in \Ltw$ for $k\geq 1$ by the decay estimates of $U_0$.
    Similarly, since $U_1-U_0$ is compactly supported and $\Lambda^{2\alpha}U_1$ belongs to $\Ltw$ by the same interpolation estimate described above, we have
    \begin{equation}
        \int_{\reall^2}\dr{\nu\Delta(U_1-U_0) - \Lambda^{2\alpha}U_1-(\lambda-1)\Lambda^{2\alpha}U_0}\cdot V'_\lambda\;dy \leq \delta\|V'_\lambda\|^2_{L^2(\reall^2)} + C(U_0,R_0,\delta).
    \end{equation}
    Collecting all these estimates together with \eqref{H_0 alpha poincare ineq}, once we choose $\delta$ small enough (depending only on $\alpha$ and $R$), \eqref{energy estimates for V prime lambda} yields
    \begin{equation}
        \int_{\reall^2}\nu|\nabla V'_\lambda|^2+|\Lambda^\alpha V'_\lambda|^2\;dy+\frac{\lambda(1-\alpha)}{\alpha}\int_{\reall^2}|V'_\lambda|^2\;dy \leq C(U_0,R).
    \end{equation}
    We have dropped the dependence on $R_0$ and $\alpha$ since they are constants once fixed.
    Finally, since $U_1-U_0\in C^\infty_0(\reall^2)$, the norms of which depend only on $U_0$, it follows that
    \begin{equation}
        \int_{\reall^2}\nu|\nabla V_\lambda|^2+|\Lambda^\alpha V_\lambda|^2\;dy+\frac{\lambda(1-\alpha)}{\alpha}\int_{\reall^2}|V_\lambda|^2\;dy \leq C(U_0,R),\quad\forall\,\lambda\in[0,1].
    \end{equation}
    \textbf{Conclusion:} Now that $\Tc:\Hc_R\to \Hc_R$ is well-defined, continuous and compact, and the solutions to $V_\lambda = \lambda\Tc(V_\lambda)$ are uniformly bounded in the $\Hc_R$-norm for $\lambda\in [0,1]$, by the Leray-Schauder fixed point theorem, there exists some $V_{R,\nu}\in \Hc_R$ that solves $V_{R,\nu} = \Tc(V_{R,\nu})$. One important observation: once the existence of the fixed point $V_{R,\nu}$ is established, one can go over the energy estimates above again for $V_{R,\nu}$ and remove the dependence of the constant on $R$. This is because when $\lambda=1$, the error terms of size $\delta\|V_{R,\nu}\|^2$ can be absorbed into the coercive part $\frac{1-\alpha}{\alpha}\|V_{R,\nu}\|^2_{\Ltw}$, once we take $\delta$ small (depending only on $\alpha$ this time). Hence, we do not need to turn to the Poincar\'e inequality \eqref{H_0 alpha poincare ineq} for help, and there is no dependence on $R$ in the estimate. To summarize, for the fixed point $V_{R,\nu}\in\Hc_R$, the following estimate holds:
    \begin{align}\label{proposition: existence of V R nu_uniform esti of V_R nu}
         \int_{\reall^2}\nu|\nabla V_{R,\nu}|^2+|\Lambda^\alpha V_{R,\nu}|^2\;dy+\frac{1-\alpha}{\alpha}\int_{\reall^2}|V_{R,\nu}|^2\;dy \leq C(U_0).
    \end{align}
\end{proof}
One crucial point in \eqref{proposition: existence of V R nu_uniform esti of V_R nu} is that the estimate for $V_{R,\nu}$ does not depend on $R$ or $\nu$ (on the right hand side). This helps when we are to take the limits $R\to+\infty$ and $\nu\to 0$. 
\begin{proposition}\label{proposition: existence of V nu}
    There is a solution $V_\nu\in H^\alpha_{\sigma}(\reall^2)\cap H^1_\sigma(\reall^2)$ to the system \eqref{self similar eq for V plus viscosity}.
\end{proposition}
\begin{proof}
    Since we have 
    \[
        \|V_{R,\nu}\|^2_{\Hc_R}+\frac{1-\alpha}{\alpha}\|V_{R,\nu}\|^2_{\Ltw}\leq C(U_0,\alpha),\quad\forall\,R>0,
    \]
    there exists a sequence $R_j\to+\infty$, and some $V_\nu\in H^\alpha_{\sigma}(\reall^2)\cap H^1_\sigma(\reall^2)$, such that
    \begin{align}\label{convergences of V_R_j,nu}
    \begin{split}
        &V_{R_j,\nu}\rightharpoonup V_{\nu},\quad\text{ in }\;H^\alpha_{\sigma}(\reall^2)\cap H^1_\sigma(\reall^2),\\
        &V_{R_j,\nu}\to V_\nu,\quad\text{ in }\;L^4(B_R),\quad\forall\,R\in(0,+\infty).
    \end{split}
    \end{align}
    Now we verify that $V_\nu$ is a solution to \eqref{self similar eq for V plus viscosity}. For any $\vp\in C^\infty_{0,\sigma}(\reall^2)$, by the construction of $V_{R_j,\nu}$ we have
    \begin{align}
        \int_{\reall^2}\nu\nabla V_{R_j,\nu}\cdot\nabla\vp+\Lambda^\alpha V_{R_j,\nu}\cdot\Lambda^\alpha\vp\;dy = \int_{\reall^2} \dr{\frac{2\alpha-1}{2\alpha}V_{R_j,\nu}+\frac{1}{2\alpha}y\cdot\nabla V_{R_j,\nu}+F(V_{R_j,\nu})}\cdot\vp\;dy.
    \end{align}
    It is easy to see from \eqref{convergences of V_R_j,nu} that
    \begin{align}
        &\int_{\reall^2}\nu\nabla V_{R_j,\nu}\cdot\nabla\vp+\Lambda^\alpha V_{R_j,\nu}\cdot\Lambda^\alpha\vp+\dr{-\frac{2\alpha-1}{2\alpha}V_{R_j,\nu}-\frac{1}{2\alpha}y\cdot\nabla V_{R_j,\nu}+U_0\cdot \nabla V_{R_j,\nu}+V_{R_j,\nu}\cdot\nabla U_0}\cdot\vp\;dy\\
        &\quad\to \int_{\reall^2}\nu\nabla  V_\nu\cdot\nabla\vp+\Lambda^\alpha V_\nu\cdot\Lambda^\alpha\vp+\dr{-\frac{2\alpha-1}{2\alpha}V_\nu-\frac{1}{2\alpha}y\cdot\nabla V_\nu+U_0\cdot \nabla V_\nu+V_\nu\cdot\nabla U_0}\cdot\vp\;dy,\quad j\to+\infty
    \end{align}
    Furthermore, we observe that
    \begin{align}
        \bigg|\int_{\reall^2} (V_{R_j,\nu}\otimes V_{R_j,\nu})&\cdot\nabla\vp- (V_{\nu}\otimes V_{\nu})\cdot\nabla\vp\;dy \bigg|= \da{\int_{\reall^2}\dr{V_{R_j,\nu}\otimes(V_{R_j,\nu}-V_{\nu})+(V_{R_j,\nu}-V_\nu)\otimes V_{\nu}}\cdot\nabla\vp\;dy}\\
        &\lesssim (\|V_{R_j,\nu}\|_{L^4(\reall^2)}+\|V_\nu\|_{L^4(\reall^2)})\|\nabla\vp\|_{L^2(\reall^2)}\|V_{R_j,\nu}-V_\nu\|_{L^4(B_{\supp(\vp)})}\to0,\quad j\to +\infty.
    \end{align}
    Combining all these results, we have shown that $V_\nu$ solves \eqref{self similar eq for V plus viscosity} in the sense of distribution. Moreover, by the weak lower semicontinuity of the norm, the limiting function satisfies
    \begin{equation}
        \int_{\reall^2}\nu|\nabla V_\nu|^2+|\Lambda^\alpha V_\nu|^2+\frac{1-\alpha}{\alpha}|V_\nu|^2\;dy\leq C(U_0,\alpha).
    \end{equation}
\end{proof}
Finally, we are ready to prove the main theorem in this section.\\
\begin{proof}[Proof of Theorem \ref{theorem: existence of solution V in H alpha}]
Since we have 
\[
    \|V_\nu\|_{H^\alpha(\reall^2)}\leq C(U_0,\alpha),
\]
which is independent of $\nu$, there exists a sequence $\nu_j\to 0$ and $V\in H^\alpha_\sigma(\reall^2)$ such that
\begin{align}
     &V_{\nu_j}\rightharpoonup V,\quad\text{ in }\;H^\alpha_{\sigma}(\reall^2),\\
    &V_{\nu_j}\to V,\quad\text{ in }\;L^4(B_R),\quad\forall\,R\in(0,+\infty).
\end{align}
Then, using the same argument as in the previous proof and the observation that $\nu_j (-\Delta) V_{\nu_j}\to 0$ in the sense of distribution as $j\to +\infty$, we can show that $V$ is a solution to \eqref{self similar eq for V} in the sense of distribution. 
\end{proof}
Once we construct a solution $U = U_0+V$ with $V\in H^\alpha(\reall^2)$ to \eqref{self similar eq for U}, the pressure can be recovered (up to a constant) by solving a Poisson equation.
\begin{lemma}[Control of pressure]\label{lemma: control of pressure}
    Let $V\in H^\alpha(\reall^2)$ be a velocity field for $\alpha\in(\frac{1}{2},1)$. Then, there exists pressure $P$ that solves 
    \[
    -\Delta P = \div\div((U_0+V)\otimes(U_0+V)). 
\]
If we further decompose 
\begin{align}
    P &= (-\Delta)^{-1}\div\div((U_0\otimes V)+(V\otimes U_0)+(V\otimes V))+(-\Delta)^{-1}\div\div(U_0\otimes U_0)\\
    &:=P_1+P_2 ,
\end{align}
then we have the control
\begin{equation}
    \|\nabla^{k-1}P_1\|_{H^{2\alpha-1}(\reall^2)}\leq C(U_0)(\|V\|^2_{H^{k-1+\alpha}(\reall^2)}+1),\quad \|\nabla^{k}P_2\|_{\Ltw}\leq C(U_0,k),\quad \forall\;k\geq 1.
\end{equation}
The same estimates hold for the mollified pressure $P_\eps := P_{1,\eps}+P_{2,\eps}$ as well.
\end{lemma}
\begin{proof}
First of all, the operator
\[
(-\Delta)^{-1}\operatorname{div}\operatorname{div}
\]
is a zero-order Fourier multiplier, with symbol $\frac{\xi_i\xi_j}{|\xi|^2}$, and hence is bounded on $H^s(\mathbb{R}^2)$ for every $s\in\mathbb{R}$. Therefore, once the source term lies in $H^{2\alpha-1}(\mathbb{R}^2)$, the same regularity follows for $P_1$.
For $P_1$, by the chain rule we have
\[
   \nabla^{{k-1}}(U_0\otimes V) = \sum_{j=0}^{k-1}(\nabla^j U_0\otimes \nabla^{k-1-j}V),\quad  \nabla^{{k-1}}(V\otimes V) = \sum_{j=0}^{k-1}(\nabla^j V\otimes \nabla^{k-1-j}V).
\]
When $V \in H^{k-1+\alpha}(\mathbb{R}^2)$, Lemma \ref{appendix: product estimate for the L infty case} implies that the mixed terms
\[
\nabla^{k-1}(U_0 \otimes V), \qquad \nabla^{k-1}(V \otimes U_0)
\]
belong to $H^\alpha(\mathbb{R}^2)$. Since $\alpha \in \left(\frac12,1\right)$, we have $\alpha > 2\alpha - 1$,
and hence the continuous Sobolev embedding
\[
H^\alpha(\mathbb{R}^2) \hookrightarrow H^{2\alpha-1}(\mathbb{R}^2).
\]
Therefore, these mixed terms are also controlled in $H^{2\alpha-1}(\mathbb{R}^2)$. On the other hand, the quadratic term $\nabla^{k-1}(V \otimes V)$ belongs to $H^{2\alpha-1}(\mathbb{R}^2)$ by the multiplicative embedding
\[
\mathfrak{m}:H^\alpha(\mathbb{R}^2)\times H^\alpha(\mathbb{R}^2)\hookrightarrow H^{2\alpha-1}(\mathbb{R}^2)
\]
from Lemma \ref{appendix: product estimate for H alpha times H alpha to H 2alpha-1}. Thus the source term for $\nabla^{k-1} P_1$ lies in $H^{2\alpha-1}(\mathbb{R}^2)$, which yields the stated bound for $P_1$.
For $P_2$, we note that $\nabla^k (U_0\otimes U_0)\in\Ltw$ for $k\geq 1$, thanks to the pointwise estimates of $U_0$ in Lemma \ref{lemma: estimates on U_0}. For the last statement, it suffices to note that mollification commutes with differential operators (as well as $(-\Delta)^{-1}$), is bounded on Sobolev spaces, and preserves the pointwise decay properties. 
\end{proof}

\section{Regularity upgrade and weighted energy estimates}\label{Sec 4}
Our goal in this section is to improve the regularity of $V$ to the space $\cap_{k\in\mathbb{N}}H^k(\reall^2)$ as well as the weighted Sobolev space characterized by $\|\bkt{y}^{\frac{1}{2}} V\|_{H^{1+\alpha}(\reall^2)}$ when $\alpha\in (\frac{2}{3},1)$. 
Our starting point is an arbitrary weak solution $V\in H^\alpha_\sigma(\reall^2)$ to the system \eqref{self similar eq for V} (not necessarily the one constructed in Theorem \ref{theorem: existence of solution V in H alpha}). 
In order to overcome the low initial regularity issue, we consider $V_\eps$, the mollification of $V$, so that the gradient of the velocity field makes sense. After deriving the desired uniform estimates of $V_\eps$, sending $\eps\to 0^+$ yields the same estimates for $V$. Note, however, that mollification introduces extra commutators in our equation. Specifically, for any $\eps>0$, $V_\eps\in \cap_{k\geq 0}H^k(\reall^2)$ is a solution to the system
\begin{align}\label{equation for mollified V epsilon}
\begin{split}
        \Lambda^{2\alpha}V_\eps - \frac{2\alpha-1}{2\alpha}V_\eps - \frac{1}{2\alpha}y\cdot\nabla V_\eps + \nabla P_\eps&= -(U_0\cdot\nabla U_0)_\eps -V_\eps\cdot\nabla U_0 -U_0\cdot\nabla V_\eps -V_\eps\cdot\nabla V_\eps\\
        &\quad +\frac{1}{2\alpha}\dr{(y\cdot\nabla V)_\eps - y\cdot\nabla V_\eps}+\dr{V_\eps\cdot\nabla U_0 - (V\cdot\nabla U_0)_\eps}\\
        &\quad+\dr{U_0\cdot\nabla V_\eps - (U_0\cdot\nabla V)_\eps}+\dr{V_\eps\cdot\nabla V_\eps - (V\cdot\nabla V)_\eps}\\
        &:= F(V_\eps)+\Cc(V,\eps),\qquad\qquad \div V_\eps = 0,
\end{split}
\end{align}
where we define
\begin{equation}
    F(V_\eps):= -(U_0\cdot\nabla U_0)_\eps -V_\eps\cdot\nabla U_0 -U_0\cdot\nabla V_\eps -V_\eps\cdot\nabla V_\eps,
\end{equation}
and the commutators
\begin{align}
    \Cc(V,\eps)&:= \frac{1}{2\alpha}\dr{(y\cdot\nabla V)_\eps - y\cdot\nabla V_\eps}+\dr{V_\eps\cdot\nabla U_0 - (V\cdot\nabla U_0)_\eps}\\
        &\quad+\dr{U_0\cdot\nabla V_\eps - (U_0\cdot\nabla V)_\eps}+\dr{V_\eps\cdot\nabla V_\eps - (V\cdot\nabla V)_\eps}.
\end{align}
The mollified pressure can be expressed as
\begin{equation}
    P_\eps = (-\Delta)^{-1}\div\div \dr{\dr{(U_0+V)\otimes(U_0+V)}_\eps}.
\end{equation}
We remark that \eqref{equation for mollified V epsilon} can be tested with any $\vp$ such that $\bkt{y}\vp(y)\in\Ltw$. 
For brevity, from now on we denote
\begin{align}
    &T_{1,\eps} := \frac{1}{2\alpha}((y\cdot \nabla V)_\eps - y\cdot\nabla\Ve),\quad T_{2,\eps}:=(\Ve\cdot\nabla U_0 - (V\cdot\nabla U_0)_\eps),\\
    &T_{3,\eps} :=(U_0\cdot\nabla\Ve - (U_0\cdot\nabla V)_\eps),\quad T_{4,\eps}:=(\Ve\cdot\nabla\Ve - (V\cdot\nabla V)_\eps). 
\end{align}
The following estimates of the commutator terms will be useful.
\begin{lemma}\label{lemma: commutators T_i eps estimates}
    Let $m\in \Z_{\geq 0}$ and denote the spatial derivatives $\pa_j\;(j=1,2)$ generically by $\pa$ . It holds that:\\
    (1) For $\alpha\in(\frac{1}{2},1)$,
    \begin{align}
        &\|\pa^m T_{1,\eps}\|_{H^\alpha(\reall^2)}+\|\pa^m T_{2,\eps}\|_{H^\alpha(\reall^2)}\leq C(U_0)\|V\|_{H^{m+\alpha}(\reall^2)},\quad \|\pa^mT_{2,\eps}\|_{\Ltw}\leq C(U_0)\eps^{\alpha}\|V\|_{H^{m+\alpha}(\reall^2)},\\
        &\|\pa^mT_{3,\eps}\|_{\dot H^{-\alpha}(\reall^2)}\leq C(U_0)\eps^{2\alpha-1}\|V\|_{H^{m+\alpha}(\reall^2)}.
    \end{align}
    (2) For $\alpha\in (\frac{2}{3},1)$, 
    \begin{align}
        \|\pa^mT_{4,\eps}\|_{\dot H^{-\alpha}(\reall^2)}\leq C\eps^{3\alpha-2}\|V\|^2_{H^{m+\alpha}(\reall^2)}.
    \end{align}
\end{lemma}
\begin{proof}
    For $T_{1,\eps}$, we note that
\begin{align}
    T_{1,\eps} &= -\frac{1}{2\alpha}\intt \rho_\eps(y-z)(y-z)\cdot \nabla V(z)\;dz \\
    &=-\frac{1}{2\alpha}\intt \dr{2\rho_\eps(y-z)+(y-z)\cdot\nabla\rho_\eps(y-z)}V(z)\;dz\\
    & = -\frac{1}{2\alpha}\eta_\eps*V,
\end{align}
where we denote
\[
   \eta(y):= 2\rho(y)+y\cdot\nabla\rho(y), \quad \eta_\eps(y):=\frac{1}{\eps^2}\eta\dr{\frac{y}{\eps}}.
\]
Since $\eta\in C^\infty_0(\reall^2)$, $\eta_\eps*\cdot: H^s(\reall^2) \to H^s(\reall^2)$ is uniformly bounded in $\eps$, for any $s\in \reall$. Thus, we have
\begin{equation}
     \|\pa^m T_{1,\eps}\|_{H^\alpha(\reall^2)}\leq C\|V\|_{H^{m+\alpha}(\reall^2)}.
\end{equation}
For $T_{2,\eps}$, by Lemma \ref{appendix: product estimate for the L infty case} and the fact that $U_0\in W^{k,\infty}(\reall^2)$ for any integer $k$, we have
\begin{align}
    \|\pa^mT_{2,\eps}\|_{H^\alpha(\reall^2)}&\leq \sum_{j=0}^{m}\|\pa^j\Ve\cdot\nabla \pa^{m-j}U_0\|_{H^\alpha(\reall^2)} +\sum_{j=0}^{m}\| (\pa^jV\cdot\nabla \pa^{m-j}U_0)_\eps\|_{H^\alpha(\reall^2)} \\
    &\leq C(U_0)\|V\|_{H^{m+\alpha}(\reall^2)}.
\end{align}
Similarly, by Lemma \ref{lemma: mollification} we obtain
\begin{align}
    \|\pa^mT_{2,\eps}\|_{\Ltw}&\leq \sum_{j=0}^{m}\|(\pa^jV_\eps-\pa^j V)\cdot\nabla \pa^{m-j}U_0\|_{\Ltw} \\
    &\quad+\sum_{j=0}^{m}\|\pa^jV\cdot\nabla \pa^{m-j}U_0 - (\pa^jV\cdot\nabla \pa^{m-j}U_0)_\eps\|_{\Ltw}\\
    &\leq C(U_0)\eps^{\alpha}\|V\|_{H^{m+\alpha}(\reall^2)}.
\end{align}
For $T_{3,\eps}$, we note that
\begin{align}
    T_{3,\eps} = U_0\cdot\nabla\Ve - (U_0\cdot\nabla V)_\eps = U_0\cdot\nabla(\Ve - V) + U_0\cdot\nabla V - (U_0\cdot\nabla V)_\eps. 
\end{align}
Note that by Lemma \ref{lemma: mollification}, Lemma \ref{appendix: product estimate for the L infty case}, and $\alpha>\frac{1}{2}$, for any $0\leq j\leq m$
\begin{align}
    \|\pa^{m-j}U_0\cdot\nabla(\pa^jV_\eps-\pa^jV)\|_{\dot H^{-\alpha}(\reall^2)} &= \|\div(\pa^{m-j}U_0\otimes (\pa^jV_\eps - \pa^jV))\|_{\dot H^{-\alpha}(\reall^2)}\\
    &\leq \|\pa^{m-j}U_0\otimes(\pa^jV_\eps -\pa^j V)\|_{\dot H^{1-\alpha}(\reall^2)}\\
    &\leq C(U_0)\|\pa^jV-\pa^j\Ve\|_{H^{1-\alpha}(\reall^2)}\\
    &\leq C(U_0)\eps^{2\alpha-1}\|V\|_{H^{m+\alpha}(\reall^2)}.
\end{align}
and similarly, we have
\begin{align}
    \|\pa^{m-j}U_0\cdot\nabla \pa^jV -(\pa^{m-j}U_0\cdot\nabla \pa^jV)_\eps \|_{\dot H^{-\alpha}(\reall^2)} &\leq \|\pa^{m-j}U_0\otimes \pa^jV - (\pa^{m-j}U_0\otimes \pa^jV)_\eps\|_{\dot H^{1-\alpha}(\reall^2)}\\
    &\leq 
    C(U_0)\eps^{2\alpha-1}\|V\|_{H^{m+\alpha}(\reall^2)}.
\end{align}
Finally, for $T_{4,\eps}$, we write
\begin{align}
    T_{4,\eps} =\Ve\cdot\nabla\Ve - (V\cdot\nabla V)_\eps =  \div((V+\Ve)\otimes(\Ve - V)) + \div(V\otimes V - (V\otimes V)_\eps).
\end{align}
By Lemma \ref{lemma: mollification}, the multiplicative embedding $\mathfrak{m}: H^{\alpha}(\reall^2)\times H^{2-2\alpha}(\reall^2)\hookrightarrow H^{1-\alpha}(\reall^2)$ in Lemma \ref{appendix: product estimate for H alpha times H alpha to H 2alpha-1} and the fact that $\alpha>\frac{2}{3}$, for any $0\leq j\leq m$ we have
\begin{align}
    \|\div(\pa^j(V+\Ve)\otimes\pa^{m-j}(\Ve - V))\|_{\dot H^{-\alpha}(\reall^2)}&\leq \|\pa^j(V+\Ve)\otimes\pa^{m-j}(\Ve - V)\|_{H^{1-\alpha}(\reall^2)}\\
    &\leq C\|V\|_{H^{m+\alpha}(\reall^2)}\|\pa^{m-j}\Ve - \pa^{m-j}V\|_{H^{2-2\alpha}(\reall^2)}\\
    &\leq C\eps^{3\alpha-2}\|V\|^2_{H^{m+\alpha}(\reall^2)}.
\end{align}
Similarly, by $\mathfrak{m}: H^{\alpha}(\reall^2)\times H^{\alpha}(\reall^2)\hookrightarrow H^{2\alpha-1}(\reall^2)$, for any $0\leq j\leq m$, we obtain
\begin{align}
    \|\div(\pa^jV\otimes \pa^{m-j}V - (\pa^jV\otimes \pa^{m-j}V)_\eps)\|_{\dot H^{-\alpha}(\reall^2)}&\leq \|\pa^jV\otimes \pa^{m-j}V - (\pa^jV\otimes \pa^{m-j}V)_\eps\|_{\dot H^{1-\alpha}(\reall^2)}\\
    &\leq C\eps^{3\alpha-2}\|\pa^jV\otimes \pa^{m-j}V\|_{H^{2\alpha-1}(\reall^2)}\\
    &\leq C\eps^{3\alpha-2}\|V\|^2_{H^{m+\alpha}(\reall^2)}.
\end{align}
\end{proof}

\subsection{$H^{m+\alpha}$-estimates}
We denote the spatial partial derivatives $\pa_j\;(j=1,2)$ generically by $\pa$. We note that $\pa^m(y\cdot\nabla\Ve) = y\cdot\nabla(\pa^m\Ve)+m\pa^m\Ve$. Then, applying $\pa^m$ to \eqref{equation for mollified V epsilon} yields
\begin{align}\label{eq for mollified partial m V}
\Lambda^{2\alpha}V_\eps^m - \frac{2\alpha-1+m}{2\alpha}V_\eps^m - \frac{1}{2\alpha}y\cdot\nabla V_\eps^m + \nabla \pa^m P_\eps = \pa^m F(V_\eps)+ \pa^m\Cc(V,\eps),\quad \div V^m_\eps = 0,       
\end{align}
where we denote $\Ve^m := \pa^m\Ve$ for simplicity.
\begin{proposition}\label{proposition: H m+alpha estimate}
    For any $m\in \Z_{\geq 0}$, let $V\in H^{m+\alpha}(\reall^2)$ be a weak solution to \eqref{self similar eq for V} for $\alpha\in (\frac{2}{3},1)$. Then, it holds that $V\in H^{m+1+\alpha}(\reall^2)$ and
    \begin{equation}
        \|V\|_{H^{m+1+\alpha}(\reall^2)}\leq C(U_0,\|V\|_{H^{m+\alpha}(\reall^2)}).
    \end{equation}
    In particular, any solution $V\in H^\alpha(\reall^2)$ to \eqref{self similar eq for V} for $\alpha\in(\frac{2}{3},1)$ is actually smooth.
\end{proposition}
\begin{proof}
    As mentioned before, the threshold $\alpha=\frac{2}{3}$ arises from the control on the nonlinear term when we improve the regularity of $V$ from $H^\alpha(\reall^2)$ to $H^{1+\alpha}(\reall^2)$. 
    Denote the difference operator
    \[
    \Db f(y):= \frac{f(y+he_k)-f(y)}{h},\quad h\in\reall\backslash\{0\}.
    \]
    It is helpful to note the standard fact:
    \begin{equation}
        \|\Db f\|_{\Ltw}\leq \|\nabla f\|_{L^2(\reall^2)},\quad \forall\,h\neq 0,
    \end{equation}
    which is a result of the Fundamental Theorem of Calculus. We also have the following analogue of the chain rule
    \begin{equation}
        \Db(fg)(y) = (\Db f)(y)g(y)+f(y+he_k)(\Db g)(y), 
    \end{equation}
    and the analogue of integration by parts
    \begin{equation}
        \int f\Dbn g = -\int (\Db f) g .
    \end{equation}
    The application of the difference operator instead of partial derivatives facilitates the control of various commutators, as we do not ``really" need to control the high order error terms. This point will be clear in the estimates below.
    We denote the weight function $g_\delta(y):=\bkt{\delta y}^{-1} = \frac{1}{(1+\delta^2|y|^2)^{\frac{1}{2}}}$. 
    Here are some useful observations about $g_\delta$ ($0<\beta<1$):
    \begin{equation}
        \|\nabla g_\delta\|_{L^\infty} =\dn{\delta\frac{\delta y}{\bkt{\delta y}^3}}_{L^\infty} \leq C\delta,\quad |y\cdot \nabla g_\delta(y)|= \frac{\delta^2|y|^2}{\bkt{\delta y}^3} \leq |g_\delta(y)|,\quad  \|g_{\delta}\|_{\dot C^\beta}\leq C,
    \end{equation}
    and
    \begin{align}
        &\|g^2_\delta\|_{\dot C^{\beta}} = \sup_{a\in\reall^2\backslash\{0\}}\frac{1}{|a|^{\beta}}\dn{\frac{\delta^2(|y+a|^2-|y|^2)}{\bkt{\delta y}^2\bkt{\delta(y+a)}^2}}\leq C \sup_{a\in\reall^2\backslash\{0\}}\dn{\frac{\delta^2(2|y|\cdot |a|^{1-\beta}+|a|^{2-\beta})}{1+\delta^2|y|^2+\delta^2|a|^2}}\leq C\delta^{\beta}.
    \end{align}
    We also have the following refined commutator estimates which is proved in the appendix (Lemma \ref{appendix: refined commutator estimate}):
     \begin{equation}\label{proposition: H m+alpha estimate_refined commutator esti}
        \|g_\delta^{-\beta}[\Lambda^\alpha,g^2_\delta]f\|_{L^2}\leq C(\alpha,\beta)\delta^{\alpha}\|f\|_{L^2(\reall^2)},\quad 0\leq \beta<1+\alpha. 
    \end{equation}
    Now we are ready to prove the proposition. Note that  $-\Dbn(g^2_\delta\Db\Ve^m)$ is a valid test function to  \eqref{eq for mollified partial m V} as $\|\bkt{y}\Dbn(g^2_\delta\Db\Ve^m)\|_{\Ltw}\leq +\infty$, so that pairing with the drift term is justified. Testing \eqref{eq for mollified partial m V} with $-\Dbn(g^2_\delta\Db\Ve^m)$ yields
    \begin{align}\label{proposition: H m+alpha_test eq in H1 est}
    \begin{split}
        &\int_{\reall^2}  \Lambda^\alpha(\Db \Ve^m)\cdot\Lambda^\alpha(g^2_\delta\Db\Ve^m)-\frac{2\alpha-1+m}{2\alpha}|g_\delta\Db\Ve^m|^2\\
        &\quad-\frac{1}{2\alpha}\dr{\Db (y\cdot\nabla\Ve^m)}\cdot(g_\delta^2\Db\Ve^m)-\Db \nabla\pa^m P_\eps\cdot\dr{g^2_\delta\Db \Ve^m} \;dy\\
        &=\int_{\reall^2}\Db \pa^mF(\Ve)\cdot (g^2_\delta\Db\Ve^m)+\Db\pa^m\Cc(V,\eps)\cdot (g^2_\delta\Db\Ve^m)\;dy. 
    \end{split}
    \end{align}
    We will estimate term by term.\\
    \textbf{Estimates of the LHS:} First of all, for the diffusion term, we have
    \begin{align}
        \int_{\reall^2}  \Lambda^\alpha(\Db \Ve^m)\cdot\Lambda^\alpha(g^2_\delta\Db\Ve^m)\;dy = \|g_\delta\Lambda^\alpha(\Db\Ve^m)\|^2_{\Ltw}+ \intt \Lambda^\alpha(\Db\Ve^m)\cdot[\Lambda^\alpha,g^2_\delta](\Db\Ve^m)\;dy. 
    \end{align}
    The commutator is estimated by \eqref{proposition: H m+alpha estimate_refined commutator esti} as
    \begin{align}
        \da{\intt \Lambda^\alpha(\Db\Ve^m)\cdot[\Lambda^\alpha,g^2_\delta](\Db\Ve^m)\;dy}\leq \frac{1}{2}\|g_\delta\Lambda^\alpha(\Db\Ve^m)\|^2_{\Ltw}+C\delta^{2\alpha}\|\Db\Ve^m\|^2_{\Ltw}. 
    \end{align}
    For the drift term, through integration by parts, we get
    \begin{align}
        -\intt\frac{1}{2\alpha}\dr{\Db (y\cdot\nabla\Ve^m)}\cdot(g_\delta^2\Db\Ve^m)\;dy &= -\frac{1}{2\alpha}\intt y\cdot\nabla(\Db\Ve^m)\cdot(g^2_\delta\Db\Ve^m)\;dy\\
        &\quad-\frac{1}{2\alpha}\intt \pa_k \Ve^m(\cdot+he_k)\cdot(g^2_\delta\Db\Ve^m)\;dy\\
        &= \frac{1}{2\alpha}\|g_\delta\Db\Ve^m\|^2_{\Ltw}+\frac{1}{2\alpha}\intt y\cdot\nabla g_\delta g_\delta|\Db\Ve^m|^2\;dy\\
        &\quad-\frac{1}{2\alpha}\intt \pa_k \Ve^m(\cdot+he_k)\cdot(g^2_\delta\Db\Ve^m)\;dy.
    \end{align}
    By the boundedness of $y\cdot\nabla g_\delta$, we have
    \begin{align}
        \da{\frac{1}{2\alpha}\intt y\cdot\nabla g_\delta g_\delta|\Db\Ve^m|^2\;dy}\leq C\|g_\delta\Db\Ve^m\|^2_{\Ltw}.
    \end{align}
    Furthermore, by the commutator estimate \eqref{proposition: H m+alpha estimate_refined commutator esti} and the fact that $\alpha>\frac{1}{2}$, we get
    \begin{align}
        \da{\intt \pa_k \Ve^m(\cdot+he_k)\cdot(g^2_\delta\Db\Ve^m)\;dy} &= \|\pa_k\Ve^m\|_{\dot H^{-\alpha}(\reall^2)}\|g^2_\delta\Db\Ve^m\|_{\dot H^{\alpha}(\reall^2)}\\
        &\leq \|\Ve^m\|_{\dot H^{1-\alpha}(\reall^2)}\|g^2_\delta\Db\Ve^m\|_{\dot H^{\alpha}(\reall^2)}\\
        &\leq \frac{1}{4}\|g_\delta\Lambda^\alpha(\Db\Ve^m)\|^2_{\Ltw}+C\delta^{2\alpha}\|\Db\Ve^m\|_{\Ltw}+ C\|V\|^2_{H^{m+\alpha}(\reall^2)}.
    \end{align}
    For the pressure term, we decompose
    \begin{align}
        P_\eps &= (-\Delta)^{-1}\div\div\dr{(V\otimes U_0)_\eps+(U_0\otimes V)_\eps+(V\otimes V)_\eps}+(-\Delta)^{-1}\div\div((U_0\otimes U_0)_\eps)\\
        &:= P_{1,\eps}+P_{2,\eps}.
    \end{align}
    By Lemma \ref{lemma: control of pressure}, we have
    \begin{align}
        \|\pa^mP_{1,\eps}\|_{\Ltw}\leq C(U_0)\dr{1+\|V\|^2_{H^{m+\alpha}(\reall^2)}},
    \end{align}
    and
    \begin{equation}
        \|\pa^m\nabla P_{2,\eps}\|_{\Ltw}\leq C(U_0).
    \end{equation}
    Therefore, using integration by parts, we obtain
    \begin{align}
        \da{\intt \Db\nabla \pa^m P_\eps\cdot(g^2_\delta\Db\Ve^m)\;dy} &= \da{\intt -\pa^m P_{1,\eps}\Dbn(g_\delta\nabla g_\delta\cdot\Db\Ve^m)+\Db \pa^m P_{2,\eps} g_\delta\nabla g_\delta\cdot\Db\Ve^m\;dy}\\
        &\leq \|\pa^m P_{1,\eps}\|_{\Ltw} \|\Dbn(g_\delta\nabla g_\delta\cdot\Db\Ve^m)\|_{\Ltw}\\
        &\quad +C\delta\|\nabla\pa^m P_{2,\eps}\|_{\Ltw}\|g_\delta\Db\Ve^m\|_{\Ltw}\\
        &\leq C(U_0)\dr{1+\|V\|^3_{H^{m+\alpha}(\reall^2)}},
    \end{align}
    once we choose $\delta < |h|^2$. Collecting all these estimates, we obtain that for $\delta<|h|^2$,
    \begin{align}\label{proposition: H m+alpha estimate_LHS}
        \text{LHS}\geq \frac{1}{4}\|g_\delta\Lambda^\alpha(\Db\Ve^m)\|^2_{\Ltw} - C\|\Db\Ve^m\|^2_{\Ltw}- C(U_0)\dr{1+\|V\|^3_{H^{m+\alpha}(\reall^2)}}.
    \end{align}
    \textbf{Estimates of the $F(V_\eps)$ term on the RHS:} We expand the term involving $F(\Ve)$ as
    \begin{align}
        \intt \Db\pa^m F(\Ve)\cdot(g^2_\delta\Db\Ve^m)\;dy &= -\intt \Db\pa^m(U_0\cdot\nabla U_0)_\eps\cdot(g^2_\delta\Db\Ve^m)+\Db\pa^m(\Ve\cdot\nabla U_0)\cdot(g^2_\delta\Db\Ve^m)\\
        &\quad +\Db\pa^m(U_0\cdot\nabla \Ve)\cdot(g^2_\delta\Db\Ve^m)+\Db\pa^m(\Ve\cdot\nabla \Ve)\cdot(g^2_\delta\Db\Ve^m)\;dy\\
        & := F_1+F_2+F_3+F_4.
    \end{align}
    For the first term above, since $\Db\pa^m(U_0\cdot\nabla U_0)_\eps\in\Ltw$ by the pointwise estimates of $U_0$, we have
    \begin{align}
        |F_1|\leq C(U_0)\|g_\delta\Db\Ve^m\|_{\Ltw}.
    \end{align}
    Similarly, for the second term, by the chain rule and H\"older's inequality we get
    \begin{equation}
        |F_2|\leq C(U_0)\dr{\|g_\delta \Db\Ve^m\|^2_{\Ltw}+\|\Ve\|^2_{H^m(\reall^2)}}.
    \end{equation}
    For the third term, we note that
    \begin{align}
        \Db\pa^m(U_0\cdot\nabla \Ve) &= \sum_{j=0}^{m}\pa^{m-j} U_0(\cdot+he_k)\cdot\nabla(\Db\pa^j\Ve)+\sum_{j=0}^{m}\pa^{m-j} \Db U_0\cdot\nabla(\pa^j\Ve)\\
        &= U_0(\cdot+he_k)\cdot\nabla(\Db\Ve^m)+\pa U_0(\cdot+he_k)\cdot\nabla(\Db\pa^{m-1}\Ve)+ \Db U_0\cdot\nabla\Ve^m\\
        &\quad+\dr{\sum_{j=0}^{m-2}\pa^{m-j} U_0(\cdot+he_k)\cdot\nabla(\Db\pa^j\Ve)+\sum_{j=0}^{m-1}\pa^{m-j} \Db U_0\cdot\nabla(\pa^j\Ve)}\\
        &:= \Psi_1(y)+\Psi_2(y)+\Psi_3(y)+\Psi_4(y).
    \end{align}
    For $\Psi_1$, through integration by parts, we obtain
    \begin{align}
        \da{\intt \Psi_1\cdot(g^2_\delta\Db\Ve^m)\;dy} = \da{\intt U_0(\cdot+he_k)\cdot\nabla g_\delta g_\delta|\Db\Ve^m|^2\;dy}\leq C(U_0)\|\Ve^m\|^2_{\Ltw},
    \end{align}
    once we choose $\delta< |h|^2$. For $\Psi_2$, we have
    \begin{align}
        \da{\intt \Psi_2\cdot(g^2_\delta\Db\Ve^m)\;dy}&\leq C(U_0)\da{\intt\nabla(\Db\pa^{m-1}\Ve)\cdot(g^2_\delta\Db\Ve^m)\;dy}\\
        &\leq C(U_0)\|\nabla(\Db\pa^{m-1}\Ve)\|_{\dot H^{-\alpha}(\reall^2)}\|g^2_\delta\Db\Ve^m\|_{\dot H^\alpha(\reall^2)}\\
        &\leq C(U_0)\|\Db\pa^{m-1}\Ve\|_{\dot H^{1-\alpha}(\reall^2)}\|\Lambda^\alpha(g^2_\delta\Db\Ve^m)\|_{\Ltw}\\
        &\leq \frac{1}{32}\|g_\delta\Lambda^\alpha(\Db\Ve^m)\|^2_{\Ltw} + C(U_0)\dr{\|\Db\Ve^m\|^2_{\Ltw}+\|\Ve\|^2_{H^{m+\alpha}(\reall^2)}},
    \end{align}
    where we have used the fact that $\|\Db f\|_{H^s(\reall^2)}\leq \|f\|_{H^{1+s}(\reall^2)}$ (which is clear from the Fourier multiplier estimate $\frac{|e^{ih\xi_k}-1|}{|h|}\leq |\xi_k|$). For $\Psi_3$, similarly we have
    \begin{align}
        \da{\intt \Psi_3\cdot(g^2_\delta\Db\Ve^m)\;dy}&\leq C(U_0)\da{\intt\nabla(\pa^{m}\Ve)\cdot(g^2_\delta\Db\Ve^m)\;dy}\\
        &\leq C(U_0)\|\nabla(\Ve^m)\|_{\dot H^{-\alpha}(\reall^2)}\|g^2_\delta\Db\Ve^m\|_{\dot H^\alpha(\reall^2)}\\
        &\leq\frac{1}{32}\|g_\delta\Lambda^\alpha(\Db\Ve^m)\|^2_{\Ltw} + C(U_0)\dr{\|\Db\Ve^m\|^2_{\Ltw}+\|\Ve\|^2_{H^{m+\alpha}(\reall^2)}}.
    \end{align}
    The remaining $\Psi_4$ are lower order terms, and we estimate by the H\"older inequality that
    \begin{align}
        \da{\intt \Psi_4\cdot(g^2_\delta\Db\Ve^m)\;dy}&\leq C(U_0)\dr{\|\Db\Ve^m\|^2_{\Ltw}+\|\Ve\|^2_{H^{m}(\reall^2)}}.
    \end{align}
    To summarize, we obtain
    \begin{align}
        |F_3|\leq \frac{1}{16}\|g_\delta\Lambda^\alpha(\Db\Ve^m)\|^2_{\Ltw} + C(U_0)\dr{\|\Db\Ve^m\|^2_{\Ltw}+\|\Ve\|^2_{H^{m+\alpha}(\reall^2)}}.
    \end{align}
    The last term (nonlinear term $F_4$) is where the assumption $\alpha>\frac{2}{3}$ matters. First, by the chain rule,
    \begin{align}
        \Db\pa^m(\Ve\cdot\nabla \Ve) &= \sum_{j=0}^{m}\pa^{m-j} \Ve(\cdot+he_k)\cdot\nabla(\Db\pa^j\Ve)+\sum_{j=0}^{m}\pa^{m-j} \Db \Ve\cdot\nabla(\pa^j\Ve)\\
        &= \Ve(\cdot+he_k)\cdot\nabla(\Db\Ve^m)+\Db \Ve^m\cdot\nabla(\Ve)\\
        &\quad+ \dr{\sum_{j=0}^{m-1}\pa^{m-j} \Ve(\cdot+he_k)\cdot\nabla(\Db\pa^j\Ve)+\sum_{j=1}^{m}\pa^{m-j} \Db \Ve\cdot\nabla(\pa^j\Ve)}\\
        &:=\Phi_1+\Phi_2+\Phi_3.
    \end{align}
    For $\Phi_1$, through integration by parts and the Sobolev embedding $H^\alpha(\reall^2)\hookrightarrow L^4(\reall^2)$, we get
    \begin{align}
        \da{\intt \Phi_1\cdot(g^2_\delta\Db\Ve^m)\;dy }&= \da{\intt\Ve(\cdot+he_k)\cdot\nabla g_\delta g_\delta |\Db\Ve^m|^2\;dy}\\
        &\leq C\delta\|\Ve\|_{\Ltw}\|\Db\Ve^m\|^2_{L^4(\reall^2)}\\
        &\leq C\|\Ve\|^3_{H^{m+\alpha}(\reall^2)},
    \end{align}
    once we have $\delta<|h|^2$. For the second term, we note that
    \begin{align}
        \intt \Phi_2\cdot(g^2_\delta\Db\Ve^m)\;dy &= \intt \div(g_\delta\Db\Ve^m\otimes\Ve)\cdot(g^\delta\Db\Ve^m)\;dy\\
        &\quad+\intt(\nabla g_\delta\cdot \Db\Ve^m)(\Ve\cdot g_\delta\Db\Ve^m)\;dy\\
        &:=\Upsilon_{2,1}+\Upsilon_{2,2}.
    \end{align}
    The estimate of $\Upsilon_{2,2}$ is the same as $\Phi_1$. Once we take $\delta<|h|^2$, we have
    \begin{equation}
        |\Upsilon_{2,2}|\leq C\|\Ve\|^3_{H^{m+\alpha}(\reall^2)}.
    \end{equation}
    For $\Upsilon_{2,1}$, we have
    \begin{align}
    |\Upsilon_{2,1}|&\leq \|\div(g_\delta\Db\Ve^m\otimes\Ve)\|_{\dot H^{-\alpha}(\reall^2)}\|g_\delta\Db\Ve^m\|_{\dot H^\alpha(\reall^2)}\\
    &\leq C\|g_\delta\Db\Ve^m\otimes\Ve\|_{\dot H^{1-\alpha}(\reall^2)}\|g_\delta\Db\Ve^m\|_{\dot H^\alpha(\reall^2)}.
    \end{align}
    For the nonlinear part, by the embedding $\mathfrak{m}:\dot H^{2-2\alpha}(\reall^2)\times \dot H^{\alpha}\hookrightarrow \dot H^{1-\alpha}(\reall^2)$ in Lemma \ref{appendix: product estimate for H alpha times H alpha to H 2alpha-1}, we have
    \begin{align}
        \|g_\delta\Db\Ve^m\otimes\Ve\|_{\dot H^{1-\alpha}(\reall^2)}\leq C\|g_\delta\Db\Ve^m\|_{\dot H^{2-2\alpha}}\|\Ve\|_{\dot H^\alpha(\reall^2)}.
    \end{align}
    Then, since $\alpha>\frac{2}{3}$, the interpolation through the H\"older's inequality yields
    \begin{equation}
        \|g_\delta\Db\Ve^m\|_{\dot H^{2-2\alpha}(\reall^2)}\leq \|g_\delta\Db\Ve^m\|^{\frac{3\alpha-2}{\alpha}}_{\Ltw}\|g_\delta\Db\Ve^m\|^{\frac{2-2\alpha}{\alpha}}_{\dot H^\alpha(\reall^2)}.
    \end{equation}
    Finally, by Young's inequality ($ab\leq\frac{a^p}{p}+\frac{b^q}{q}  \;,\frac{1}{p}+\frac{1}{q}=1$), we obtain
    \begin{align}
        |\Upsilon_{2,1}|&\leq C\|g_\delta\Db\Ve^m\|^{\frac{3\alpha-2}{\alpha}}_{\Ltw}\|g_\delta\Db\Ve^m\|^{\frac{2-\alpha}{\alpha}}_{\dot H^\alpha(\reall^2)}\|\Ve\|_{\dot H^\alpha(\reall^2)}\\
        &\leq \frac{1}{64}\|g_\delta\Db\Ve^m\|^2_{\dot H^\alpha(\reall^2)}+C \|\Ve\|^{\frac{2\alpha}{3\alpha-2}}_{\dot H^\alpha(\reall^2)}\|g_\delta\Db\Ve^m\|^2_{\Ltw}\\
        &\leq  \frac{1}{64}\|g_\delta\Lambda^\alpha(\Db\Ve^m)\|^2_{\Ltw}+C(\|\Ve\|_{H^\alpha(\reall^2)})\|\Db\Ve^m\|^2_{\Ltw}.
    \end{align}
    In summary, we have
    \begin{align}
        \da{\intt \Phi_2\cdot(g^2_\delta\Db\Ve^m)\;dy}\leq \frac{1}{64}\|g_\delta\Lambda^\alpha(\Db\Ve^m)\|^2_{\Ltw}+C(\|\Ve\|_{H^\alpha(\reall^2)})\|\Db\Ve^m\|^2_{\Ltw}+C\|\Ve\|^3_{H^{m+\alpha}(\reall^2)}.
    \end{align}
    The remaining $\Phi_3$ is of lower order. Indeed, for $0\leq j\leq m-1$, by the embeddings $\mathfrak{m}:H^\alpha(\reall^2)\times H^{\alpha}(\reall^2)\hookrightarrow H^{2\alpha-1}(\reall^2)$ and $H^{2\alpha-1}(\reall^2)\hookrightarrow H^{1-\alpha}(\reall^2)$, we have
    \begin{align}
        &\da{\intt \pa^{m-j} \Ve(\cdot+he_k)\cdot\nabla(\Db\pa^j\Ve)\cdot(g^2_\delta\Db\Ve^m)\;dy}\\
        &\quad\leq \|\pa^{m-j}\Ve(\cdot+he_k)\otimes\Db\pa^{j}\Ve\|_{\dot H^{1-\alpha}(\reall^2)}\|g^2_\delta\Db\Ve\|_{\dot H^{\alpha}(\reall^2)}\\
        &\quad\leq \frac{1}{64}\|g_\delta\Lambda^\alpha(\Db\Ve^m)\|^2_{\Ltw}+C\|\Db\Ve^m\|^2_{\Ltw}+C\|\Ve\|^4_{H^{m+\alpha}(\reall^2)}.
    \end{align}
    Similarly, for $1\leq j\leq m$, we get
    \begin{align}
        &\da{\intt \pa^{m-j} \Db \Ve\cdot\nabla(\pa^j\Ve)\cdot(g^2_\delta\Db\Ve^m)}\\
        &\quad\leq \|\pa^{m-j}\Db\Ve\otimes\pa^{j}\Ve\|_{\dot H^{1-\alpha}(\reall^2)}\|g^2_\delta\Db\Ve\|_{\dot H^{\alpha}(\reall^2)}\\
        &\quad\leq \frac{1}{64}\|g_\delta\Lambda^\alpha(\Db\Ve^m)\|^2_{\Ltw}+C\|\Db\Ve^m\|^2_{\Ltw}+C\|\Ve\|^4_{H^{m+\alpha}(\reall^2)}.
    \end{align}
    Collecting the estimates of $\Phi_i$'s, we obtain the estimate for the nonlinear term:
    \begin{align}
       |F_4|&= \da{\intt \Db\pa^m(\Ve\cdot\nabla\Ve)\cdot(g^2_\delta\Db\Ve^m)\;dy}\\
        &\leq \frac{1}{16}\|g_\delta\Lambda^\alpha(\Db\Ve^m)\|^2_{\Ltw}+C(\|\Ve\|_{H^\alpha(\reall^2)})\|\Db\Ve^m\|^2_{\Ltw}+C\dr{\|\Ve\|^2_{H^{m+\alpha}(\reall^2)}+\|\Ve\|^4_{H^{m+\alpha}(\reall^2)}}.
    \end{align}
    Finally, collecting all the estimates of $F_i$'s, we obtain
    \begin{align}\label{proprosition: H m+alpha estimate_RHS F}
        &\da{\intt\Db\pa^m F(\Ve)\cdot(g^2_\delta\Db\Ve^m)\;dy}\\
        &\quad\leq \frac{1}{8}\|g_\delta\Lambda^\alpha(\Db\Ve^m)\|^2_{\Ltw}+C(\|\Ve\|_{H^\alpha(\reall^2)},U_0)\|\Db\Ve^m\|^2_{\Ltw}+C(U_0)\dr{1+\|\Ve\|^4_{H^{m+\alpha}(\reall^2)}}.
    \end{align}
    \textbf{Estimates of the $\Cc(V,\eps)$ term on the RHS:} Recall that there are four commutator terms $T_{j,\eps}\; (j=1,2,3,4)$. To estimate them, we mainly rely on Lemma \ref{lemma: commutators T_i eps estimates}. For $T_{1,\eps}$, by $\|\pa^mT_{1,\eps}\|_{H^\alpha(\reall^2)}\leq C\|V\|_{H^{m+\alpha}(\reall^2)}$ we have
    \begin{align}
        \da{\intt \Db\pa^m T_{1,\eps}\cdot(g^2_\delta\Db\Ve^m)\;dy}&\leq \|\Db\pa^m T_{1,\eps}\|_{\dot H^{-\alpha}(\reall^2)}\|g^2_\delta\Db\Ve^m\|_{\dot H^\alpha(\reall^2)}\\
        &\leq C\|\pa^mT_{1,\eps}\|_{\dot H^{1-\alpha}(\reall^2)}\|\Lambda^\alpha(g^2_\delta\Db\Ve^m)\|_{\Ltw}\\
        &\leq \frac{1}{16}\|g^2_\delta\Lambda^\alpha(\Db\Ve^m)\|_{\Ltw}+C\|\Db\Ve^m\|^2_{\Ltw}+C\|V\|^2_{H^{m+\alpha}(\reall^2)}.
    \end{align}
    For $T_{2,\eps}$, by $\|\pa^m T_{2,\eps}\|_{\Ltw}\leq C(U_0)\eps^\alpha\|V\|_{H^{m+\alpha}(\reall^2)}$, we have
    \begin{align}
        \da{\intt \Db\pa^m T_{2,\eps}\cdot(g^2_\delta\Db\Ve^m)\;dy}&\leq \|\pa^m T_{2,\eps}\|_{\Ltw}\|\Dbn(g^2_\delta\Db\Ve^m)\|_{\Ltw}\\
        &\leq \|V\|^2_{H^{m+\alpha}(\reall^2)},
    \end{align}
    once we take $\eps^\alpha<|h|^2$.
    For $T_{3,\eps}$, by $\|\pa^m T_{3,\eps}\|_{\dot H^{-\alpha}(\reall^2)}\leq C(U_0)\eps^{2\alpha-1}\|V\|_{H^{m+\alpha}(\reall^2)}$, we get
    \begin{align}
        \da{\intt \Db\pa^m T_{3,\eps}\cdot(g^2_\delta\Db\Ve^m)\;dy}&\leq \|\pa^m T_{3,\eps}\|_{\dot H^{-\alpha}(\reall^2)}\|\Dbn(g^2_\delta\Db\Ve^m)\|_{\dot H^{\alpha}(\reall^2)}\\
        &\leq C(U_0)\|V\|^2_{H^{m+\alpha}(\reall^2)},
    \end{align}
    once we take $\eps^{2\alpha-1}<|h|^2$. Finally, for $T_{4,\eps}$, by  $\|\pa^m T_{4,\eps}\|_{\dot H^{-\alpha}(\reall^2)}\leq C(U_0)\eps^{3\alpha-2}\|V\|^2_{H^{m+\alpha}(\reall^2)}$, we get
    \begin{align}
        \da{\intt \Db\pa^m T_{4,\eps}\cdot(g^2_\delta\Db\Ve^m)\;dy}&\leq \|\pa^m T_{4,\eps}\|_{\dot H^{-\alpha}(\reall^2)}\|\Dbn(g^2_\delta\Db\Ve^m)\|_{\dot H^{\alpha}(\reall^2)}\\
        &\leq C\|V\|^3_{H^{m+\alpha}(\reall^2)},
    \end{align}
    once we take $\eps^{3\alpha-2}<|h|^2$. In summary, we obtain the estimate for the commutators:
    \begin{align}\label{proposition: H m+alpha estimate_RHS C}
        &\da{\intt \Db\pa^m\Cc(V,\eps)\cdot(g^2_\delta\Db\Ve^m)\;dy}\\
        &\quad\leq \frac{1}{16}\|g^2_\delta\Lambda^\alpha(\Db\Ve^m)\|_{\Ltw}+C\|\Db\Ve^m\|^2_{\Ltw}+C\dr{\|V\|^2_{H^{m+\alpha}(\reall^2)}+\|V\|^3_{H^{m+\alpha}(\reall^2)}}.
    \end{align}
    as long as $\eps$ is sufficiently small (depending on $h$).\\
    \textbf{Conclusion}: Combining \eqref{proposition: H m+alpha estimate_LHS}, \eqref{proprosition: H m+alpha estimate_RHS F}, and \eqref{proposition: H m+alpha estimate_RHS C}, we obtain
    \begin{align}
        \frac{1}{16}\|g_\delta\Lambda^\alpha(\Db\Ve^m)\|^2_{\Ltw} \leq C(U_0,\|V\|_{H^\alpha(\reall^2)})\|\Db\Ve^m\|^2_{\Ltw}+C(U_0,\|V\|_{H^{m+\alpha}(\reall^2)}).
    \end{align}
    A key observation is that the term $\|\Db\Ve^m\|_{\Ltw}$, which is of the order $H^{m+1}(\reall^2)$, can be controlled by interpolation:
    \begin{align}
        C(U_0,\|V\|_{H^\alpha(\reall^2)})\|\Db\Ve^m\|^2_{\Ltw}&\leq C(U_0,\|V\|_{H^\alpha(\reall^2)})\|\Lambda^\alpha(\Db\Ve^m)\|_{\Ltw}\|\Db\Ve^m\|_{\dot H^{-\alpha}(\reall^2)}\\
        &\leq\frac{1}{32}\|\Lambda^\alpha(\Db\Ve^m)\|^2_{\Ltw}+C(U_0,\|V\|_{H^\alpha(\reall^2)})\|\Ve^m\|^2_{\dot H^{1-\alpha}(\reall^2)}\\
        &=\frac{1}{32}\|\Lambda^\alpha(\Db\Ve^m)\|^2_{\Ltw}+C(U_0,\|V\|_{H^{m+\alpha}(\reall^2)}),
    \end{align}
    where we have used the control $\|\Db f\|_{\dot H^{-\alpha}(\reall^2)}\leq C\|f\|_{\dot H^{1-\alpha}(\reall^2)}$ from the Fourier multiplier estimate. It follows that
    \begin{align}
         \frac{1}{16}\|g_\delta\Lambda^\alpha(\Db\Ve^m)\|^2_{\Ltw} \leq \frac{1}{32}\|\Lambda^\alpha(\Db\Ve^m)\|^2_{\Ltw}+C(U_0,\|V\|_{H^{m+\alpha}(\reall^2)}).
    \end{align}
    By Fatou's lemma, taking $\delta\to 0$ yields
    \begin{align}
        \frac{1}{32}\|\Lambda^\alpha(\Db\Ve^m)\|^2_{\Ltw} \leq C(U_0,\|V\|_{H^{m+\alpha}(\reall^2)}).
    \end{align}
    Finally, taking $\eps\to 0^+$ and then $h\to 0$, we get
    \begin{align}
        \frac{1}{32}\|\Lambda^\alpha(\pa_k\pa^m V)\|^2_{\Ltw}\leq C(U_0,\|V\|_{H^{m+\alpha}(\reall^2)}).
    \end{align}
    Summing over $k=1,2$ and over all choices of $\pa^m$ yields the result.
\end{proof}

\subsection{Weighted $H^\alpha$-estimate}
In this subsection, we derive the preliminary weighted $H^\alpha$-estimate of $V$, which serves as a basis in the weighted $H^{1+\alpha}$-estimate below.
\begin{proposition}\label{proposition: weighted H alpha esti}
    Let $V\in H^\alpha_\sigma(\reall^2)$ be a weak solution to \eqref{self similar eq for V} for $\alpha\in(\frac{2}{3},1)$. Then, for any $\beta\in (0,\alpha)$ it holds that $\bkt{y}^\beta V(y)\in H^\alpha(\reall^2)$ and
    \begin{equation}
        \|\bkt{y}^\beta V\|_{H^\alpha(\reall^2)}\leq C(U_0,\beta,\|V\|_{H^\alpha(\reall^2)}).
    \end{equation}
\end{proposition}
\begin{proof}
    Consider the mollified equation \eqref{equation for mollified V epsilon}. Due to the presence of the drift term (i.e., $y\cdot\nabla V$), we need to define our test function carefully to maintain coercivity at the $L^2$ level. Define the weight function
    \begin{equation}
        w_{\delta,R}(y) := \Theta_R(w(y))h_{\delta,R}(y),
    \end{equation}
    where $w(y) := (1+|y|^2)^\frac{\beta}{2}=\bkt{y}^\beta$ and $\Theta_R(s)$ is a smooth function such that
    \begin{equation}
        \Theta_R(s) = \begin{cases}
            s,\quad 0\leq s \leq R^\beta,\\
            R^\beta+1,\quad  s>R^\beta+2,
        \end{cases}\quad 0\leq \Theta_R'(s)\leq 1.
    \end{equation}
    The extra modification term $h_{\delta,R}$ is defined as
    \begin{equation}
        h_{\delta,R}(y):= \begin{cases}
            1,&\quad |y|\leq R,\\
            \dr{\frac{|y|}{R}}^{-\delta},&\quad R<|y|\leq R^{\frac{2}{\delta}+1},\\
            R^{\frac{2}{\delta}-1}|y|^{-1},&\quad |y|>R^{\frac{2}{\delta}+1}.
        \end{cases}
    \end{equation}
    A minor remark is that for simplicity $h_{\delta,R}$ defined above only belongs to $H^1$ as its gradient is not continuous, but one can of course smooth it without changing its essential properties. 
    From the construction, $|w_{\delta,R}(y)|\leq R^\beta+1$, and $w_{\delta,R}(y)\to w(y)$ as $R\to +\infty$ for any $y\in \reall^2$. Moreover, the H\"older continuity of $w_{\delta,R}$ inherits from that of $w$:
    \begin{equation}
        \|\nabla w_{\delta,R}\|_{L^\infty(\reall^2)}+\|w_{\delta,R}\|_{\dot C^{\beta}(\reall^2)}\leq C,
    \end{equation}
    where $C$ is some universal constant independent of $\delta$ and $R$. It follows from Lemma \ref{lemma: commutator estimates} that
    \begin{equation}
        \|[\Lambda^\alpha,w_{\delta,R}]f\|_{\Ltw}\leq C\|f\|_{\Ltw},
    \end{equation}
    which means that the commutator $[\Lambda^\alpha,w_{\delta,R}]$ is under control.
    The purpose of the construction of $h_{\delta,R}$ is the following inequality:
    \begin{equation}\label{proposition: weighted H alpha esti_ineq of w delta R}
        |y\cdot\nabla h_{\delta,R}(y)| \leq 2\max(\delta h_{\delta,R},\;R^{-2}),
    \end{equation}
    where $\delta$ is chosen later to satisfy certain smallness condition.
    The tail $|y|^{-1}$ guarantees that $\Ve w_{\delta,R}^2$ is a valid test function to \eqref{equation for mollified V epsilon}, i.e., $\|\bkt{y}\Ve w_{\delta,R}^2\|_{\Ltw}<\infty$. 
    Now, we test \eqref{equation for mollified V epsilon} with $w_{\delta,R}\Ve$ and obtain
    \begin{align}\label{proposition: weighted H alpha esti_test eq}
    \begin{split}
        \intt \Lambda^{2\alpha}\Ve\cdot(w^2_{\delta,R}\Ve) - \frac{2\alpha-1}{2\alpha}|w_{\delta,R}\Ve|^2-\frac{1}{2\alpha}y\cdot\nabla\Ve\cdot(w^2_{\delta,R}\Ve) - P_\eps\div(w^2_{\delta,R}\Ve)\;dy \\
        = \intt F(\Ve)\cdot w^2_{\delta,R}\Ve+\Cc(V,\eps)\cdot w^2_{\delta,R}\Ve\;dy.
    \end{split}
    \end{align}
    It remains to investigate \eqref{proposition: weighted H alpha esti_test eq} term by term.\\
    \textbf{Estimates on the left hand side:} By integration by parts, we have
    \begin{align}
        \text{LHS} &= \|\Lambda^\alpha(w_{\delta,R} \Ve)\|^2_{\Ltw}+\frac{1-\alpha}{\alpha}\|w_{\delta,R}\Ve\|^2_{\Ltw}\\
        &\quad+\intt \Lambda^\alpha\Ve\cdot[\Lambda^\alpha,w_{\delta,R}](w_{\delta,R}\Ve) - [\Lambda^\alpha,w_{\delta,R}]\Ve\cdot\Lambda^\alpha(w_{\delta,R}\Ve)\;dy\\
        &\quad+\intt\frac{1}{2\alpha}y\cdot\nabla w_{\delta,R} w_{\delta,R}|\Ve|^2 - P_\eps\Ve\cdot\nabla(w_{\delta,R}^2)\;dy.
    \end{align}
    By the control of commutators, we get
    \begin{align}
        &\da{\intt \Lambda^\alpha\Ve\cdot[\Lambda^\alpha,w_{\delta,R}](w_{\delta,R}\Ve) - [\Lambda^\alpha,w_{\delta,R}]\Ve\cdot\Lambda^\alpha(w_{\delta,R}\Ve)\;dy}\\
        &\quad\leq \frac{1}{2}\|\Lambda^\alpha(w_{\delta,R}\Ve)\|^2_{\Ltw}+\frac{1-\alpha}{8\alpha}\|w_{\delta,R}\Ve\|^2_{\Ltw}+C\|\Ve\|^2_{H^\alpha(\reall^2)}.
    \end{align}
    For the remainder of the drift term, we calculate that
    \begin{align}
        y\cdot \nabla w_{\delta,R} w_{\delta,R} &= \beta|y|^2(1+|y|^2)^{\frac{\beta}{2}-1}\Theta_R'h_{\delta,R}w_{\delta,R}+y\cdot\nabla h_{\delta,R}\Theta_R w_{\delta,R}.
    \end{align}
    Since $\Theta_R'\geq 0$, we have
    \[
        \beta|y|^2(1+|y|^2)^{\frac{\beta}{2}-1}\Theta_R'h_{\delta,R}w_{\delta,R}\geq 0.
    \]
    Besides, by \eqref{proposition: weighted H alpha esti_ineq of w delta R},
    \begin{equation}
        \da{y\cdot\nabla h_{\delta,R}\Theta_R w_{\delta,R}}\leq 2\delta w^2_{\delta,R}+CR^{2\beta-2}.
    \end{equation}
    Therefore, if we fix $\delta\leq \frac{1-\alpha}{4}$, then we have
    \begin{align}
        \intt\frac{1}{2\alpha}y\cdot\nabla w_{\delta,R} w_{\delta,R}|\Ve|^2\;dy\geq -\frac{1-\alpha}{4\alpha}\|w_{\delta,R}\Ve\|^2_{\Ltw} - C\|\Ve\|^2_{\Ltw}.
    \end{align}
    For the pressure term, as before we decompose
    \begin{align}
        P_\eps &= (-\Delta)^{-1}\div\div\dr{(V\otimes U_0)_\eps+(U_0\otimes V)_\eps+(V\otimes V)_\eps}+(-\Delta)^{-1}\div\div((U_0\otimes U_0)_\eps)\\
        &:= P_{1,\eps}+P_{2,\eps}.
    \end{align}
    By the embedding $\mathfrak{m}:H^\alpha(\reall^2)\times H^\alpha(\reall^2)\hookrightarrow H^{2\alpha-1}(\reall^2)$ (Lemma \ref{appendix: product estimate for H alpha times H alpha to H 2alpha-1}), we know that $\|P_{1,\eps}\|_{\Ltw}\leq C(U_0)(1+\|V\|^2_{H^\alpha(\reall^2)})$.
    Then, by the fact that $|\nabla(w^2_{\delta,R})(y)|\leq C|w_{\delta,R}(y)|$, we obtain
    \begin{align}
        \da{\intt P_{1,\eps}\Ve\cdot\nabla(w^2_{\delta,R})\;dy}&\leq \frac{1-\alpha}{16\alpha}\|w_{\delta,R} V_\eps\|^2_{\Ltw}+C \|P_{1,\eps}\|^2_{\Ltw}\\
        &\leq \frac{1-\alpha}{16\alpha}\|w_{\delta,R} V_\eps\|^2_{\Ltw}+C(U_0)(1+\|V\|^4_{H^\alpha(\reall^2)}).
    \end{align}
    For $P_{2,\eps}$, observe that $\nabla(-\Delta)^{-1}\div$ is a Calderón–Zygmund operator (as it is a composition of Riesz transforms), and it is standard that the weight $\bkt{y}^{2\beta}$ belongs to the $\Ac_2$-class (more generally, in $\reall^n$ any $\bkt{y}^{\gamma}$ belongs to $\Ac_2$ for $-n<\gamma<n$; see for example \cite[Chapter \Rmnum{5}]{Stein_HarmonicAnalysis_1993}). Then, it follows from Theorem \ref{theorem: A2 conjecture}, the pointwise estimates of $U_0$, and the fact $\alpha>\frac{2}{3}$, that
    \begin{align}
        \intt|\nabla P_{2,\eps}(y)|^2\bkt{y}^{2\beta}\;dy\leq \intt |(U_0\cdot\nabla U_0)_\eps(y)|^2\bkt{y}^{2\beta}\;dy\leq C(U_0)\intt\bkt{y}^{2\beta+2-8\alpha}\;dy\leq C(U_0).
    \end{align}
    Then, H\"older's inequality yields
    \begin{align}
        \da{\intt \nabla P_{2,\eps}\cdot\Ve w_{\delta,R}^2\;dy}\leq\|w_{\delta,R}\nabla P_{2,\eps}\|_{\Ltw}\|w_{\delta,R}\Ve\|_{\Ltw} \leq\frac{1-\alpha}{16\alpha}\|w_{\delta,R}\Ve\|^2_{\Ltw}+C(U_0)
    \end{align}
    Collecting all these estimates, we obtain
    \begin{align}\label{proposition: weighted H alpha esti_LHS}
        \text{LHS}\geq \frac{1}{2}\|\Lambda^\alpha(w_{\delta,R}\Ve)\|^2_{\Ltw}+\frac{1-\alpha}{2\alpha}\|w_{\delta,R}\Ve\|^2_{\Ltw}-C(U_0)(1+\|V\|^4_{H^\alpha(\reall^2)}).
    \end{align}
    \textbf{Estimates of the $F(\Ve)$ term on the right hand side:} We expand the $F(\Ve)$ term as
    \begin{align}
        \intt F(\Ve)\cdot (w^2_{\delta,R}\Ve)\;dy = -\intt\dr{(U_0\cdot\nabla U_0)_\eps+\Ve\cdot\nabla U_0 + U_0\cdot\nabla\Ve + \Ve\cdot\nabla\Ve}\cdot (w^2_{\delta,R}\Ve)\;dy.
    \end{align}
    By Lemma \ref{lemma: estimates on U_0}, since $\alpha>\frac{2}{3}$ and $\beta<\alpha$, we have
    $$
    |\bkt{y}^\beta (U_0\cdot\nabla U_0)_\eps(y)|\leq C\bkt{y}^{1-4\alpha+\beta}.
    $$
    It follows that $\bkt{y}^\beta(U_0\cdot\nabla U_0)_\eps\in \Ltw$, and
    \begin{align}
        \da{\intt (U_0\cdot\nabla U_0)_\eps w^2_{\delta,R}\Ve\;dy}&\leq \frac{1-\alpha}{16\alpha}\|w_{\delta,R}\Ve\|^2_{\Ltw}+C\|w_{\delta,R}(U_0\cdot\nabla U_0)_\eps\|^2_{\Ltw}\\
        &\leq \frac{1-\alpha}{16\alpha}\|w_{\delta,R}\Ve\|^2_{\Ltw}+C(U_0).
    \end{align}
    For the linear terms, using integration by parts and the fact that $\|w_{\delta,R}\nabla U_0\|_{L^\infty(\reall^2)}\leq C$, we get
    \begin{align}
        \da{\intt (\Ve\cdot\nabla U_0+U_0\cdot\nabla\Ve)w^2_{\delta,R}\Ve\;dy}&\leq \da{\intt \Ve\cdot\nabla U_0w^2_{\delta,R}\Ve\;dy}+\da{\intt U_0\cdot\nabla w_{\delta,R} w_{\delta,R}|\Ve|^2\;dy}\\
        &\leq \frac{1-\alpha}{32\alpha}\|w_{\delta,R}\Ve\|^2_{\Ltw}+C(U_0)\|\Ve\|^2_{\Ltw}.
    \end{align}
    Finally, for the nonlinear term, integration by parts and the Sobolev embedding $H^\alpha(\reall^2)\hookrightarrow L^4(\reall^2)$ yield
    \begin{align}
        \da{\intt (\Ve\cdot\nabla\Ve)\cdot \Ve w^2_{\delta,R}\;dy} &= \da{\intt \Ve\cdot\nabla w_{\delta,R} w_{\delta,R}|\Ve|^2\;dy}\\
        &\leq \frac{1-\alpha}{32\alpha}\|w_{\delta,R}\Ve\|^2_{\Ltw}+\|\Ve\|^4_{H^\alpha(\reall^2)}.
    \end{align}
    Collecting all these estimates, we obtain
    \begin{align}\label{proposition: weighted H alpha_RHS F}
        \da{\intt F(\Ve)w^2_{\delta,R}\Ve\;dy} \leq  \frac{1-\alpha}{8\alpha}\|w_{\delta,R}\Ve\|^2_{\Ltw}+C(U_0)(1+\|\Ve\|^4_{H^\alpha(\reall^2)}).
    \end{align}
    \textbf{Estimates of the $\Cc(V,\eps)$ term on the right hand side:} For $T_{1,\eps}$, we note that
    \begin{align}
        T_{1,\eps}(y)w_{\delta,R}(y) &= -\frac{1}{2\alpha}\intt\eta_\eps(y-z)V(z)w_{\delta,R}(y)\;dz\\
        & = -\frac{1}{2\alpha}\intt\eta_\eps(y-z)V(z)w_{\delta,R}(z)\;dz-\frac{1}{2\alpha}\intt\eta_\eps(y-z)V(z)(w_{\delta,R}(y)-w_{\delta,R}(z))\;dz\\
        &:= J_1(\eps,\delta) + J_2(\eps,\delta).
    \end{align}
    Since $\intt\eta_\eps(z)\;dz = \intt\eta(z)\;dz = \intt \nabla\cdot(z\rho(z))\;dz=0$ and $V w_{\delta,R}\in L^2(\reall^2)$ (although its $L^2$-norm depends on $\delta$ and $R$), we have
    \begin{equation}
        \|J_1(\eps,\delta)\|_{\Ltw} = \frac{1}{2\alpha}\|\eta_\eps*(V w_{\delta,R})\|_{\Ltw}\to 0,\quad{as}\quad \eps\to 0^+,
    \end{equation}
    for any fixed $\delta,R$. Since $|w_{\delta,R}(y)-w_{\delta,R}(z)|\leq \|\nabla w_{\delta,R}\|_{L^\infty}|y-z|\leq C|y-z|$, by Young's convolution inequality, we obtain
    \begin{align}
        \|J_2(\eps,\delta)\|_{\Ltw} \leq C\eps\|(|y|\eta(y))_\eps*|V|\|_{\Ltw}\leq C\eps\|V\|_{\Ltw}.
    \end{align}
    Therefore,  H\"older's inequality yields
    \begin{align}
        \da{\intt T_{1,\eps}w^2_{\delta,R}\Ve\;dy}\leq \frac{1-\alpha}{8\alpha}\|w_{\delta,R}\Ve\|^2_{\Ltw}+C\|\eta_\eps*(Vw_{\delta,R})\|^2_{\Ltw}+C\|V\|^2_{\Ltw}.
    \end{align}
    For $T_{2,\eps}$, by Lemma \ref{lemma: commutators T_i eps estimates}, once we take $\eps^\alpha < R^{-2\beta}$ (so that $\eps^\alpha|w^2_{\delta,R}(y)|\leq C$), we obtain
    \begin{align}
        \da{\intt T_{2,\eps}w^2_{\delta,R}\Ve\;dy} &\leq\|T_{2,\eps}\|_{\Ltw}\|w^2_{\delta,R} \Ve\|_{\Ltw}\\
        &\leq C(U_0)\eps^\alpha\|V\|_{H^\alpha(\reall^2)}\||w^2_{\delta,R}\Ve\|_{\Ltw}\\
        &\leq C(U_0)\|V\|^2_{H^\alpha(\reall^2)}. 
    \end{align}
    Similarly, by Lemma \ref{lemma: commutators T_i eps estimates}, once we take $\eps^{3\alpha-2}<R^{-2\beta}$, we get
    \begin{align}
        \da{\intt (T_{3,\eps}+T_{4,\eps})w^2_{\delta,R}\Ve\;dy} &\leq \|T_{3,\eps}+T_{4,\eps}\|_{\dot H^{-\alpha(\reall^2)}}\|w^2_{\delta,R}\Ve\|_{\dot H^\alpha(\reall^2)}\\
        &\leq C(U_0)(\eps^{2\alpha-1}\|V\|_{H^\alpha(\reall^2)}+\eps^{3\alpha-2}\|V\|^2_{H^\alpha(\reall^2)})\|w^2_{\delta,R}\Ve\|_{\dot H^\alpha(\reall^2)}\\
        &\leq C(U_0)(\|V\|^2_{H^\alpha(\reall^2)}+\|V\|^3_{H^\alpha(\reall^2)}).
    \end{align}
    To summarize, we have
    \begin{align}\label{proposition: weighted H alpha_RHS C}
        &\da{\intt \Cc(V,\eps)w^2_{\delta,R}\Ve\;dy}\\
        &\quad\leq \frac{1-\alpha}{8\alpha}\|w_{\delta,R}\Ve\|^2_{\Ltw}+C\|\eta_\eps*(V w_{\delta,R})\|^2_{\Ltw}+C(U_0)(\|V\|^2_{H^\alpha(\reall^2)}+\|V\|^3_{H^\alpha(\reall^2)}).
    \end{align}
    \textbf{Conclusion:} Combining \eqref{proposition: weighted H alpha esti_LHS}, \eqref{proposition: weighted H alpha_RHS F} and \eqref{proposition: weighted H alpha_RHS C}, we have
    \begin{align}
        \frac{1}{2}\|\Lambda^\alpha(w_{\delta,R}\Ve)\|^2_{\Ltw}+\frac{1-\alpha}{4\alpha}\|w_{\delta,R}\Ve\|^2_{\Ltw}\leq C\|\eta_\eps*(w_{\delta,R}V)\|^2_{\Ltw}+C(U_0)(1+\|V\|^4_{H^\alpha(\reall^2)}).
    \end{align}
    First, since $\|\eta_\eps*(w_{\delta,R}V)\|^2_{\Ltw}\to 0$ as $\eps\to 0^+$, sending $\eps\to 0^+$ yields
    \begin{align}
         \frac{1}{2}\|\Lambda^\alpha(w_{\delta,R}V)\|^2_{\Ltw}+\frac{1-\alpha}{4\alpha}\|w_{\delta,R}V\|^2_{\Ltw}\leq C(U_0)(1+\|V\|^4_{H^\alpha(\reall^2)}).
    \end{align}
    Then, taking $R\to +\infty$, Fatou's lemma yields
    \begin{align}
         \frac{1}{2}\|\Lambda^\alpha(\bkt{y}^\beta V)\|^2_{\Ltw}+\frac{1-\alpha}{4\alpha}\|\bkt{y}^\beta V\|^2_{\Ltw}\leq C(U_0)(1+\|V\|^4_{H^\alpha(\reall^2)}).
    \end{align}
    This completes the proof.
\end{proof}
\subsection{Weighted $H^{1+\alpha}$-estimate}
\begin{proposition}\label{proposition: weighted H 1+alpha esti}
    Let $\alpha\in (\frac{2}{3},1)$. Suppose $V\in H^{1+\alpha}_\sigma(\reall^2)$ is a weak solution to \eqref{self similar eq for V}, such that $\bkt{y}^\beta V\in H^\alpha(\reall^2)$ for some $\beta\in (0,\alpha)$. Then, it holds that $\bkt{y}^\beta V(y)\in H^{1+\alpha}(\reall^2)$, and
    \begin{equation}
        \|\bkt{y}^\beta V(y)\|_{H^{1+\alpha}(\reall^2)} \leq C(U_0,\beta,\|V\|_{H^{1+\alpha}(\reall^2)},\|\bkt{y}^\beta V\|_{H^\alpha(\reall^2)})
    \end{equation}
\end{proposition}
\begin{proof}
    The key point in the proof is that: as the diffusion provides the leading order coercive part \\$\|\bkt{y}^\beta V\|_{H^{1+\alpha}(\reall^2)}$, the lower order errors of size $\|\bkt{y}V\|_{H^1(\reall^2)}$ can be controlled by the interpolation between $\|\bkt{y}^\beta V\|_{H^{1+\alpha}(\reall^2)}$ and $\|\bkt{y}^\beta V\|_{H^{\alpha}(\reall^2)}$.
    Now, we apply the partial derivative $\pa$ to the mollified velocity system \eqref{equation for mollified V epsilon} and obtain \eqref{eq for mollified partial m V} for $m=1$. Recall the notation $\Ve^1:=\pa\Ve$. Define the weight function $w_\delta(y):=\frac{(1+|y|^2)^\frac{\beta}{2}}{(1+\delta^2|y|^2)^{\frac{\beta}{2}+\frac{1}{4}}}$. Then, it is clear that $w^2_\delta\Ve^1$ is a valid test function to \eqref{eq for mollified partial m V} ($m=1$). We note that
    \begin{align}
        |y\cdot\nabla w_\delta(y)|\leq C|w_\delta(y)|,\quad\|\nabla w_\delta\|_{L^\infty(\reall^2)}+\|w_\delta\|_{\dot C^\beta(\reall^2)}\leq C.
    \end{align}
    Thus, by Lemma \ref{lemma: commutator estimates}, the commutators of the weight function are under control.
    Testing \eqref{eq for mollified partial m V} ($m=1$) with $w^2_\delta\Ve^1$ yields
    \begin{align}\label{proposition: weighted H 1+alpha esti_test eq}
    \begin{split}
        \intt \Lambda^{2\alpha}\Ve^1\cdot(w^2_{\delta}\Ve^1) - |w_{\delta}\Ve^1|^2-\frac{1}{2\alpha}y\cdot\nabla\Ve^1\cdot(w^2_{\delta}\Ve^1) - \pa P_\eps\div(w^2_{\delta}\Ve^1)\;dy \\
        = \intt \pa F(\Ve)\cdot w^2_{\delta}\Ve^1+\pa\Cc(V,\eps)\cdot w^2_{\delta}\Ve^1\;dy.
    \end{split}
    \end{align}
    \textbf{Estimates on the LHS:} For the diffusion term, we have
    \begin{align}
         \intt \Lambda^{2\alpha}\Ve^1\cdot(w^2_{\delta}\Ve^1)\;dy &= \|\Lambda^\alpha(w_{\delta}\Ve^1)\|^2_{\Ltw}+\intt \Lambda^\alpha\Ve^1\cdot[\Lambda^\alpha,w_{\delta}](w_{\delta}\Ve^1)\;dy\\
         &\quad- \intt [\Lambda^\alpha,w_{\delta}]\Ve^1\cdot\Lambda^\alpha(w_{\delta}\Ve^1)\;dy.
    \end{align}
    By the commutator estimates, we get
    \begin{align}
        &\quad\da{\intt \Lambda^\alpha\Ve^1\cdot[\Lambda^\alpha,w_{\delta}](w_{\delta}\Ve^1)-[\Lambda^\alpha,w_{\delta}]\Ve^1\cdot\Lambda^\alpha(w_{\delta}\Ve^1)\;dy}\\
        &\leq \frac{1}{2}\|\Lambda^\alpha(w_{\delta}\Ve^1)\|^2_{\Ltw}+C\|w_{\delta}\Ve^1\|^2_{\Ltw}+C\|\Ve\|^2_{H^{1+\alpha}(\reall^2)}.
    \end{align}
    For the drift term, integration by parts gives
    \begin{align}
        -\intt\frac{1}{2\alpha}y\cdot\nabla\Ve^1\cdot(w^2_{\delta}\Ve^1)\;dy = \frac{1}{2\alpha}\|w_{\delta}\Ve^1\|^2_{\Ltw} +\frac{1}{2\alpha}\intt y\cdot\nabla w_\delta w_\delta|\Ve^1|^2\;dy,
    \end{align}
    and
    \begin{align}
        \da{\frac{1}{2\alpha}\intt y\cdot\nabla w_\delta w_\delta|\Ve^1|^2\;dy}\leq C\|w_{\delta}\Ve^1\|^2_{\Ltw}.
    \end{align}
    Finally, for the pressure term, since $\|\nabla P_\eps\|_{\Ltw}\leq C(U_0)(1+\|V\|^2_{H^{1+\alpha}(\reall^2)})$ due to Lemma \ref{lemma: control of pressure}, we have 
    \begin{align}
        \da{\intt \pa P_\eps\div(w^2_{\delta}\Ve^1)\;dy}&=\da{2\intt \pa P_\eps w_\delta\nabla w_{\delta}\cdot\Ve^1\;dy}\\
        &\leq C\|w_\delta \Ve^1\|^2_{\Ltw}+C(U_0)\dr{1+\|V\|^4_{H^{1+\alpha}(\reall^2)}}.
    \end{align}
    In summary, we conclude that
    \begin{align}\label{proposition: weighted H 1+alpha esti_LHS esti}
        \text{LHS}\geq \frac{1}{2}\|\Lambda^\alpha(w_\delta\Ve^1)\|^2_{\Ltw}-C\|w_\delta \Ve^1\|^2_{\Ltw}-C(U_0)\dr{1+\|V\|^4_{H^{1+\alpha}(\reall^2)}}
    \end{align}
    \textbf{Estimate of the $F(\Ve)$ term on the RHS}: We expand this part as
    \begin{align}
        \intt  \pa F(\Ve)\cdot(w_{\delta}^2\Ve^1)\;dy = -\intt\pa\dr{(U_0\cdot\nabla U_0)_\eps+\Ve\cdot\nabla U_0 + U_0\cdot\nabla\Ve + \Ve\cdot\nabla\Ve}\cdot (w^2_{\delta}\Ve^1)\;dy.
    \end{align}
    By Lemma \ref{lemma: estimates on U_0}, since $\alpha>\frac{2}{3}$ and $\beta<\alpha$, we have
    $$
    |\bkt{y}^\beta \pa(U_0\cdot\nabla U_0)_\eps(y)|\leq C\bkt{y}^{1-4\alpha+\beta}.
    $$
    It follows that $\bkt{y}^\beta\pa(U_0\cdot\nabla U_0)_\eps\in \Ltw$, and
    \begin{align}
        \da{\intt \pa(U_0\cdot\nabla U_0)_\eps w^2_{\delta}\Ve^1\;dy}\leq \|w_{\delta}\Ve^1\|^2_{\Ltw}+C(U_0).
    \end{align}
    For the linear terms, integration by parts yields
    \begin{align}
        \intt \pa\dr{\Ve\cdot\nabla U_0+U_0\cdot\nabla\Ve}\cdot(w^2_\delta\Ve^1)\;dy&=\intt \Ve\cdot\nabla(\pa U_0)\cdot(w^2_\delta\Ve^1)+(\Ve^1\cdot\nabla U_0+\pa U_0\cdot\nabla\Ve)\cdot(w^2_\delta\Ve^1)\\
        &-U_0\cdot \nabla w_\delta w_\delta|\Ve^1|^2\;dy.
    \end{align}
    Since $|\nabla^k U_0(y)w_\delta(y)|\leq C(U_0)\bkt{y}^{\beta-2\alpha}\leq C(U_0)$ for any $k\geq 1$, it is clear from the previous expression that
    \begin{align}
        \da{
        \intt \pa\dr{\Ve\cdot\nabla U_0+U_0\cdot\nabla\Ve}\cdot(w^2_\delta\Ve^1)\;dy}\leq C\|w_\delta \Ve^1\|^2_{\Ltw}+C(U_0)\|\Ve\|^2_{H^1(\reall^2)}.
    \end{align}
    For the nonlinear term, we obtain by integration by parts that
    \begin{align}
        -\intt \pa(\Ve\cdot\nabla\Ve)\cdot(w^2_\delta\Ve^1)\;dy &=  \intt \Ve\cdot\nabla w_\delta w_\delta|\Ve^1|^2 - \Ve^1\cdot\nabla\Ve\cdot(w^2_\delta\Ve^1)\;dy\\
        &=\intt \Ve\cdot\nabla w_\delta w_\delta|\Ve^1|^2 +\Ve^1\cdot\nabla w_\delta\Ve\cdot(w_\delta\Ve^1)- \Ve^1\cdot\nabla(w_\delta\Ve)\cdot(w_\delta\Ve^1)\;dy\\
        &:= N_1+N_2+N_3.
    \end{align}
    For the first two terms above, by the Sobolev embedding $H^{1+\alpha}(\reall^2)\hookrightarrow L^\infty(\reall^2)$, we have
    \begin{align}
        |N_1+N_2|&\leq \|w_\delta\Ve^1\|^2_{\Ltw}+C\|\Ve^1\otimes\Ve\|^2_{\Ltw}\\
        &\leq\|w_\delta\Ve^1\|^2_{\Ltw}+C\|\Ve\|^4_{H^{1+\alpha}(\reall^2)}.
    \end{align}
    For the third term above, by the embedding $\mathfrak{m}:H^\alpha(\reall^2)\otimes H^\alpha(\reall^2)\hookrightarrow H^{2\alpha-1}(\reall^2)$ together with the fact that $\alpha>\frac{2}{3}$, we obtain
    \begin{align}
        |N_3| &= \da{\intt \div(\Ve^1\otimes\Ve w_\delta)\cdot(w_\delta\Ve^1)\;dy}\\
        &\leq \|\div(\Ve^1\otimes\Ve w_\delta)\|_{\dot H^{-\alpha}(\reall^2)}\|\Lambda^\alpha(w_\delta\Ve^1)\|_{\Ltw}\\
        &\leq \frac{1}{4}\|\Lambda^\alpha(w_\delta\Ve^1)\|^2_{\Ltw}+ \|\Ve^1\otimes\Ve w_\delta\|^2_{\dot H^{1-\alpha}(\reall^2)}\\
        &\leq \frac{1}{4}\|\Lambda^\alpha(w_\delta\Ve^1)\|^2_{\Ltw}+C\|\Ve\|^4_{H^{1+\alpha}(\reall^2)}+C\|\bkt{y}^\beta\Ve\|^4_{H^\alpha(\reall^2)}.
    \end{align}
    Collecting the estimates of $N_i$'s, we have
    \begin{align}
        &\quad\da{\intt \pa(\Ve\cdot\nabla\Ve)\cdot(w^2_\delta\Ve^1)\;dy }\\
        &\leq \frac{1}{4}\|\Lambda^\alpha(w_\delta\Ve^1)\|^2_{\Ltw}+C\|w_\delta\Ve^1\|^2_{\Ltw}+C\|\Ve\|^4_{H^{1+\alpha}(\reall^2)}+C\|\bkt{y}^\beta\Ve\|^4_{H^\alpha(\reall^2)}.
    \end{align}
    Finally, we obtain the estimate of the $F(\Ve)$ term:
    \begin{align}\label{proposition: weighted H 1+alpha esti_RHS F}
    \begin{split}
        &\quad\da{\intt \pa F(\Ve)\cdot(w^2_\delta\Ve^1)\;dy }\\
        &\leq \frac{1}{4}\|\Lambda^\alpha(w_\delta\Ve^1)\|^2_{\Ltw}+C\|w_\delta\Ve^1\|^2_{\Ltw}+C(U_0)\dr{1+\|\Ve\|^4_{H^{1+\alpha}(\reall^2)}}+C\|\bkt{y}^\beta\Ve\|^4_{H^\alpha(\reall^2)}.
    \end{split}
    \end{align}
    \textbf{Estimates of the $\Cc(V,\eps)$ term on the RHS:} We recall that $\pa\Cc(V,\eps) = \sum_{i=1}^4\pa T_{i,\eps}$. For $T_{1,\eps}$, the same as in the previous proof, we note that
    \begin{align}
        \pa T_{1,\eps}(y)w_{\delta}(y) &= -\frac{1}{2\alpha}\intt\eta_\eps(y-z)\pa V(z)w_{\delta,R}(y)\;dz\\
        & = -\frac{1}{2\alpha}\intt\eta_\eps(y-z)\pa V(z)w_{\delta}(z)\;dy-\frac{1}{2\alpha}\intt\eta_\eps(y-z)\pa V(z)(w_{\delta}(y)-w_{\delta}(z))\;dy\\
        &:= J_1(\eps) + J_2(\eps).
    \end{align}
    Since $\intt\eta_\eps(z)\;dz = \intt\eta(z)\;dz = \intt \nabla\cdot(z\rho(z))\;dz=0$ and $\pa V w_{\delta}\in L^2(\reall^2)$, the norm of which depends on $\delta$, we have
    \begin{equation}
        \|J_1(\eps)\|_{\Ltw} = \frac{1}{2\alpha}\|\eta_\eps*(\pa V w_{\delta})\|_{\Ltw}\to 0,\quad{as}\quad \eps\to 0^+,
    \end{equation}
    for any fixed $\delta$. Since $|w_{\delta}(y)-w_{\delta}(z)|\leq \|\nabla w_{\delta}\|_{L^\infty}|y-z|\leq C|y-z|$, by Young's convolution inequality, we get
    \begin{align}
        \|J_2(\eps)\|_{\Ltw} \leq C\eps\|(|y|\eta(y))_\eps*|\pa V|\|_{\Ltw}\leq C\eps\|\pa V\|_{\Ltw}.
    \end{align}
    Therefore, H\"older's inequality yields
    \begin{align}
        \da{\intt \pa T_{1,\eps}w^2_{\delta}\Ve^1\;dy}\leq \|w_{\delta}\Ve^1\|^2_{\Ltw}+C\|\eta_\eps*(\pa Vw_{\delta})\|^2_{\Ltw}+C\|V\|^2_{H^1(\reall^2)}.
    \end{align}
    For $T_{2,\eps}$, Lemma \ref{lemma: commutators T_i eps estimates} implies that
    \begin{align}
        \da{\intt \pa T_{2,\eps}w^2_{\delta}\Ve\;dy}&\leq \|\pa T_{2,\eps}\|_{\Ltw}\|w^2_{\delta}\Ve^1\|_{\Ltw}\\
        &\leq C(U_0)\eps^\alpha\|V\|_{H^{1+\alpha}(\reall^2)}\|w_\delta^2\Ve^1\|_{\Ltw}\\
        &\leq C(U_0)\|V\|^2_{H^{1+\alpha}(\reall^2)},
    \end{align}
    once we take $\eps^\alpha\leq \delta^{2\beta}$. For $T_{3,\eps}$, again by Lemma \ref{lemma: commutators T_i eps estimates}, we get
    \begin{align}
     \da{\intt \pa T_{3,\eps}w^2_{\delta}\Ve\;dy}&\leq \|\pa T_{3,\eps}\|_{\dot H^{-\alpha}(\reall^2)}\|w^2_{\delta}\Ve^1\|_{\dot H^{\alpha}(\reall^2)}\\
        &\leq C(U_0)\eps^{2\alpha-1}\|V\|_{H^{1+\alpha}(\reall^2)}\|\Lambda^\alpha(w_\delta^2\Ve^1)\|_{\Ltw}\\
        &\leq C(U_0)\|V\|^2_{H^{1+\alpha}(\reall^2)},
    \end{align}
    once we take $\eps^{2\alpha-1}\leq \delta^{2\beta}$. Similarly, for $T_{4,\eps}$, by Lemma \ref{lemma: commutators T_i eps estimates} we have
    \begin{align}
     \da{\intt \pa T_{4,\eps}w^2_{\delta}\Ve\;dy}&\leq \|\pa T_{4,\eps}\|_{\dot H^{-\alpha}(\reall^2)}\|w^2_{\delta}\Ve^1\|_{\dot H^{\alpha}(\reall^2)}\\
        &\leq C\eps^{3\alpha-2}\|V\|^2_{H^{1+\alpha}(\reall^2)}\|\Lambda^\alpha(w_\delta^2\Ve^1)\|_{\Ltw}\\
        &\leq C\|V\|^3_{H^{1+\alpha}(\reall^2)},
    \end{align}
    once we take $\eps^{3\alpha-2}\leq \delta^{2\beta}$. Collecting all these estimates yields
    \begin{align}\label{proposition: weighted H 1+alpha esti_RHS C}
        \da{\intt \pa \Cc(V,\eps)\cdot(w^2_\delta\Ve^1)\;dy}\leq \|w_{\delta}\Ve^1\|^2_{\Ltw}+C\|\eta_\eps*(\pa Vw_{\delta})\|^2_{\Ltw}+C(U_0)\dr{1+\|V\|^3_{H^{1+\alpha}(\reall^2)}},
    \end{align}
    when $\eps$ is sufficiently small.\\
    \textbf{Conclusion:} Combining \eqref{proposition: weighted H 1+alpha esti_LHS esti}, \eqref{proposition: weighted H 1+alpha esti_RHS F}, and \eqref{proposition: weighted H 1+alpha esti_RHS C}, we have
    \begin{align}
        \frac{1}{4}\|\Lambda^\alpha(w_\delta\Ve^1)\|^2_{\Ltw}&\leq C\|w_\delta\Ve^1\|^2_{\Ltw}+C\|\eta_\eps*(\pa Vw_{\delta})\|^2_{\Ltw}\\
        &\quad+C\|\bkt{y}^\beta\Ve\|^4_{H^\alpha(\reall^2)}+C(U_0)\dr{1+\|V\|^4_{H^{1+\alpha}(\reall^2)}}.
    \end{align}
    First of all, let $\eps\to 0^+$, since $\lim_{\eps\to 0^+}\|\eta_\eps*(\pa Vw_{\delta})\|^2_{\Ltw}=0$, we have
    \begin{align}
        \frac{1}{4}\|\Lambda^\alpha(w_\delta\pa V)\|^2_{\Ltw}\leq C\|w_\delta \pa V\|^2_{\Ltw}+C\|\bkt{y}^\beta V\|^4_{H^\alpha(\reall^2)}+C(U_0)\dr{1+\|V\|^4_{H^{1+\alpha}(\reall^2)}}, 
    \end{align}
    where we have used that
    \begin{align}
        \dn{\intt \rho_\eps(y-z)\Lambda^{s}V(z)\bkt{y}^{\beta}\;dz}_{L^2_y}&\leq \dn{\intt \rho_\eps(y-z)\Lambda^{s}V(z)\bkt{z}^{\beta}\;dz}_{L^2_y}\\
        &\quad+\dn{\intt \rho_\eps(y-z)\Lambda^{s}V(z)(\bkt{z}^{\beta}-\bkt{y}^\beta)\;dz}_{L^2_y}\\
        &\leq C\|\bkt{y}^\beta\Lambda^s V\|_{\Ltw},\quad s\geq 0.
    \end{align}
    Importantly, by the interpolation inequality $\|\pa f\|^2_{\Ltw}\leq \eps'\|\Lambda^\alpha \pa f\|^2_{\Ltw} +C(\eps')\|f\|^2_{\Ltw}$ and commutator estimates, we have
    \begin{align}
        C\|w_\delta \pa V\|^2_{\Ltw}&\leq C\|\pa(w_\delta V)\|^2_{\Ltw}+C\|V\|^2_{\Ltw}\\
        &\leq \frac{1}{8}\|\Lambda^{\alpha}\pa(w_\delta V)\|^2_{\Ltw}+C\|w_\delta V\|^2_{\Ltw}+C\|V\|^2_{\Ltw}\\
        &\leq \frac{1}{8}\|\Lambda^\alpha(w_\delta\pa V)\|^2_{\Ltw}+C\|\bkt{y}^\beta V\|^2_{\Ltw}+C\|V\|^2_{H^\alpha(\reall^2)}.
    \end{align}
    It follows that
    \begin{align}
         \frac{1}{8}\|\Lambda^\alpha(w_\delta\pa V)\|^2_{\Ltw}\leq C\|\bkt{y}^\beta V\|^4_{H^\alpha(\reall^2)}+C(U_0)\dr{1+\|V\|^4_{H^{1+\alpha}(\reall^2)}}.
    \end{align}
    Finally, taking $\delta\to 0$, by Fatou's lemma and commutator estimates, we obtain
    \begin{align}
        \|\bkt{y}^\beta V\|_{H^{1+\alpha}(\reall^2)}\leq C(U_0,\beta,\|V\|_{H^{1+\alpha}(\reall^2)},\|\bkt{y}^\beta V\|_{H^\alpha(\reall^2)}).
    \end{align}
    This completes the proof.
\end{proof}

\section{Decay estimates of $V$}\label{Sec 5}
Now, let $\alpha\in (\frac{2}{3},1)$. By the previous section we know that there exists some $V\in \cap_{k\in\mathbb{N}}H^{k}(\reall^2)$ that solves \eqref{self similar eq for V}. On the other hand, we can view $V$ as a solution to the following elliptic equation with force:
\begin{align}\label{elliptic eq with F force}
    (-\Delta)^\alpha V -\frac{2\alpha-1}{2\alpha}V-\frac{1}{2\alpha}y\cdot\nabla V = F,
\end{align}
where, denoting the Leray projection by $\P := \text{Id} - \nabla(\Delta)^{-1}\div$,
\begin{equation}
    F = -\mathbb{P}[\div((U_0+V)\otimes (U_0+V))].
\end{equation}
Then, we rewrite it as an inhomogeneous fractional heat equation:
\begin{align}\label{inhomo fractional heat eq}
    \pa_t v + (-\Delta)^\alpha v = f,
\end{align}
where
\begin{align}
    v(x,t) :=  t^{\frac{1}{2\alpha}-1}V\dr{\frac{x}{t^{\frac{1}{2\alpha}}}},\quad f(x,t) = t^{\frac{1}{2\alpha}-2}F\dr{\frac{x}{t^{\frac{1}{2\alpha}}}}.
\end{align}
For any $p\in [1,+\infty]$ and $f\in L^1_{loc}((0,T);L^p(\reall^n))$, it is standard by the Duhamel principle that 
\begin{align}\label{fractional heat Duhamel}
    v(x,t) = \int_0^{t}\int_{\reall^n} H_\alpha(x-y,t-s)f(y,s)\;dt ds\in L^\infty((0,T);L^p(\reall^n))
\end{align}
is a solution to \eqref{inhomo fractional heat eq} in the distributional sense, such that $\|v(\cdot,t)\|_{L^p}\to 0$ as $t\to 0^+$. In addition, such solution is unique, i.e., if $u(x,t)\in L^\infty((0,T);L^p(\reall^n))$ is another solution to \eqref{inhomo fractional heat eq} such that $\|u(\cdot,t)\|_{L^p}\to 0$ as $t\to 0^+$, then necessarily $u=v$. The proof is a standard one, but for the sake of completeness we give it in the appendix (Lemma \ref{appendix: uniqueness of the fractional heat eq}). In our case, using the fact that $V\in H^{2}(\reall^2)$ from the previous section and the decay properties of $U_0$ (Lemma \ref{lemma: estimates on U_0}), we have $F\in L^q(\reall^2)$ for any $q\geq 2$. We also note that
\[
    \|v(\cdot,t)\|_{L^p_x} = t^{\frac{1}{2\alpha}-1+\frac{1}{\alpha p}}\|V\|_{L^p(\reall^2)},\quad \|f(\cdot,t)\|_{L^p_x} =  t^{\frac{1}{2\alpha}-2+\frac{1}{\alpha p}}\|F\|_{L^p(\reall^2)}.
\]
Therefore, when $ \frac{1}{2\alpha}-1+\frac{1}{\alpha p}>0$ and $2\leq p$, i.e., $2\leq p<\frac{2}{2\alpha-1}$, the unique solution to \eqref{inhomo fractional heat eq} in the class $L^\infty((0,T);L^p(\reall^2))$ with zero $L^p$-trace is given by the expression \eqref{fractional heat Duhamel} (with the particular forcing term that we have just discussed). In particular, setting $U = V+U_0$, we have
\begin{align}
    V(x) = v(x,1) &= \int_{0}^{1}\intt H_{\alpha}(x-y,1-s)s^{\frac{1}{2\alpha}-2}F\dr{\frac{y}{s^{\frac{1}{2\alpha}}}}\;dyds\\
    &= -\int_{0}^{1}\intt H_{\alpha}(x-y,1-s)s^{\frac{1}{2\alpha}-2}(\text{Id} - \nabla(\Delta)^{-1}\mathrm{div})\div(U\otimes U)\dr{\frac{y}{s^\frac{1}{2\alpha}}}\;dyds\\
    &:=-\int_{0}^{1}s^{\frac{1}{\alpha}-2}\intt \nabla\Oc(x-y,1-s):(U(y/s^{\frac{1}{2\alpha}})\otimes U(y/s^{\frac{1}{2\alpha}}))\;dyds,
\end{align}
where we define the non-local Oseen kernel $\Oc =\{\Oc_{jk}\}_{1\leq j,k\leq 2}$ as a symmetric $2$-tensor: 
\begin{equation}\label{Oseen kernel def}
    \Oc_{jk}(x,t) := \delta_{jk}H_\alpha(x,t) - \pa_j\pa_k (\Delta)^{-1} H_\alpha(x,t). 
\end{equation}
We can also write $V$ in local coordinates (note the order of contraction),
\begin{align}
    V_j(x) = -\int_{0}^{1}s^{\frac{1}{\alpha}-2}\intt \pa_i\Oc_{jk}(x-y,1-s) U_i(y/s^{1/2\alpha})U_k(y/s^{1/2\alpha})\;dyds, 
\end{align}
where we adopt the Einstein convention of summing repeated indices.
It has the scaling property $\nabla\Oc(x,t) = t^{-\frac{3}{2\alpha}}\nabla\Oc(t^{-\frac{1}{2\alpha}}x,1)$. Moreover,
since the operator $\nabla^{k}\pa_j\pa_k(\Delta)^{-1}$ can be represented as the convolution with a kernel that decays like $|y|^{-2}$, one can show that $|\nabla^{k}_x\Oc(x,1)|\leq C\bkt{x}^{-2-k}$. In fact, a more general decay estimate is available
\begin{equation}\label{decay esti of the Oceen kernel}
    |\nabla^k_x(-\Delta)^{\frac{\beta}{2}}\Oc(x,t)|\leq C(k,\beta) (t^{\frac{1}{2\alpha}}+|x|)^{-2-\beta-k},\quad \beta+k>-2.
\end{equation}
We refer the readers to Lemma \ref{appendix: estimates of the Oseen kernel} in the appendix for a proof. This estimate, together with the following lemma, will be helpful in deriving sharp pointwise estimates for $V$:
\begin{lemma}\label{lemma: frac esti of U_0 U_0}
    Let $\bar u_0\in C^{1,\beta}(\Sp^1)$ for some $\beta\in (0,1]$ and let $\alpha\in (\frac{1}{2},1)$. Then, for any $\gamma\in(0,\beta)$, it holds that 
    \begin{equation}\label{lemma: frac esti of U_0 U_0_div case}
        |(-\Delta)^{\frac{\gamma}{2}}(U_0\cdot\nabla U_0)(x)|\leq C(\alpha,\gamma,\bar u_0)\bkt{x}^{1-4\alpha-\gamma},
    \end{equation}
    and
    \begin{equation}\label{lemma: frac esti of U_0 U_0_no div case}
        |(-\Delta)^{\frac{\gamma}{2}}(U_0\otimes U_0)(x)|\leq C(\alpha,\gamma,\bar u_0)\bkt{x}^{2-4\alpha-\gamma},
    \end{equation}
\end{lemma}
\begin{proof}
    We first prove \eqref{lemma: frac esti of U_0 U_0_div case}.
    By the equivalent definition of the fractional Laplacian,
    \begin{equation}
        (-\Delta)^{\frac{\gamma}{2}}f = c_\gamma\intt\frac{f(x)-f(y)}{|x-y|^{2+\gamma}}\;dy,\quad \gamma\in(0,1).
    \end{equation}
    Thus, we have
    \begin{align}
        (-\Delta)^{\frac{\gamma}{2}}( U_0\cdot\nabla U_0)(x) &= ((-\Delta)^{\frac{\gamma}{2}}U_0\cdot\nabla U_0)(x)+ ( U_0\cdot\nabla(-\Delta)^{\frac{\gamma}{2}}U_0)(x)\\
        &\quad-c_\gamma\intt\frac{((U_0(x)-U_0(y))\cdot(\nabla U_0(x)-\nabla U_0(y))}{|x-y|^{2+\gamma}}\;dy.
    \end{align}
     Thanks to Lemma \ref{lemma: estimates on U_0}, we have 
     \begin{align}
         &|U_0(x)|\leq C\bkt{x}^{1-2\alpha}, \quad|\nabla U_0(x)|\leq C\bkt{x}^{-2\alpha},\\
         &|(-\Delta)^{\frac{\gamma}{2}}U_0(x)|\leq C\bkt{x}^{1-2\alpha-\gamma}, \quad|(-\Delta)^{\frac{\gamma}{2}}\nabla U_0(x)|\leq C\bkt{x}^{-2\alpha-\gamma},
     \end{align}
     and the first two terms above possess the desired decay rate. Now it remains to estimate the third term. Without loss of generality, we assume $|x|>1$. We decompose the integral into three different regions: $\Omega_1:=\{y:|x-y|\leq \frac{|x|}{2}\}$, $\Omega_2:=\{y:|y|>4|x|\}$, and $\Omega_3:=\{y:|x-y|>\frac{|x|}{2},|y|\leq 4|x|\}$. First of all, in $\Omega_1$ we have $y\sim x$, and it follows from the mean value theorem that
     \begin{align}
         \da{\int_{\Omega_1}\frac{(U_0(x)-U_0(y))\cdot(\nabla U_0(x)-\nabla U_0(y))}{|x-y|^{2+\gamma}}\;dy}&\leq \|\nabla U_0\|^2_{L^\infty(\Omega_1)}\da{\int_{\Omega_1}\frac{1}{|x-y|^{1+\gamma}}\;dy}\\
         &\leq C\bkt{x}^{1-4\alpha-\gamma}.
     \end{align}
     Secondly, by the far field decays of $U_0$ and $\nabla U_0$,
     \begin{align}
          \da{\int_{\Omega_2}\frac{(U_0(x)-U_0(y))\cdot(\nabla U_0(x)-\nabla U_0(y))}{|x-y|^{2+\gamma}}\;dy}&\leq \bkt{x}^{1-4\alpha}\da{\int_{|x-y|>\frac{|x|}{2}}\frac{1}{|x-y|^{2+\gamma}}\;dy}\\
          &\leq C\bkt{x}^{1-4\alpha-\gamma}.
     \end{align}
     For $\Omega_3$, we separate the $U_0(y)\cdot\nabla U_0(y)$ part out and use integration by parts to obtain,
     \begin{align}
          &\quad\bigg|\int_{\Omega_3}\frac{(U_0(x)-U_0(y))\cdot(\nabla U_0(x)-\nabla U_0(y))}{|x-y|^{2+\gamma}}\;dy\Bigg|\\
          &\leq\da{\int_{\Omega_3}\frac{U_0(x)\cdot\nabla(U_0(x)-U_0(y))+U_0(y)\cdot \nabla U_0(x)}{|x-y|^{2+\gamma}}\;dy}\\
          &\quad+\da{\int_{\Omega_3}U_0(y)\cdot\nabla\dr{\frac{1}{|x-y|^{2+\gamma}}}U_0(y)\;dy-\int_{\pa \Omega_3}\frac{(U_0(y)\cdot\vec n )U_0(y)}{|x-y|^{2+\gamma}}\;d\sigma}
          \\
          &\leq C\bkt{x}^{-2-\gamma}\da{\int_{|y|\leq 4|x|}\dr{\bkt{x}^{1-4\alpha}+\bkt{x}^{-2\alpha}\bkt{y}^{1-2\alpha}+\bkt{y}^{-2\alpha}\bkt{x}^{1-2\alpha}}\;dy}\\
          &\quad+\bkt{x}^{-\gamma-3}\int_{|y|\leq 4|x|}\bkt{y}^{2-4\alpha}\;dy+\bkt{x}^{1-4\alpha-\gamma}
          \\&\leq C\bkt{x}^{1-4\alpha-\gamma},
     \end{align}
     where we have used that $\int_{|y|\leq \frac{|x|}{2}}\bkt{y}^{2-4\alpha}\;dy\leq C\bkt{x}^{4-4\alpha}$ since $\alpha<1$. Collecting all these estimates completes the proof. The proof of \eqref{lemma: frac esti of U_0 U_0_no div case} is actually easier: one can directly estimate the region $\Omega_3$ without integration by parts, as $U_0\otimes U_0$ already has slow decay. The rest are exactly the same and we omit the details. 
\end{proof}
Before we establish the decay estimates of $V$, two simple estimates of the convolution with the Oseen kernel are useful:
\begin{lemma}\label{Lemma: key Oseen kernel esti for U_0 U_0}
    Let $U_0 = e^{-\Lambda^{2\alpha}u_0}$ be the fractional heat profile with $u_0(x) = |x|^{1-2\alpha}\bar u_0(\frac{x}{|x|})$ and $\alpha\in (\frac{1}{2},1)$. We denote
    \begin{equation}
        O_1(x):=  \int_{0}^{1}s^{\frac{1}{2\alpha}-2}\intt \Oc(x-y,1-s)\cdot(U_0\cdot\nabla U_0)(y/s^{\frac{1}{2\alpha}})\;dyds
    \end{equation}
    Then, the following estimates hold:\\
    (\romannumeral 1) When $\bar u_0\in C^{0,1}(\Sp^1)$,
    \begin{equation}\label{lemma: key Oseen kernel esti_low reg log loss}
        |\nabla^kO_1(x)|\leq C\bkt{x}^{1-4\alpha}\log(2+|x|),\quad\forall\; k\in\Z_{\geq0}.
    \end{equation}
    (\romannumeral 2) When $\bar u_0\in C^{1,\beta}(\Sp^1)$ for some $\beta\in(0,1]$, 
    \begin{equation}\label{lemma: key Oseen kernel esti_low reg no log loss}
        |\nabla^kO_1(x)|\leq C\bkt{x}^{1-4\alpha},\quad\forall\; k\in\Z_{\geq0}.
    \end{equation}
    (\romannumeral 3) When $\bar u_0\in C^{\infty}(\Sp^1)$, 
    \begin{equation}\label{lemma: key Oseen kernel esti_high reg}
        |\nabla^k O_1(x)|\leq C\bkt{x}^{1-4\alpha-k},\quad\forall\; k\in\Z_{\geq0}.
    \end{equation}
\end{lemma}
\begin{proof}
     These results are applications of the pointwise estimates of $\Oc$ and $\nabla^kU_0$, as well as integration by parts.
     We decompose the integral into three different regions: $\Omega_1:=\{y:|x-y|\leq \frac{|x|}{2}\}$, $\Omega_2:=\{y:|y|>4|x|\}$, and $\Omega_3:=\{y:|x-y|>\frac{|x|}{2},|y|\leq 4|x|\}$. We also assume $|x|>2$.\\
     \textit{Proof of \eqref{lemma: key Oseen kernel esti_low reg log loss}}: It suffices to consider $k=1$, since $\pa^k (U_0\cdot\nabla U_0)$ has equivalent or better decay property as $U_0\cdot\nabla U_0$. For $\Omega_1$, by the decay estimates of $\Oc$ in \eqref{decay esti of the Oceen kernel} and those of $U_0$,
     \begin{align}
         &\quad\da{\int_{0}^{1}s^{\frac{1}{2\alpha}-2}\int_{\Omega_1} \Oc(x-y,1-s)\cdot(U_0\cdot\nabla U_0)(y/s^{\frac{1}{2\alpha}})\;dyds}\\
         &\leq C \int_{0}^1\int_{\Omega_1}\frac{1}{((1-s)^\frac{1}{2\alpha}+|x-y|)^2}\cdot\frac{1}{(s^\frac{1}{2\alpha}+|y|)^{4\alpha-1}}\;dyds\\
         &\leq C\bkt{x}^{1-4\alpha}\int_0^1 \int_{|x-y|\leq \frac{|x|}{2}}\frac{1}{((1-s)^\frac{1}{2\alpha}+|x-y|)^2}\;dyds\\
         &\leq C\bkt{x}^{1-4\alpha}\int_0^1 \int_{|z|\leq(1-s)^{-\frac{1}{2\alpha}}\frac{|x|}{2} }\frac{1}{(1+|z|)^2}\;dzds\\
         & \leq C\bkt{x}^{1-4\alpha}\log(2+|x|). 
     \end{align}
     This is exactly where the logarithm loss appears: it comes from integrating the kernel over the near-field region $\Omega_1$, where the borderline singularity produces the factor $\log(2+|x|)$. For $\Omega_2$, it is straightforward to show that
     \begin{align}
         \da{\int_{0}^{1}s^{\frac{1}{2\alpha}-2}\int_{\Omega_2} \Oc(x-y,1-s)\cdot(U_0\cdot\nabla U_0)(y/s^{\frac{1}{2\alpha}})\;dyds}
         &\leq C \int_0^1\int_{|y|\geq 4|x|}\frac{1}{|y|^{4\alpha+1}}\;dyds\\
         &\leq C\bkt{x}^{1-4\alpha}.
     \end{align}
     For $\Omega_3$, through integration by parts, we note that
     \begin{align}
         \int_{\Omega_3} \Oc(x-y,1-s)\cdot(U_0\cdot\nabla U_0)(y/s^{\frac{1}{2\alpha}})\;dy &= -s^{\frac{1}{2\alpha}}\int_{\Omega_3} \nabla\Oc(x-y,1-s):(U_0\otimes U_0)(y/s^{\frac{1}{2\alpha}})\;dy\\
         &\quad+s^{\frac{1}{2\alpha}}\int_{\pa \Omega_3}(\Oc(x-y,1-s)\cdot U_0(y/s^{\frac{1}{2\alpha}}))(U_0(y/s^{\frac{1}{2\alpha}})\cdot\vec n)\;d\sigma
     \end{align}
     For the first term above, we have the pointwise estimate
     \begin{align}
         \da{\int_{\Omega_3} \nabla\Oc(x-y,1-s):(U_0\otimes U_0)(y/s^{\frac{1}{2\alpha}})\;dy}&\leq \bkt{x}^{-3}\int_{|y|\leq 4|x|}\frac{1}{(1+|y/s^{\frac{1}{2\alpha}}|)^{4\alpha-2}}\;dy\\
         &\leq Cs^{2-\frac{1}{\alpha}}\bkt{x}^{1-4\alpha}.
     \end{align}
     For the second term, since $\pa\Omega_3 = \{y: |y|=4|x|\}\cup\{y:|x-y| = \frac{|x|}{2}\}$, we get
     \begin{align}
         \da{\int_{\pa \Omega_3}(\Oc(x-y,1-s)\cdot U_0(y/s^{\frac{1}{2\alpha}}))(U_0(y/s^{\frac{1}{2\alpha}})\cdot\vec n)\;d\sigma}\leq Cs^{2-\frac{1}{\alpha}}\bkt{x}^{-4\alpha}\int_{\pa\Omega_3}\;d\sigma\leq Cs^{2-\frac{1}{\alpha}}\bkt{x}^{1-4\alpha}.
     \end{align}
     Combining these estimates, we have
     \begin{align}
          \da{\int_{0}^{1}s^{\frac{1}{2\alpha}-2}\int_{\Omega_3} \Oc(x-y,1-s)\cdot(U_0\cdot\nabla U_0)(y/s^{\frac{1}{2\alpha}})\;dyds}
         &\leq C \bkt{x}^{1-4\alpha}.
     \end{align}
     In summary, we obtain
     \begin{align}
         |O_1(x)|\leq C\bkt{x}^{1-4\alpha}\log(2+|x|),
     \end{align}
     when $\bar u_0\in C^{0,1}(\Sp^1)$.\\
     \textit{Proof of \eqref{lemma: key Oseen kernel esti_low reg no log loss}}: When $\bar u_0\in C^{1,\beta}(\Sp^1)$ for $\beta\in (0,1]$, there is a technique to remove the logarithm loss. As before, it suffices to assume $k=1$. Fix some $\gamma\in (0,\min(\beta,4-4\alpha))$, we rewrite the integral as
     \begin{align}
         \intt(-\Delta_y)^{-\frac{\gamma}{2}}\Oc(x-y,1-s)\cdot(-\Delta_y)^{\frac{\gamma}{2}}(U_0\cdot\nabla U_0)(y/s^{\frac{1}{2\alpha}})\;dy
     \end{align}
     By \eqref{decay esti of the Oceen kernel}, we have
     \begin{align}
         |\nabla^j(-\Delta_y)^{-\frac{\gamma}{2}}\Oc(x-y,1-s)|\leq C((1-s)^{\frac{1}{2\alpha}}+|x-y|)^{-2-j+\gamma},\quad j=0,1, 
     \end{align}
     and by Lemma \ref{lemma: frac esti of U_0 U_0},
     \begin{align}
         &|(-\Delta_y)^{\frac{\gamma}{2}}(U_0\cdot\nabla U_0)(y/s^{\frac{1}{2\alpha}})|\leq Cs^{2-\frac{1}{2\alpha}}(s^{\frac{1}{2\alpha}}+|y|)^{1-4\alpha-\gamma},\\
         & |(-\Delta_y)^{\frac{\gamma}{2}}(U_0\otimes U_0)(y/s^{\frac{1}{2\alpha}})|\leq Cs^{2-\frac{1}{\alpha}}(s^{\frac{1}{2\alpha}}+|y|)^{2-4\alpha-\gamma}.
     \end{align}
     Then, we estimate the integration on the $\Omega_1$ region,
     \begin{align}
         &\quad\da{\int_{0}^{1}s^{\frac{1}{2\alpha}-2}\int_{\Omega_1} (-\Delta_y)^{-\frac{\gamma}{2}}\Oc(x-y,1-s)\cdot(-\Delta_y)^{\frac{\gamma}{2}}(U_0\cdot\nabla U_0)(y/s^{\frac{1}{2\alpha}})\;dyds}\\
         &\leq C \int_{0}^1\int_{\Omega_1}\frac{1}{((1-s)^\frac{1}{2\alpha}+|x-y|)^{2-\gamma}}\cdot\frac{1}{(s^\frac{1}{2\alpha}+|y|)^{4\alpha+\gamma-1}}\;dyds\\
         &\leq C\bkt{x}^{1-4\alpha-\gamma}\int_0^1 \int_{|x-y|\leq \frac{|x|}{2}}\frac{1}{((1-s)^\frac{1}{2\alpha}+|x-y|)^{2-\gamma}}\;dyds\\
         &\leq C\bkt{x}^{1-4\alpha-\gamma}\int_0^1 \int_{|z|\leq\frac{|x|}{2} }\frac{1}{|z|^{2-\gamma}}\;dzds\\
         & \leq C\bkt{x}^{1-4\alpha}. 
     \end{align}
     The estimate on $\Omega_2$ is the same as before:
     \begin{align}
         &\quad\da{\int_{0}^{1}s^{\frac{1}{2\alpha}-2}\int_{\Omega_2} (-\Delta_y)^{-\frac{\gamma}{2}}\Oc(x-y,1-s)\cdot(-\Delta_y)^{\frac{\gamma}{2}}(U_0\cdot\nabla U_0)(y/s^{\frac{1}{2\alpha}})\;dyds}\\
         &\leq C \int_0^1\int_{|y|\geq 4|x|}\frac{1}{(1+|y|)^{4\alpha+1}}\;dyds\\
         &\leq C\bkt{x}^{1-4\alpha}.
     \end{align}   
     The estimate of $\Omega_3$ is also similar. Through integration by parts, the boundary term contributes to the desired $\bkt{x}^{1-4\alpha}$ as before, and the other term is estimated as  
     \begin{align}
         \da{\int_{\Omega_3} \nabla(-\Delta_y)^{-\frac{\gamma}{2}}\Oc(x-y,1-s):(-\Delta_y)^{\frac{\gamma}{2}}(U_0\otimes U_0)(y/s^{\frac{1}{2\alpha}})\;dy}&\leq \bkt{x}^{-3+\gamma}\int_{|y|\leq 4|x|}\frac{s^{2-\frac{1}{\alpha}}}{(s^{\frac{1}{2\alpha}}+|y|)^{4\alpha+\gamma-2}}\;dy\\
         &\leq Cs^{2-\frac{1}{\alpha}}\bkt{x}^{1-4\alpha},
     \end{align}
     once we note that $4\alpha+\gamma-2<2$ and that $\int_{|y|\leq 4|x|}\frac{1}{|y|^{4\alpha+\gamma-2}}\;dy\leq C |x|^{4-4\alpha-\gamma}$. This explains why we choose $\gamma<4-4\alpha$. Collecting all these estimates yields \eqref{lemma: key Oseen kernel esti_low reg no log loss}.\\
     \textit{Proof of \eqref{lemma: key Oseen kernel esti_high reg}}: For any $k\in\Z_{\geq0}$, we write 
     \begin{align}
         \pa^k O_1(x) = \int_0^1 s^{\frac{1}{\alpha}-2}\intt \nabla(-\Delta_y)^{-1}\Oc(x-y,1-s):\pa^k(-\Delta_y)(U_0\otimes U_0)(y/s^\frac{1}{2\alpha})\;dyds.  
     \end{align}
     When $\bar u_0\in C^{\infty}(\Sp^1)$, by \eqref{lemma: estimates on U_0} we have
     \begin{align}
         |\pa^k(-\Delta_y)(U_0\otimes U_0)(y/s^{\frac{1}{2\alpha}})|\leq C s^{2-\frac{1}{\alpha}}(s^{\frac{1}{2\alpha}}+|y|)^{-4\alpha-k}.
     \end{align}
     Besides, by \eqref{decay esti of the Oceen kernel}, we have
     \begin{align}
         |(-\Delta_y)^{-1}\Oc(x-y,1-s)|\leq C((1-s)^{\frac{1}{2\alpha}}+|x-y|)^{-1}.
     \end{align}
     The rest of the proof is straightforward. On $\Omega_1$,
     \begin{align}
         &\quad\da{ \int_0^1 s^{\frac{1}{\alpha}-2}\int_{\Omega_1} \nabla(-\Delta_y)^{-1}\Oc(x-y,1-s):\pa^k(-\Delta_y)(U_0\otimes U_0)(y/s^\frac{1}{2\alpha})\;dyds}\\
         &\leq C\bkt{x}^{-4\alpha-k}\int_0^1\int_{|x-y|\leq \frac{|x|}{2}}\frac{1}{((1-s)^{\frac{1}{2\alpha}}+|x-y|)}\;dyds\\
         &\leq C\bkt{x}^{1-4\alpha-k}.
     \end{align}
     On $\Omega_2$, we get
     \begin{align}
         &\quad\da{ \int_0^1 s^{\frac{1}{\alpha}-2}\int_{\Omega_2} \nabla(-\Delta_y)^{-1}\Oc(x-y,1-s):\pa^k(-\Delta_y)(U_0\otimes U_0)(y/s^\frac{1}{2\alpha})\;dyds}\\
         &\leq C \int_0^1 \int_{|y|\geq 4|x|}\frac{1}{|y|^{4\alpha+k+1}}\;dy\\
         &\leq C\bkt{x}^{1-4\alpha-k}.
     \end{align}
     On $\Omega_3$, we use the same integration by parts trick as before, moving the derivatives $\pa^k(-\Delta_y)$ from the $U_0$ part to the $\Oc$ part: 
     \begin{align}
         &\quad\int_{\Omega_3} \nabla(-\Delta_y)^{-1}\Oc(x-y,1-s)\cdot\pa^k(-\Delta_y)(U_0\otimes U_0)(y/s^\frac{1}{2\alpha})\;dy \\
         &= (-1)^k\int_{\Omega_3}\nabla\pa^k\Oc(x-y,1-s):(U_0\otimes U_0)(y/s^{\frac{1}{2\alpha}})\;dy+B_k(x),
    \end{align}
    where $B_k(x)$ are all the boundary terms (integrations on $\pa \Omega_3$). Again by \eqref{decay esti of the Oceen kernel}, it is clear that
    \begin{align}
        |B_k(x)|\leq Cs^{2-\frac{1}{\alpha}}\bkt{x}^{1-4\alpha-k},
    \end{align}
    and
    \begin{align}
        \da{\int_{\Omega_3}\nabla\pa^k\Oc(x-y,1-s):(U_0\otimes U_0)(y/s^{\frac{1}{2\alpha}})\;dy}&\leq C\bkt{x}^{-3-k}\int_{|y|\leq 4|x|}\frac{s^{2-\frac{1}{\alpha}}}{|y|^{4\alpha-2}}\;dy\\
        &\leq Cs^{2-\frac{1}{\alpha}}\bkt{x}^{1-4\alpha-k}.
    \end{align}
    Therefore, we obtain
    \begin{align}
        \da{ \int_0^1 s^{\frac{1}{\alpha}-2}\int_{\Omega_3} \nabla(-\Delta_y)^{-1}\Oc(x-y,1-s):\pa^k(-\Delta_y)(U_0\otimes U_0)(y/s^\frac{1}{2\alpha})\;dyds}\leq C\bkt{x}^{1-4\alpha-k}.
    \end{align}
    Collecting these estimates gives \eqref{lemma: key Oseen kernel esti_high reg}.
\end{proof}
Now we are ready to prove the main result in this section.
\begin{proposition}\label{proposition: decay esti of V}
    Let $V\in \cap_{m\in\Z_{\geq 0}}H^{m}(\reall^2)$ be a solution to \eqref{self similar eq for V}, where $U_0 = e^{-\Lambda^{2\alpha}}u_0$, $u_0(x) = |x|^{1-2\alpha}\bar u_0(\frac{x}{|x|})$, and $\alpha\in (\frac{2}{3},1)$. We assume that we have the pointwise decay estimates for $V$ as $|V(x)|\leq C(1+|x|)^{1-2\alpha}$. Then, it holds that:\\
    (\romannumeral 1)(Low regularity case) When $\bar u_0\in C^{0,1}(\Sp^1)$,
    \begin{align}\label{proposition: decay esti of V_low reg 1}
     |\nabla^{k}V(x)|\leq C(1+|x|)^{1-4\alpha}\log(2+|x|),\quad\forall\; k\in \Z_{\geq 0}.
    \end{align}
    When $\bar u_0\in C^{1,\beta}(\Sp^1)$ for some $\beta\in(0,1]$, 
    \begin{align}\label{proposition: decay esti of V_low reg 2}
     |\nabla^{k}V(x)|\leq C(1+|x|)^{1-4\alpha},\quad\forall\; k\in \Z_{\geq 0}.
    \end{align}
    (\romannumeral 2)(High regularity case)
    When $\bar u_0\in C^{\infty}(\Sp^1)$,
    \begin{align}\label{proposition: decay esti of V_high reg}
     |\nabla^{k}V(x)|\leq C(1+|x|)^{1-4\alpha-k},\quad\forall\; k\in \Z_{\geq 0}.
    \end{align}
    The constants above depend on $\bar u_0,k,\alpha$ only. 
\end{proposition}
We remark that the global $H^{m+\alpha}(\reall^2)$-estimates established in Proposition \ref{proposition: H m+alpha estimate} play a critical role here, since they provide pointwise bounds on derivatives of arbitrary order of $V$. Based on these estimates, we can further improve their decay rates to the optimal orders based on the Duhamel representation.  \\
\indent One minor remark is that although the estimates for intermediate regularity cases are available, the ones listed above are the most representative, so we will only focus on them. Other results can be easily derived by following the same ideas in the proof below.
\begin{proof}[Proof of Proposition \ref{proposition: decay esti of V}]
    We first prove the estimates for $V$ itself, and then illustrate how to derive estimates of its derivatives. We decompose 
    \begin{align}
        V(x) &= -\int_{0}^{1}s^{\frac{1}{\alpha}-2}\intt \nabla\Oc(x-y,1-s):(U_0\otimes U_0)(y/s^{\frac{1}{2\alpha}})\;dyds\\
        &\quad-\int_{0}^{1}s^{\frac{1}{\alpha}-2}\intt \nabla\Oc(x-y,1-s):(V\otimes U_0+U_0\otimes V+V\otimes V)(y/s^{\frac{1}{2\alpha}})\;dyds\\
        &:= O_1(x)+O_2(x).
    \end{align}
    Roughly speaking, the decay of $O_1$ determines the decay of $V$, since for $O_2$ the source always decays faster than $V$. Indeed, the proof is relatively simple once we have Lemma \ref{Lemma: key Oseen kernel esti for U_0 U_0}. As before, we assume $|x|\geq 2$ and decompose $\reall^2$ into $\Omega_1:=\{y:|x-y|\leq \frac{|x|}{2}\}$, $\Omega_2:=\{y:|y|>4|x|\}$, and $\Omega_3:=\{y:|x-y|>\frac{|x|}{2},|y|\leq 4|x|\}$.\\
    \textit{Proof of \eqref{proposition: decay esti of V_low reg 1} and \eqref{proposition: decay esti of V_low reg 2}}:
    Denote 
    \begin{equation}
        G(x):= (V\otimes U_0+U_0\otimes V+V\otimes V)(x).
    \end{equation}
    Suppose we have the a priori decay estimate as
    \begin{align}
        |G(y/s^{\frac{1}{2\alpha}})|\leq Cs^{\frac{q}{2\alpha}} (s^{\frac{1}{2\alpha}}+|y|)^{-q}
    \end{align}
    for some $q>0$. Then, we have the direct estimate
    \begin{align}
        &\quad\da{\int_0^1 s^{\frac{1}{\alpha}-2}\intt\nabla\Oc(x-y,1-s):G(y/s^{\frac{1}{2\alpha}})\;dyds} \\
        &=\da{\int_0^1 s^{\frac{1}{\alpha}-2}\dr{\int_{\Omega_1}+\int_{\Omega_2}+\int_{\Omega_3}}\nabla\Oc(x-y,1-s):G(y/s^{\frac{1}{2\alpha}})\;dyds} \\
        &\leq C\bkt{x}^{-q}\int_0^1 s^{\frac{q}{2\alpha}+\frac{1}{\alpha}-2}\int_{|x-y|\leq \frac{|x|}{2}}\frac{1}{((1-s)^\frac{1}{2\alpha}+|x-y|)^3}\;dyds\\
        &\quad+ C\int_0^1 s^{\frac{q}{2\alpha}+\frac{1}{\alpha}-2}\int_{|y|\geq 4|x|}\frac{1}{|y|^{3+q}}\;dyds+C|x|^{-3}\int_0^1 s^{\frac{q}{2\alpha}+\frac{1}{\alpha}-2}\int_{|y|\leq 4|x|}\frac{1}{(s^{\frac{1}{2\alpha}}+|y|)^q}\;dyds\\
        &\leq C\max(\bkt{x}^{-q},\bkt{x}^{-3}),
    \end{align}
    where we have used that $\int_0^1 s^{\frac{q}{2\alpha}+\frac{1}{\alpha}-2}(1-s)^{-\frac{1}{2\alpha}}\;ds<+\infty$ since $q>0$ and $\alpha\in(\frac{1}{2},1)$. We remark that the $\bkt{x}^{-3}$ limitation above comes from the $\Omega_3$ part, as $\nabla\Oc$ decays like $\bkt{x}^{-3}$. At the beginning, by assumption we have $G(x)\leq C\bkt{x}^{2-4\alpha}$ (i.e., $q=4\alpha-2$). By Lemma \ref{Lemma: key Oseen kernel esti for U_0 U_0}, since $|O_1(x)|\leq C\bkt{x}^{1-4\alpha}\log(2+|x|)$ or $|O_1(x)|\leq C\bkt{x}^{1-4\alpha}$, the decay  which is faster than $\bkt{x}^{2-4\alpha}$, we can deduce that $|V(x)|\leq C\bkt{x}^{2-4\alpha}$. At each iteration, we gain an additional decay factor $\langle x\rangle^{-(2\alpha-1)}$ for $V$, which matches the decay exponent of $U_0$. Repeating this process, we finally obtain
    \begin{equation}
        |V(x)|\leq C|O_1(x)|\leq \begin{cases}
            C\bkt{x}^{1-4\alpha}\log(2+|x|),\quad \bar u_0\in C^{0,1}(\Sp^1),\\
            C\bkt{x}^{1-4\alpha},\quad \bar u_0\in C^{1,\beta}(\Sp^1),\quad \beta\in (0,1].
        \end{cases}
    \end{equation}
    \textit{Proof of \eqref{proposition: decay esti of V_high reg}}: We prove by induction. Suppose for any $k\in \Z_{\geq 0}$, we have $|\nabla^{j}V(x)|\lesssim \bkt{x}^{1-4\alpha-j}$ for any $0\leq j\leq k$. We now show that $|\nabla^{k+1}V|\lesssim \bkt{x}^{-4\alpha-k}$. We note that $\pa^{k+1}V(x) = \pa^{k+1}O_1(x)+\pa^{k+1}O_2(x)$, and by Lemma \ref{Lemma: key Oseen kernel esti for U_0 U_0}, we have
    \[
    |\pa^{k+1}O_1(x)|\leq C\bkt{x}^{-4\alpha-k}.
    \]
    Thus, it remains to estimate $O_2(x)$ and apply bootstrap argument. Since $V\in \cap_{m\in\Z_{\geq0}}H^{m}(\reall^2)$, the Sobolev embedding yields $\|\nabla^{k+1} V\|_{L^\infty(\reall^2)}<+\infty$. Therefore, it holds at the beginning that
    \begin{equation}
        |\pa^{k+1} G(x)|\leq C\bkt{x}^{1-2\alpha}.
    \end{equation}
    Decomposing the domain of integrations into $\Omega_i$'s, following the same lines of calculation as before we have
    \begin{align}
        \da{\int_0^1 s^{\frac{1}{\alpha}-2}\dr{\int_{\Omega_1}+\int_{\Omega_2}}\nabla\Oc(x-y,1-s):\pa^{k+1}G(y/s^{\frac{1}{2\alpha}})\;dyds} \leq C\bkt{x}^{1-2\alpha}.
    \end{align}
    For the $\Omega_3$ parts, through integration by parts we get
    \begin{align}
        &\quad\int_0^1 s^{\frac{1}{\alpha}-2}\int_{\Omega_3}\nabla\Oc(x-y,1-s):\pa^{k+1}G(y/s^{\frac{1}{2\alpha}})\;dyds \\
        &= (-1)^{k+1}\int_0^1 s^{\frac{1}{\alpha}-2}\int_{\Omega_3}\nabla\pa^{k+1}\Oc(x-y,1-s):G(y/s^{\frac{1}{2\alpha}})\;dyds+\tilde B_k(x),
    \end{align}
    where $\tilde B_k(x)$ denotes the boundary terms. By assumption, we have $|\pa^j G(x)|\leq C\bkt{x}^{2-6\alpha-j}$ for any $0\leq j\leq k$,
    $|\tilde B_k(x)|\leq C\bkt{x}^{-6\alpha-k}$,
    and
    \begin{align}
        &\quad\da{\int_0^1 s^{\frac{1}{\alpha}-2}\int_{\Omega_3}\nabla\pa^{k+1}\Oc(x-y,1-s):G(y/s^{\frac{1}{2\alpha}})\;dyds}\\
        &\leq C|x|^{-4-k}\int_0^1s\int_{|y|\leq 4|x|}\frac{1}{(s^{\frac{1}{2\alpha}}+|y|)^{6\alpha-2}}\;dyds\\
        &\leq C\bkt{x}^{-4-k}.
    \end{align}
    We see that the $\Omega_3$ part always decays faster than the target.
    In summary, we obtain that each round of bootstrap improves the decay exponent of $\pa^{k+1}V(x)$ by $1-2\alpha$. 
    Substituting the improved decay back into $G$ improves the decay of $G$ by the factor $\bkt{x}^{1-2\alpha}$. Iterating this argument finitely many times yields the sharp decay rate $|\nabla^{k+1}V(x)|\lesssim\bkt{x}^{-4\alpha-k}$. This concludes the whole proof.
\end{proof}
Finally, we complete the proof of the main theorem. 
\begin{proof}[Proof of Theorem \ref{theorem: Main}]
    The existence of solution in $H^\alpha(\reall^2)$ for $\alpha\in(\frac{1}{2},1)$ is established in Theorem \ref{theorem: existence of solution V in H alpha}. When $\alpha\in(\frac{2}{3},1)$, Proposition \ref{proposition: H m+alpha estimate}, Proposition \ref{proposition: weighted H alpha esti}, Proposition \ref{proposition: weighted H 1+alpha esti} and the Sobolev embedding $H^{1+\alpha}(\reall^2)\hookrightarrow L^\infty(\reall^2)$ yields that 
    \[
      V\in \cap_{m\in\Z_{\geq 0}}H^m(\reall^2)\quad\text{ and }\quad |V(x)|\leq C(1+|x|)^{1-2\alpha}. 
    \]
    Hence, Proposition \ref{proposition: decay esti of V} is applicable, which, together with Lemma \ref{lemma: estimates on U_0}, gives the pointwise estimates in the theorem.
\end{proof}
\section{Numerical investigations on a related system}\label{Sec 6}
\subsection{Introduction to the Navier-Stokes equations with time dependent viscosity}
In this section, we consider the modified equation \eqref{NS with time dependent visc} which shares similarities with \eqref{fractional NS} and is also interesting by itself.
At a formal level, for $\alpha\in(\frac{1}{2},1)$ one notices that the diffusion terms in \eqref{fractional NS} and \eqref{NS with time dependent visc} are both weakened near the initial time $t=0$ compared to the classic $2$D Navier-Stokes equation. In particular, for a fixed $\alpha$, one can check that \eqref{fractional NS} and \eqref{NS with time dependent visc} share the same scaling invariance \eqref{eq: scaling invariance}, and once we consider the same self-similar change of variables \eqref{self similar change of variables} for \eqref{NS with time dependent visc}, the corresponding self-similar profiles satisfy
\begin{align}\label{self similar eq for U in the time dep visc case}
    (-\Delta)U - \frac{2\alpha-1}{2\alpha}U - \frac{1}{2\alpha}y\cdot\nabla U + U\cdot\nabla U +\nabla P=0\quad \div U = 0.
\end{align} 
Similar to the study of \eqref{fractional NS}, given any $(1-2\alpha)$-homogeneous initial data for \eqref{NS with time dependent visc}:
\begin{equation}
    u_0(x) = |x|^{1-2\alpha}\bar u_0\dr{\frac{x}{|x|}},\quad  \div u_0 = 0,
\end{equation}
we expect that the leading order of the self-similar profile $U$ of \eqref{self similar eq for U in the time dep visc case} (if it exists) in the far field is given by the heat profile
\[
    U_0(x) := e^{\alpha\Delta}u_0,
\]
which satisfies
\begin{align}
    |y|^{2\alpha-1}U_0(y)= \bar u_0\dr{\frac{y}{|y|}}+o(1)\quad \text{as}\quad y\to +\infty.
\end{align}
Due to these similarities, we conjecture that \eqref{fractional NS} and \eqref{NS with time dependent visc} share similar nonuniqueness behaviors. This is the motivation of our numerical experiments in this section. Finally, we remark that with same techniques, one may establish results similar to Theorem \ref{theorem: Main} for the equation \eqref{NS with time dependent visc}. We will not pursue this theoretical direction further here.
\subsection{Numerical observations}
In $2$D, it is convenient to consider the vorticity equation of \eqref{NS with time dependent visc}:
\begin{equation}
    \om_t + u\cdot\nabla\om = t^{\frac{1}{\alpha}-1}\Delta\om,\quad u= K_2*\om,\quad (x,t)\in\reall^2\times (0,+\infty),
\end{equation}
where $K_2$ is the usual 2-dimensional Biot-Savart kernel. By the self-similar change of variables
\begin{equation}
    \omega(x,t) = t^{-1}\Omega\dr{y},\quad u(x,t) =t^{\frac{1}{2\alpha}-1}U\dr{y} ,\quad y:=\frac{x}{t^{\frac{1}{2\alpha}}},
\end{equation}
the self-similar vorticity profile $\Omega$ satisfies
\begin{equation}\label{eq: self-similar eq for Omega}
    \Delta\Omega + \Omega+\frac{1}{2\alpha} y\cdot\nabla\Omega-U\cdot\nabla\Omega = 0, \quad U = K_2*\Omega.
\end{equation}
Then, we have the following correspondence of the homogeneous initial data and the far field asymptotics of the profile:
\begin{equation}
    \om(x,0) = |x|^{-2\alpha}\bar\om_0\dr{\frac{x}{|x|}}\quad \Longleftrightarrow\quad|y|^{2\alpha}\Omega(y)= \bar \om_0\dr{\frac{y}{|y|}}+o(1)\quad \text{as}\quad y\to +\infty.
\end{equation}
As we have discussed in the introduction, in the literature mechanisms of nonuniqueness include: (a) a direct bifurcation of self-similar profiles, and (b) the existence of unstable modes of the linearized operator. In the following, we provide numerical evidence for both mechanisms. The numerical implementation details are included in Appendix \ref{Appendix B}. Consider the initial data which are $k$-fold symmetric (i.e. invariant under the rotation by $\frac{2\pi}{k}$)\footnote{It is worth mentioning that such initial data as well as the symmetric profiles found numerically, actually possess more symmetries, for example, odd symmetries with respect to $k$ axes separated by $\frac{2\pi}{k}$. Different types of symmetry may play different roles in nonuniqueness. However, we do not explore further in this direction here.}:
\begin{equation}
    \om_0(x) = |x|^{-2\alpha}\bar\om_0\dr{\frac{x}{|x|}},\quad \bar \om_0(\theta) = \sigma\cos(k\theta),\quad \sigma>0,\;k = 1,2,3,....
\end{equation}
We will see that both the unstable modes and the bifurcation of profiles are closely linked to the breaking of this $k$-fold symmetry. We remark that the breaking of the $\Z_2$-symmetry (either of the profiles or the eigenfunctions) has been numerically observed for the velocity field in the nonuniqueness scenarios of the $2$D and $3$D Navier-Stokes equations \cite{Guillod2023, hou2025nonuniquenesslerayhopfsolutionsunforced, albritton2026forwardselfsimilarsolutions2d}. To the best of our knowledge, our result gives the first numerical observation of the nonuniqueness of the vorticity that manifests a breaking of the $k$-fold symmetry. 
In the experiments, we fix $\alpha$ and $k$, and then gradually increase the value of $\sigma$. For each $\sigma$, we first compute the profile $\Omega_\sigma$ from the symmetric branch by solving the nonlinear equation \eqref{eq: self-similar eq for Omega} with a continuation method, then we compute the first $m$ eigenvalues of the linearized operator (arranged in the increasing order of their real parts)
\[
    \Lc_{\Omega_\sigma} f:= -\Delta f - f + (U_\sigma - \frac{y}{2\alpha})\cdot\nabla f + (K_2*f)\cdot\nabla\Omega_\sigma ,\quad U_\sigma := K_2*\Omega_\sigma.
\]
We remark that with our sign convention for the linearized operator, the unstable modes correspond to those with eigenvalues with negative real parts. In the following, we focus on an exemplary case where $\alpha = \frac{5}{8}$ and $k=3$. Keeping track of the first $10$ eigenvalues, we observe crossings of exactly $3$ eigenvalue branches through the imaginary axis (i.e., real parts shift to the negative half of the real axis) at around $\sigma = 120$: 
\begin{figure}[H]
    \centering
    \includegraphics[width=0.75\textwidth]{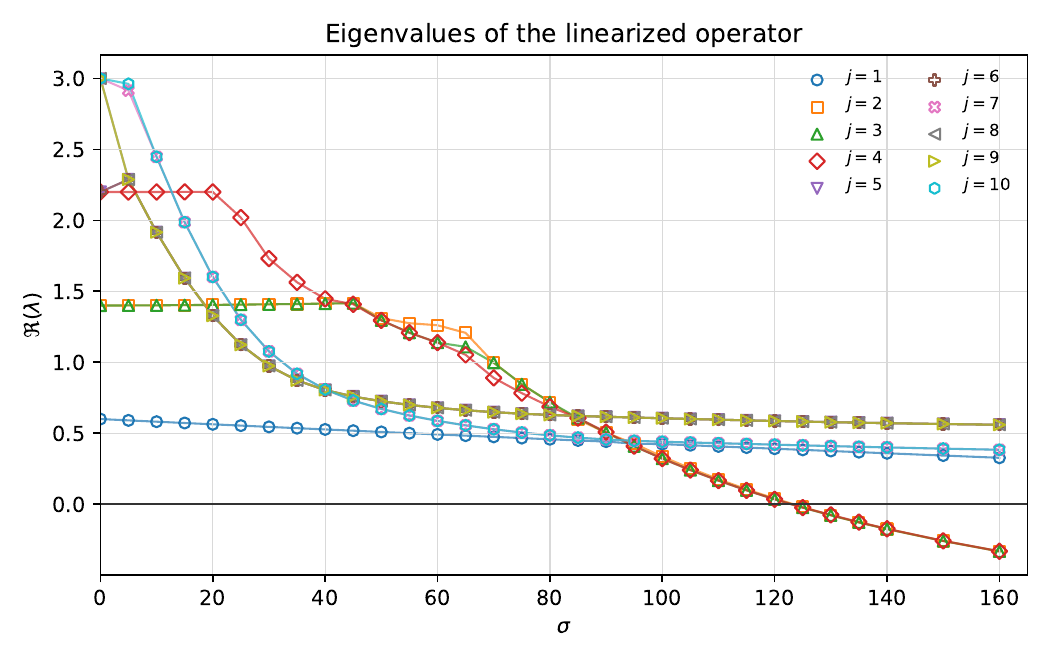}
    \caption{Numerical results of the real parts of the first ten eigenvalues with smallest real parts along the continuation branch for boundary data $\sigma\cos(3\theta)$ and $\alpha = \frac{5}{8}$.}
    \label{fig: eigenvalue-continuation}
\end{figure}
We also observe that the unstable eigenmodes break the $3$-fold symmetry of the profile:
\begin{figure}[H]
    \centering
    \includegraphics[width=0.90\textwidth]{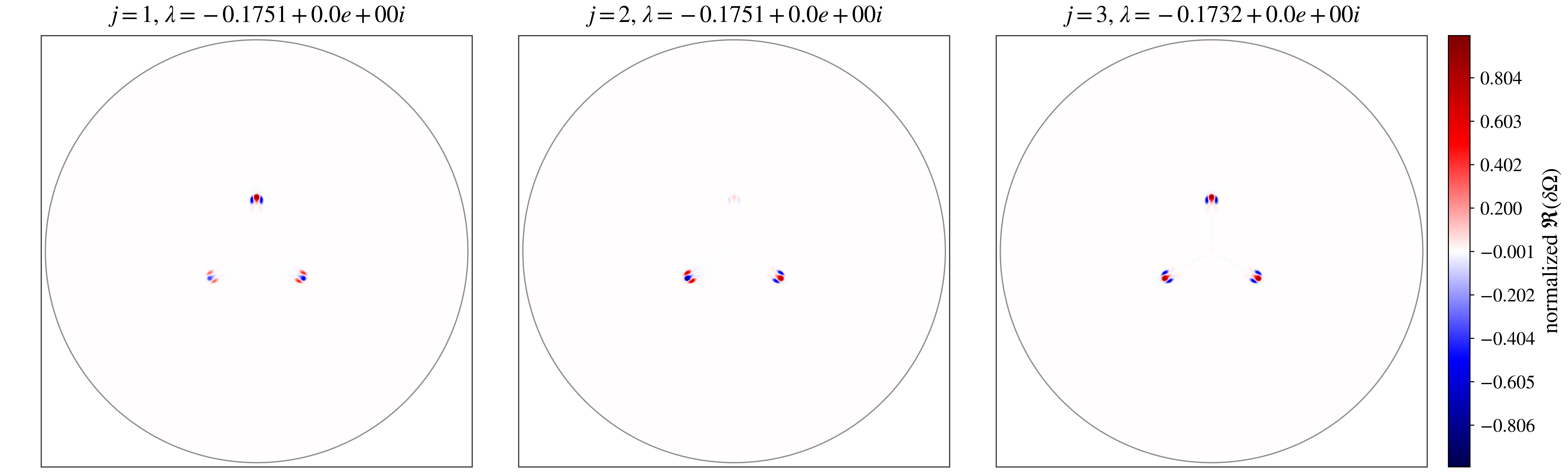}
    \caption{Heat maps of the three unstable eigenfunctions for boundary data $140\cos(3\theta)$ and $\alpha = \frac{5}{8}$.}
    \label{fig: unstable modes}
\end{figure}
Thus, due to the $3$-fold symmetry of the linearized operator, such symmetry-breaking unstable eigenspace has dimension at least $3$, which aligns with the numerical results in Figure \ref{fig: eigenvalue-continuation}.

Next, we explore the bifurcation of self-similar profiles. We add a localized symmetry-breaking perturbation (so that the boundary condition is unchanged) of the form
\[
    p(r,\theta) = c r\cos(\theta)e^{-(r/r_0)^2}\dr{1-(r/R_*)^2}^2
\]
to the profile with boundary data $140\cos(3\theta)$, take it as the initial guess, and run the nonlinear solver ($R_*$ above is the radius of the computational domain so that $p$ does not change the boundary value of the profile). We observe that the solver converged to a different, symmetry-breaking profile:
\begin{figure}[H]
    \centering
    \includegraphics[width=0.90\textwidth]{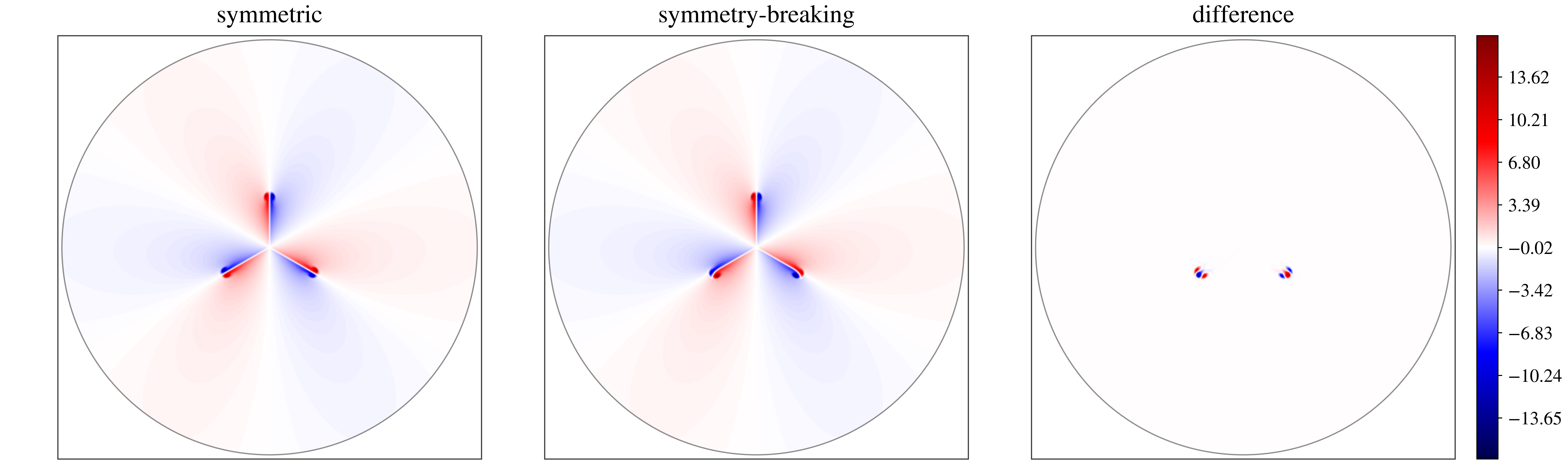}
    \caption{Comparison of the symmetric profile and the symmetry-breaking one for boundary data $140\cos(3\theta)$ and $\alpha = \frac{5}{8}$.}
    \label{fig: profile bifurcation}
\end{figure}
Starting from the bifurcation at $\sigma = 140$, we keep track of this symmetry-breaking branch in $\sigma$ by scaling the symmetry-breaking profile by the new/old $\sigma$ ratio and running the profile solver. We observe that bifurcation emerges in the same $\sigma$ bracket ($120\leq\sigma\leq 125$) where the crossings of eigenvalues take place. This aligns with the theoretical prediction that crossing of eigenvalues triggers bifurcation of profiles:
\begin{figure}[H]
    \centering
    \includegraphics[width=0.75\textwidth]{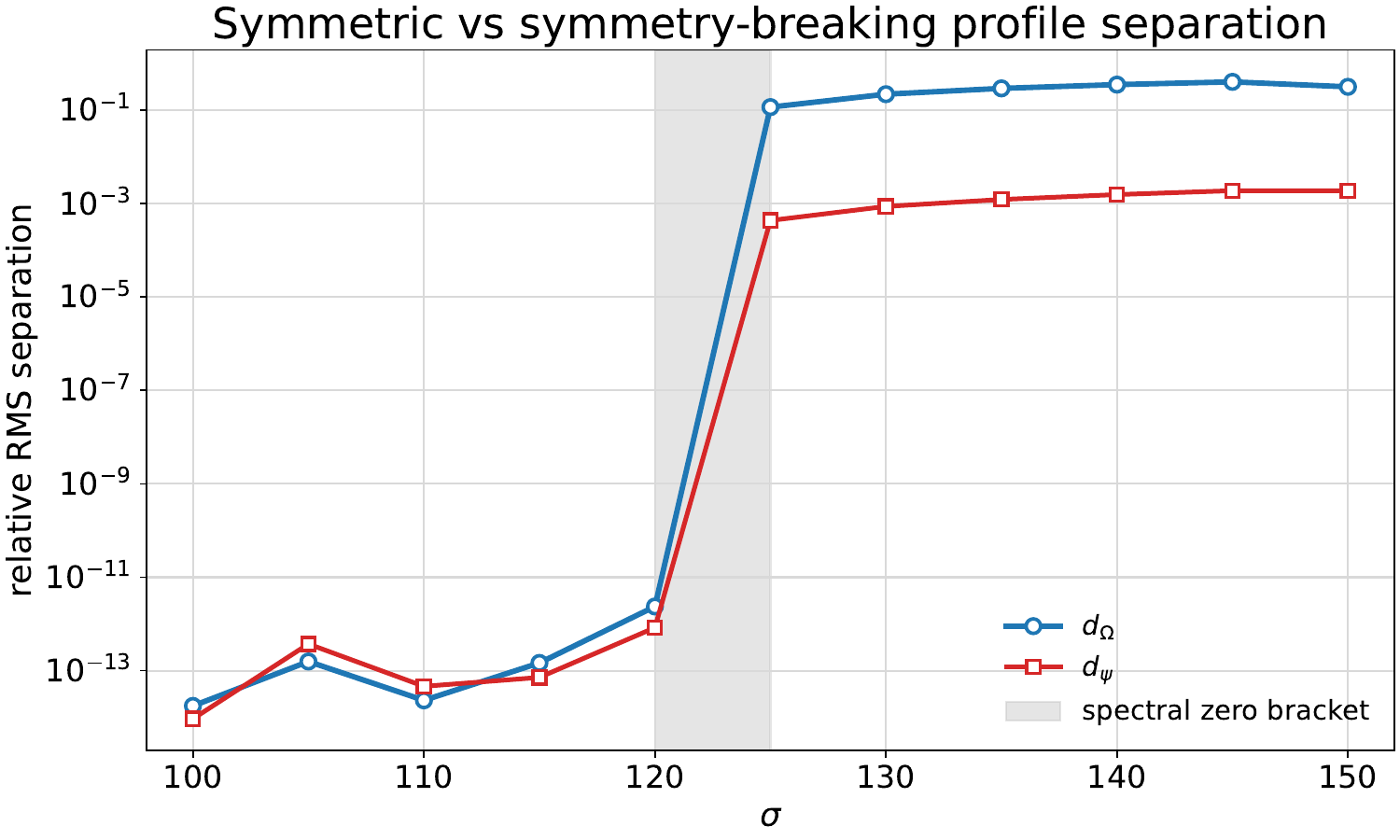}
    \caption{Tracking of the symmetry-breaking branch in $\sigma$ for boundary data $\sigma\cos(3\theta)$ and $\alpha = \frac{5}{8}$. Define the root mean square (RMS) as $\textup{RMS}(u):= \dr{\frac{1}{N}\sum_{j=1}^Nu_j^2}^{\frac{1}{2}}$.
    Then, the separation of vorticity profiles is measure by 
    $\frac{\textup{RMS}(\Omega_{\rm sym} - \Omega_{\rm broken})}{\textup{RMS}(\Omega_{\rm sym})}.$
     The same for the separation of stream function profiles.}
    \label{fig: bifurcation branch tracking}
\end{figure}
Finally, we remark that similar instability and bifurcation results appeared in other numerical experiments with different choices of $\alpha$ and $k$ as well, which implies the generality of such mechanisms for nonuniqueness. We give another example here for interested readers:
\begin{figure}[H]
    \centering
    \includegraphics[width=0.90\textwidth]{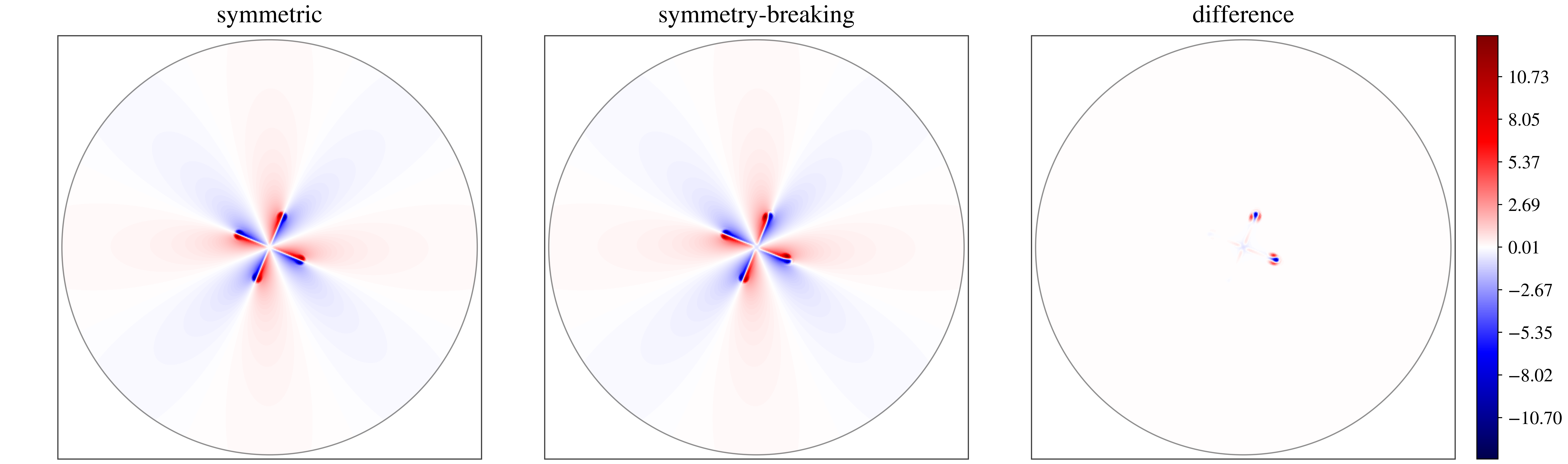}
    \caption{Comparison of the symmetric profile and the symmetry-breaking one for boundary data $240\cos(4\theta)$ and $\alpha = \frac{3}{4}$.}
    \label{fig: profile bifurcation}
\end{figure}

\appendix
\section{Appendix: Proofs of various estimates}\label{Appendix A}

\begin{lemma}[Kato–Ponce inequality]\label{appendix: general produc estimate for Sobolev spaces}
    Let $s>0$, $1<r<+\infty$, $1<p_i,q_i\leq +\infty$, with $\frac{1}{r} = \frac{1}{p_i}+\frac{1}{q_i}$ for $i = 1,2$. Then, it holds that
    \begin{equation}
        \|\Lambda^s(fg)\|_{L^r(\reall^n)} \leq C\dr{\|\Lambda^s f\|_{L^{p_1}(\reall^n)}\|g\|_{L^{q_1}(\reall^n)}+\| f\|_{L^{p_2}(\reall^n)}\|\Lambda^s g\|_{L^{q_2}(\reall^n)}},
    \end{equation}
    where $C$ depends only on $r,s,p_i,q_i$ and the spatial dimension $n$.
\end{lemma}
\begin{proof}
    This is a standard result, and we refer the readers to \cite{KatoPonce}.
\end{proof}

\begin{lemma}\label{appendix: product estimate for H alpha times H alpha to H 2alpha-1}
    Let $0<a,b<1$ and $a+b\geq 1$. Then, there exists some constant $C(a,b)$ such that for any $f\in H^a(\reall^2)$ and $g\in H^b(\reall^2)$, their product $fg$ belongs to $H^{a+b-1}(\reall^2)$ and the following holds
    \begin{align}
        \|fg\|_{H^{a+b-1}(\reall^2)}\leq C(a,b) \|f\|_{H^a(\reall^2)}\|g\|_{H^b(\reall^2)},\\
        \|fg\|_{\dot H^{a+b-1}(\reall^2)}\leq C(a,b) \|f\|_{\dot H^a(\reall^2)}\|g\|_{\dot H^b(\reall^2)}.
    \end{align}
    In other words, the multiplication operator $\mathfrak{m}: H^a(\reall^2)\times H^b(\reall^2)\to H^{a+b-1}(\reall^2)$ is a bounded operator. In particular, multiplication is a bounded operator from $H^\alpha(\reall^2)\times H^\alpha(\reall^2)$ to $H^{2\alpha-1}(\reall^2)$ for $\alpha\in[\frac{1}{2},1)$.
\end{lemma}
\begin{proof}
Denote $s:=a+b-1\geq 0$. First, when $s>0$, we control the quantity $\|\Lambda^s(fg)\|_{\Ltw}$. By Lemma \ref{appendix: general produc estimate for Sobolev spaces}, we have
\begin{equation}
    \|\Lambda^s(fg)\|_{\Ltw}\lesssim \|\Lambda^s f\|_{L^{p_1}(\reall^2)}\|g\|_{L^{q_1}(\reall^2)}+\|f\|_{L^{p_2}(\reall^2)}\|\Lambda^s g\|_{L^{q_2}(\reall^2)}, 
\end{equation}
where
\[
p_1 = \frac{2}{b},\quad q_1 = \frac{2}{1-b},\quad p_2 = \frac{2}{1-a},\quad q_2=\frac{2}{a}.
\]
By the Sobolev embeddings $\dot H^{r}(\reall^2)\hookrightarrow L^{p_r}(\reall^2)$ where $p_r = \frac{2}{1-r}$ for $0<r<1$, we get
\begin{equation}
\|\Lambda^s(fg)\|_{\Ltw}\lesssim\|f\|_{\dot H^a(\reall^2)}\|g\|_{\dot H^b(\reall^2)}.
\end{equation}
Next, we control $\|fg\|_{\Ltw}$. By H\"older's inequality, we obtain
\begin{equation}
    \|fg\|_{\Ltw}\leq \|f\|_{L^{\frac{2}{1-a}}(\reall^2)}\|g\|_{L^{\frac{2}{a}}(\reall^2)}.
\end{equation}
We note that $H^a(\reall^2)\hookrightarrow L^{\frac{2}{1-a}}(\reall^2)$ and $H^b(\reall^2)\hookrightarrow L^{\frac{2}{a}}(\reall^2)$ when $\frac{2}{1-b}\geq \frac{2}{a}$, i.e., $a+b\geq 1$. It follows that
\begin{equation}
    \|fg\|_{\Ltw}\lesssim\|f\|_{H^a(\reall^2)}\|g\|_{H^b(\reall^2)}.
\end{equation}
This completes the proof.
\end{proof}

\begin{lemma}\label{appendix: product estimate for the L infty case}
    Let $f\in H^s(\reall^n)$ for some $s>0$. Denote $\lceil s \rceil$ to be the smallest integer that is bigger than or equal to $s$, and let $g\in W^{\lceil s \rceil,\infty}(\reall^n)$. Then, it holds that
    \begin{equation}
        \|fg\|_{H^s(\reall^n)}\leq C(s)\|f\|_{H^s(\reall^n)}\|g\|_{W^{\lceil s \rceil,\infty}(\reall^n)}.
    \end{equation}
\end{lemma}
\begin{proof}
    The case when $s$ is an integer is straightforward via the chain rule. For the fractional case, we use the Gagliardo characterization:
    \begin{equation}
        \|u\|^2_{H^\sigma(\reall^n)}\sim \|u\|^2_{L^2(\reall^n)}+[u]^2_{H^\sigma(\reall^n)},\quad [u]^2_{H^\sigma(\reall^n)}:=\int_{\reall^n}\int_{\reall^n}\frac{|u(x)-u(y)|^2}{|x-y|^{n+2\sigma}}\;dxdy,\quad 0<\sigma<1.
    \end{equation}
    First, for $f\in H^\sigma(\reall^n)$ and $g\in W^{1,\infty}(\reall^n)$ for $\sigma\in (0,1)$, we have
    \begin{align}
        [fg]^2_{H^\sigma(\reall^n)} &= \int_{\reall^n}\int_{\reall^n}\frac{|f(x)g(x)-f(y)g(y)|^2}{|x-y|^{n+2\sigma}}\;dxdy\\
        &\leq \int_{\reall^n}\int_{\reall^n}\frac{|f(x)|^2|g(x)-g(y)|^2}{|x-y|^{n+2\sigma}}\;dxdy+\int_{\reall^n}\int_{\reall^n}\frac{|f(x)-f(y)|^2|g(y)|^2}{|x-y|^{n+2\sigma}}\;dxdy\\
        &\leq C\|g\|^2_{W^{1,\infty}(\reall^n)}\|f\|^2_{L^2(\reall^n)}+ C\|g\|^2_{L^\infty(\reall^n)}[f]^2_{H^\sigma(\reall^n)}.
    \end{align}
    Based on this result, the general case $f\in H^{m+\sigma}(\reall^n)$ and $g\in W^{m+1,\infty}(\reall^n)$ for $m\in \Z_{\geq 0}$ and $\sigma\in(0,1)$ follow by the chain rule.
\end{proof}

\begin{lemma}[Refined commutator estimate for fractional Laplacian]\label{appendix: refined commutator estimate}
    Let $g_\delta(y) = \frac{1}{(1+\delta^2|y|^2)^\frac{1}{2}}$ and $\alpha\in(0,1)$. Then, for any $f\in L^2(\reall^2)$ and $0\leq \beta<1+\alpha$, we have
    \begin{equation}
        \|g_\delta^{-\beta}[\Lambda^\alpha,g^2_\delta]f\|_{L^2}\leq C(\alpha,\beta)\delta^{\alpha}\|f\|_{L^2(\reall^2)}
    \end{equation}
\end{lemma}
\begin{proof}
    First of all, by the definition of the fractional Laplacian, we have
    \begin{equation}
        [\Lambda^\alpha,g^2_\delta]f(x) = c_\alpha\intt\frac{g^2_\delta(x)-g^2_\delta(y)}{|x-y|^{2+\alpha}}f(y)\;dy =  c_\alpha\intt\frac{\delta^2(|y|^2-|x|^2)f(y)}{|x-y|^{2+\alpha}(1+\delta^2|x|^2)(1+\delta^2|y|^2)}\;dy.
    \end{equation}
    Denote
    \[
       K_{\alpha,\beta}(x,y) := c_\alpha\frac{|y|^2-|x|^2}{|x-y|^{2+\alpha}(1+|x|^2)^{1-\frac{\beta}{2}}(1+|y|^2)}.
    \]
    Then we have
    \[
       g_\delta^{-\beta}[\Lambda^\alpha,g^2_\delta]f(x) = \delta^{2+\alpha}\intt K_{\alpha,\beta}(\delta x,\delta y)f(y)\;dy.
    \]
    If we can show that
    \[
      T_{\alpha,\beta}: f\mapsto \intt K(x,y)f(y)\;dy
    \]
    is a bounded operator from $\Ltw$ to $\Ltw$,
    then we obtain by change of variables $y\mapsto y/\delta$ and $x\mapsto x/\delta$,
    \begin{align}
         &\| g_\delta^{-\beta}[\Lambda^\alpha,g^2_\delta]f\|_{\Ltw} = \delta^{2+\alpha}\dn{\intt K_{\alpha,\beta}(\delta x,\delta y)f(y)}_{\Ltw} = \delta^{\alpha}\dn{\intt K_{\alpha,\beta}(x,\delta y)f(y)}_{\Ltw}\\
         &\quad= \delta^{\alpha-2}\dn{\intt K_{\alpha,\beta}(x,y)f\dr{\frac{y}{\delta}}}_{\Ltw}\leq C(T_{\alpha,\beta})\delta^{\alpha-2}\|f(y/\delta)\|_{\Ltw}=C(T_{\alpha,\beta})\delta^{\alpha}\|f\|_{\Ltw}
    \end{align}
    Now it remains to show that $T_{\alpha,\beta}:\Ltw\to\Ltw$ is bounded. We split the integration in $y$ into three regions: $|y|<\frac{1}{2}|x|$, $|y|>2|x|$, and $\frac{1}{2}|x|\leq |y|\leq 2|x|$.\\
    \textbf{Region $|y|<\frac{1}{2}|x|$:} In this region,
    \[
       \da{\frac{|y|^2-|x|^2}{|x-y|^{2+\alpha}}}\leq C|x|^{-\alpha}.
    \]
    It follows that
    \begin{align}
        \da{\int_{\{y:\,|y|<\frac{1}{2}|x|\}} K(x,y)f(y)\;dy}\leq C\frac{1}{|x|^\alpha(1+|x|^2)^{1-\frac{\beta}{2}}}\intt \frac{|f(y)|}{1+|y|^2}\;dy\leq C\frac{1}{|x|^\alpha(1+|x|^2)^{1-\frac{\beta}{2}}}\|f\|_{\Ltw}.
    \end{align}
    It then suffices to note that $\frac{1}{|x|^\alpha(1+|x|^2)^{1-\frac{\beta}{2}}}\in\Ltw$ when $0\leq\beta<1+\alpha$.\\
    \textbf{Region $|y|>2|x|$:} In this region, 
    since
    \[
         \da{\frac{|y|^2-|x|^2}{|x-y|^{2+\alpha}}}\leq C|y|^{-\alpha}
         \leq C|x|^{-\alpha},
    \]
    we have
    \begin{align}
         \da{\int_{\{y:\,|y|>2|x|\}} K(x,y)f(y)\;dy}\leq C\frac{1}{|x|^\alpha(1+|x|^2)^{1-\frac{\beta}{2}}}\intt \frac{|f(y)|}{1+|y|^2}\;dy\leq C\frac{1}{|x|^\alpha(1+|x|^2)^{1-\frac{\beta}{2}}}\|f\|_{\Ltw},
    \end{align}
    which is bounded in $\Ltw$.\\
    \textbf{Region $\frac{1}{2}|x|\leq |y|\leq 2|x|$:} In this region, we note that
    \begin{align}
        \da{\frac{|x|^2-|y|^2}{|x-y|^{2+\alpha}(1+|x|^2)^{1-\frac{\beta}{2}}(1+|y|^2)}}\leq C\frac{|x|}{|x-y|^{1+\alpha}(1+|x|^2)^{2-\frac{\beta}{2}}}.
    \end{align}
    Then, we get
    \begin{align}
         \da{\int_{\{y:\,\frac{1}{2}|x|\leq|y|\leq2|x|\}} K(x,y)f(y)\;dy}\leq C\frac{|x|}{(1+|x|^2)^{2-\frac{\beta}{2}}}\intt\frac{|f(y)|}{|x-y|^{1+\alpha}}\;dy.
    \end{align}
    By the Hardy–Littlewood–Sobolev inequality ($\|\frac{1}{|x|^{n-s}}*f\|_{L^q(\reall^n)}\leq C\|f\|_{L^p(\reall^n)}$, where $1<p<q<\infty$, $0<s<n$ and $\frac{1}{q} = \frac{1}{p}-\frac{s}{n}$), we obtain
    \begin{align}
        \dn{\intt\frac{|f(y)|}{|x-y|^{1+\alpha}}\;dy}_{L^q(\reall^2)}\leq C\|f\|_{\Ltw},\quad q = \frac{2}{\alpha}.
    \end{align}
    Moreover, since $\frac{|x|}{(1+|x|^2)^{2-\frac{\beta}{2}}}\lesssim \frac{1}{\bkt{x}^{3-\beta}}$ and $(3-\beta)\cdot\frac{2}{2-\alpha}>2$ by the assumption $\beta<1+\alpha$, we know that $\frac{|x|}{(1+|x|^2)^{2-\frac{\beta}{2}}}\in L^{\frac{2}{2-\alpha}}(\reall^2)$. Thus, application of the H\"older's inequality gives
    \begin{align}
        \dn{\frac{|x|}{(1+|x|^2)^{2-\frac{\beta}{2}}}\intt\frac{|f(y)|}{|x-y|^{1+\alpha}}\;dy}_{\Ltw}&\leq C\dn{\frac{|x|}{(1+|x|^2)^{2-\frac{\beta}{2}}}}_{L^{\frac{2}{2-\alpha}}(\reall^2)}\dn{\intt\frac{|f(y)|}{|x-y|^{1+\alpha}}\;dy}_{L^{\frac{2}{\alpha}}(\reall^2)}\\
        &\leq C\|f\|_{\Ltw}.
    \end{align}
    Combining the estimates of the three regions completes the proof.
\end{proof}
\begin{lemma}[uniqueness of the fractional heat equation]\label{appendix: uniqueness of the fractional heat eq}
    Let $u,v\in L^\infty((0,T);L^p(\reall^n))$ for $p\in [1,+\infty]$ be two solutions to the following fractional heat equation
    \begin{align}\label{appendix: fractional heat}
        \pa_t w + (-\Delta)^\alpha w = f,\quad f\in L^1_{loc}((0,T);L^p(\reall^n)),\quad \alpha>0, 
    \end{align}
    such that $u(\cdot,t)\to u_0$ in $L^p(\reall^n)$ as $t\to 0^+$, and $v(\cdot,t)\to v_0$ in $L^p(\reall^n)$ as $t\to 0^+$. If $u_0 = v_0$, then it holds that $u=v$.
\end{lemma}
\begin{proof}
    We denote the fractional heat semigroup by
    \begin{align}
        S(t)g := e^{-(-\Delta)^{\alpha}t}g = H_\alpha(t)*g.
    \end{align}
    Now, we denote $w:=u-v\in L^\infty((0,T);L^p(\reall^n))$. Thus, $w$ is a solution to \eqref{appendix: fractional heat} with $f\equiv 0$ such that $\|w(\cdot,t)\|_{L^p}\to 0$ as $t\to 0^+$. For any function $\phi\in C^\infty_0(\reall^n)$, we define the backward test function as 
    \[
       \psi(s):= S(t-s)\phi.
    \]
    Then, $\psi$ is smooth and satisfies the adjoint backward equation:
    \[
        -\pa_s \psi+ (-\Delta)^\alpha\psi = 0,\quad \psi(t) = \phi.
    \]
    Using the weak formulation for $w$ on $(\eps,t)$ with test function $\psi$, we have
    \[
      \scl{w(t)}{\psi(t)}_{L^2_x} = \scl{w(\eps)}{\psi(\eps)}_{L^2_x},\quad 0<\eps<t.
    \]
    By H\"older's inequality and the $L^{p'}$ contraction of $S(t)$ ($\frac{1}{p}+\frac{1}{p'}=1$), we get
    \begin{align}
        \da{\scl{w(\eps)}{\psi(\eps)}_{L^2_x}}\leq \|w(\eps)\|_{L^p_x}\|\psi(\eps)\|_{L^{p'}_x}\leq \|w(\eps)\|_{L^p_x}\|\phi\|_{L^{p'}_x}\to 0,\quad \text{as}\quad \eps\to 0^+.
    \end{align}
    It follows that 
    \[
       \scl{w(t)}{\psi(t)}_{L^2_x} = \scl{w(t)}{\phi}=0.
    \]
    Since our choice of $\phi\in C^\infty_0(\reall^n)$ is arbitrary and $w(t)\in L^p(\reall^n)$, it follows that $w(t)=0$ a.e., completing the proof.
\end{proof}
\begin{lemma}\label{appendix: estimates of the Oseen kernel}
    Let $\Oc(x,t)$ be the $2$D nonlocal Oseen kernel defined in \eqref{Oseen kernel def}. Then, it holds that
    \begin{equation}
    |\nabla^k_x(-\Delta)^{\frac{\beta}{2}}\Oc(x,t)|\leq C(k,\beta) (t^{\frac{1}{2\alpha}}+|x|)^{-2-\beta-k},\quad \beta+k>-2.
\end{equation}
\end{lemma}
\begin{proof}
    We adopt the approach of the Littlewood–Paley dyadic kernel estimate.
    First of all, we note that by the scaling symmetry $\Oc(x,t) = t^{-\frac{1}{\alpha}}\Oc(t^{-\frac{1}{2\alpha}}x,1)$, it suffices to prove that
    \begin{equation}
        |\nabla^k_x(-\Delta)^{\frac{\beta}{2}}\Oc(x,1)|\leq C(k,\beta)(1+|x|)^{-2-\beta-k},\quad \beta+k>-2.
    \end{equation}
    For brevity, in the following we denote 
    \[
        K(x):= \pa^\gamma(-\Delta)^{\frac{\beta}{2}}\Oc_{ab}(x,1),
    \]
    where $\gamma$ is a multi-index with $|\gamma|=k$.
    By using the Fourier transform, we obtain
    \begin{align}
        m(\xi):= \Fc\dr{K}(\xi) = (i\xi)^\gamma |\xi|^\beta P_{ab}(\xi)e^{-|\xi|^{2\alpha}},\quad    P_{ab}(\xi):= \delta_{ab} - \frac{\xi_a\xi_b}{|\xi|^2}.
    \end{align}
    Then, let $\phi(\xi)\in C^\infty_0(\reall^2)$ be such that $\supp(\phi)\subset \{\xi:3/4\leq |\xi|\leq 8/3\}$ and
    $\sum_{j\in\Z}\phi(2^{-j}\xi)=1.$
    We define $m_j(\xi) := m(\xi)\phi(2^{-j}\xi)$, and $K_j(x):= \Fc^{-1}(m_j)(x)$. Then we have $K = \sum_{j\in\Z}K_j$ in the sense of tempered distributions. In the following, by the pointwise estimates of each $K_j$ we will show that this convergence is actually pointwise, which gives us the desired decay estimate of $K$. Through change of variables $2^{-j}\xi = \eta$, we have
    \begin{align}
        K_j(x) = \intt e^{ix\cdot\xi}m_j(\xi)\;d\xi = 2^{2j}\intt e^{ix\cdot 2^j\eta}m_j(2^j\eta)\;d\eta = 2^{j(2+\beta+k)}\widetilde K_j(2^jx),
    \end{align}
    where we define
    \begin{align}
        \widetilde K_j(z):= \intt e^{iz\cdot \eta}a_j(\eta)\;d\eta,\quad a_j(\eta):= (i\eta)^\gamma |\eta|^\beta e^{-2^{2j\alpha}|\eta|^{2\alpha}}\phi(\eta).
    \end{align}
    We note that $a_j(\eta)$ is a smooth function which is supported where $3/4\leq |\eta|\leq 8/3$. Moreover, for any multi-index $\sigma$, we have
    \[
       \sup_{j\in\Z}\|\pa^\sigma a_j\|_{L^1_\eta}\leq C(\sigma,\alpha,\beta,k),
    \]
    once we notice that $\pa^\sigma a_j$ is nothing but some polynomial in $2^{2j\alpha}$ multiplied by an exponential component $e^{-2^{2j\alpha}}$, which goes to zero as $j\to +\infty$ and is uniformly bounded when $j\to -\infty$. Through integration by parts, for any index $N$ it holds that 
    \begin{align}
        |\widetilde K_j(z)|\leq C(N)\sum_{|\sigma|=N}\|\pa^\sigma a_j\|_{L^1_\eta}(1+|z|)^{-N} \leq C(N,\alpha,\beta,k)(1+|z|)^{-N}.
    \end{align}
    Equivalently, we obtain the pointwise estimate for $K_j$:
    \begin{align}
        |K_j(x)| \leq C(N,\alpha,\beta,k)2^{j(2+\beta+k)}(1+2^j|x|)^{-N}.
    \end{align}
    Now, we fix $j_0$ to be the integer such that $2^{j_0}\leq |x|< 2^{j_0+1}$ and fix $N>2+\beta+k$ (without loss of generality, we assume $|x|\geq 1$). 
    Then, since $2+\beta+k>0$, we have
    \begin{align}
        \sum_{j\in\Z}|K_j(x)| &= \sum_{j\leq -j_0}|K_{j}(x)| + \sum_{j>-j_0}|K_j(x)|\\
        &\leq C(N,\alpha,\beta,k)\dr{\sum_{j\leq -j_0}2^{j(2+\beta+k)}+|x|^{-N}\sum_{j>-j_0}2^{j(2+\beta+k-N)}}\\
        &\leq C(N,\alpha,\beta,k)|x|^{-(2+\beta+k)}.
    \end{align}
    Finally, we obtain
    \begin{align}
        |\pa^\gamma(-\Delta)^{\frac{\beta}{2}}\Oc_{ab}(x,1)|=|K(x)|\leq\sum_{j\in\Z}|K_j(x)|\leq  C(\alpha,\beta,k)(1+|x|)^{-(2+\beta+k)},\quad |\gamma|=k\in\Z_{\geq 0},\quad\beta+k>-2.
    \end{align}
\end{proof}
\begin{lemma}[Fractional Laplacian of homogeneous data]\label{appendix: a technical holder estimate on S1}
    In $\reall^2$, suppose we have $u_0(x)= |x|^{-\beta_0}\bar u_0(\frac{x}{|x|})$ for some $\bar u_0\in C^{m,\beta}(\Sp^1)$, where $\beta_0\in (0,1)$, $m\in \Z_{\geq 0}$ and $\beta\in (0,1]$. Then, for any $\gamma\in (0,\beta)$, there exists some $\tilde u_{0,\gamma}\in C^{m,\beta-\gamma}(\Sp^1)$, such that $(-\Delta)^{\frac{\gamma}{2}}u_0(x) = |x|^{-\beta_0-\gamma}\tilde u_{0,\gamma}(\frac{x}{|x|}) $, and
       \begin{equation}
       \|\tilde u_{0,\gamma}\|_{C^{m,\beta-\gamma}(\Sp^1)}\leq C(m,\beta_0,\beta,\gamma)\|\bar u_0\|_{C^{m,\beta}(\Sp^1)}.
   \end{equation}
\end{lemma}
\begin{proof}
    First of all, we have by the equivalent definition of the fractional Laplacian,
    \begin{align}
        (-\Delta)^{\frac{\gamma}{2}}u_0(x) = c_{\gamma}\intt\frac{u_0(x)-u_0(y)}{|x-y|^{2+\gamma}}\;dy.
    \end{align}
    We remark that the above integral is absolutely convergent (so principal value is not needed), thanks to the fact that $\beta>\gamma$. Consider the polar coordinates $r,\rho\in (0,+\infty)$ and $\theta,\om\in\Sp^1$. By the homogeneity of $u_0$ we can write
    \begin{align}
        (-\Delta)^{\frac{\gamma}{2}}u_0(r\theta) &= c_{\gamma}\int_{0}^{+\infty}\int_{\Sp^1}\frac{u_0(r\theta)-u_0(\rho\om)}{|r\theta - \rho \om|^{2+\gamma}}\rho\;d\om d \rho\\
        &=c_\gamma r^{-\beta_0-\gamma}\int_0^{+\infty}\int_{\Sp^1}\frac{\bar u_0(\theta) - \rho^{-\beta_0}\bar u_0(\om)}{|\theta - \rho\om|^{2+\gamma}}\;d\rho d\om.
    \end{align}
    Thus, $\tilde u_{0,\gamma}$ can be expressed explicitly as
    \begin{align}\label{appendix: S1 holder_explicit u tilde}
        \tilde u_{0,\gamma}(\theta) = c_\gamma\int_0^{+\infty}\int_{\Sp^1}\frac{\bar u_0(\theta) - \rho^{-\beta_0}\bar u_0(\om)}{|\theta - \rho\om|^{2+\gamma}}\;d\rho d\om.
    \end{align}
    Fix $\theta\in\Sp^1$ and split the $\rho$-integral in \eqref{appendix: S1 holder_explicit u tilde}:
    \[
    \int_0^\infty=\int_{(0,1/2)\cup(2,\infty)}+\int_{1/2}^2.
    \]
    Define accordingly $\tilde u_{0,\gamma}=\mathcal F(\theta)+\mathcal N(\theta)$, where
    \[
    \mathcal F(\theta):=c_{\gamma}\int_{(0,1/2)\cup(2,\infty)}\int_{\Sp^1}
    \frac{\bar u_0(\theta)-\rho^{-\beta_0}\bar u_0(\omega)}{|\theta-\rho\omega|^{2+\gamma}}\,
    \rho\,d\omega\,d\rho,
    \]
    \[
    \mathcal N(\theta):=c_{\gamma}\int_{1/2}^2\int_{\Sp^1}
    \frac{\bar u_0(\theta)-\rho^{-\beta_0}\bar u_0(\omega)}{|\theta-\rho\omega|^{2+\gamma}}\,
    \rho\,d\omega\,d\rho.
    \]
    On $(0,1/2)\cup(2,\infty)$ we have $|\theta-\rho\omega|\ge c>0$ uniformly in $(\theta,\omega)$, hence the kernel is
    $C^\infty$ in $(\theta,\omega)$ and integrable in $\rho$. Thus we obtain $\mathcal F\in C^\infty(\Sp^1)$ and
    \begin{equation}\label{eq:F-smooth}
    \|\mathcal F\|_{C^{m,\beta}(\Sp^1)}\le C\,\|\bar u_0\|_{C^{m,\beta}(\Sp^1)}.
    \end{equation}
    So the loss of $\gamma$ regularity comes only from $\mathcal N$. For $\rho\in[1/2,2]$ we write
    \[
    \bar u_0(\theta)-\rho^{-\beta_0}\bar u_0(\omega)
    =\big(\bar u_0(\theta)-\bar u_0(\omega)\big)+\big(1-\rho^{-\beta_0}\big)\bar u_0(\omega),
    \]
    and split $\mathcal N=\mathcal T\bar u_0+\mathcal R\bar u_0$ with
    \[
    (\mathcal T\phi)(\theta):=c_{\gamma}\int_{1/2}^2\int_{\Sp^1}
    \frac{\phi(\theta)-\phi(\omega)}{|\theta-\rho\omega|^{2+\gamma}}\,\rho\,d\omega\,d\rho,
    \]
    \[
    (\mathcal R\phi)(\theta):=c_{\gamma}\int_{1/2}^2\int_{\Sp^1}
    \frac{(1-\rho^{-\beta_0})\phi(\omega)}{|\theta-\rho\omega|^{2+\gamma}}\,\rho\,d\omega\,d\rho.
    \]
    For $\mathcal R$, since $|1-\rho^{-\beta_0}|\lesssim |\rho-1|$ on $[1/2,2]$ and $|\theta-\rho\omega|^2\geq (\rho-1)^2$, the kernel satisfies $\frac{|1-\rho^{-\beta_0}|}{|\theta-\rho\om|^{2+\gamma}}\lesssim \frac{1}{|\theta-\rho\om|^{1+\gamma}}$, which is integrable with respect to $d\rho d\om$. Then, one can easily show that
    \begin{equation}
    \|\mathcal R\phi\|_{C^{m,\beta}(\Sp^1)}\le C\,\|\phi\|_{C^{m,\beta}(\Sp^1)}.
    \end{equation}
    Now, we show the H\"older estimate
    \begin{equation}\label{appendix: S1 Holder_T esti}
    \|\mathcal T\phi\|_{C^{0,\beta-\gamma}(\Sp^1)}\le C\,\|\phi\|_{C^{0,\beta}(\Sp^1)},
    \qquad 0<\gamma<\beta\le 1.
    \end{equation}
    Define
    \[
    J_\gamma(\theta,\omega):=\int_{1/2}^2 \frac{\rho}{|\theta-\rho\omega|^{2+\gamma}}\,d\rho,
    \qquad \theta,\omega\in\Sp^1,
    \]
    so that
    \[
    (\mathcal T\phi)(\theta)=c_{\gamma}\int_{\Sp^1}(\phi(\theta)-\phi(\omega))\,J_\gamma(\theta,\omega)\,d\omega.
    \]
    Let $d:=|\theta-\omega|$ (chord distance). Using $|\theta|=|\omega|=1$,
    \[
    |\theta-\rho\omega|^2 = 1+\rho^2-2\rho\,\theta\cdot\omega
    = (\rho-1)^2+\rho\,|\theta-\omega|^2
    = (\rho-1)^2+\rho d^2.
    \]
    Hence for $\rho\in[1/2,2]$, we get
    \[
    c\big((\rho-1)^2+d^2\big)\le |\theta-\rho\omega|^2\le C\big((\rho-1)^2+d^2\big).
    \]
    Therefore we have
    \[
    J_\gamma(\theta,\omega)\sim \int_{1/2}^2 \big((\rho-1)^2+d^2\big)^{-(2+\gamma)/2}\,d\rho.
    \]
    With the change of variables $s=(\rho-1)/d$ (for $d\in(0,2]$) one obtains
    \begin{equation}\label{appendix: S1 holder_esti of J}
    c\,d^{-1-\gamma}\le J_\gamma(\theta,\omega)\le C\,d^{-1-\gamma}.
    \end{equation}
    It follows that the $L^\infty$-norm of $\Tc\phi$ is bounded:
    \begin{align}
        |(\Tc\phi)(\theta)|\leq C\|\phi\|_{C^{0,\beta}(\Sp^1)}\int_{\Sp^1}\frac{|\theta-\om|^\beta}{|\theta-\om|^{1+\gamma}}\;d\om\leq C\|\phi\|_{C^{0,\beta}(\Sp^1)}.
    \end{align}
    Next, we fix $0<\delta_0\ll 1$ and take any $\theta,\theta'\in\Sp^1$ such that $|\theta - \theta'|:=\delta\leq \delta_0$. We now estimate $(\Tc\phi)(\theta) - (\Tc\phi)(\theta')$. Split $\Sp^1=E_{\rm near}\cup E_{\rm far}$ with
    \[
    E_{\rm near}:=\{\omega:\ |\theta-\omega|\le 2\delta\},\qquad E_{\rm far}:=\{\omega:\ |\theta-\omega|>2\delta\}.
    \]
    On $E_{\rm near}$, by \eqref{appendix: S1 holder_esti of J} we have
    \begin{align}
        \int_{E_{\rm near}}|\phi(\theta)-\phi(\om)|J_{\gamma}(\theta,\om)\;d\om \leq C\|\phi\|_{C^{0,\beta}(\Sp^1)}\int_{|\om-\theta|<2\delta}|\theta-\om|^{\beta-1-\gamma}\;d\om \leq C\|\phi\|_{C^{0,\beta}(\Sp^1)}\delta^{\beta-\gamma},
    \end{align}
    and
    \begin{align}
        \int_{E_{\rm near}}|\phi(\theta')-\phi(\om)|J_{\gamma}(\theta',\om)\;d\om \leq C\|\phi\|_{C^{0,\beta}(\Sp^1)}\int_{|\om-\theta'|<3\delta}|\theta'-\om|^{\beta-1-\gamma}\;d\om \leq C\|\phi\|_{C^{0,\beta}(\Sp^1)}\delta^{\beta-\gamma}.
    \end{align}
    On $E_{\rm far}$, we split
    \[
    (\phi(\theta)-\phi(\om))J_\gamma(\theta,\om)-(\phi(\theta')-\phi(\om))J_\gamma(\theta',\om)
    := A(\om)+B(\om),
    \]
    where
    \[
    A(\omega):=(\phi(\theta)-\phi(\omega))\big(J_\gamma(\theta,\omega)-J_\gamma(\theta',\omega)\big),
    \quad
    B(\omega):=(\phi(\theta)-\phi(\theta'))\,J_\gamma(\theta',\omega).
    \]
    For $A(\om)$, we have
    \[
    \int_{E_{\rm far}} |A(\omega)|\,d\omega
    \leq C \delta\,\|\phi\|_{C^{0,\beta}(\Sp^1)}\int_{|\theta-\omega|>2\delta} |\theta-\omega|^{\beta-2-\gamma}\,d\omega
    \leq C\delta\,\|\phi\|_{C^{0,\beta}(\Sp^1)}\int_{2\delta}^{1} r^{\beta-2-\gamma}\,dr
    \leq C \|\phi\|_{C^{0,\beta}}\,\delta^{\beta-\gamma},
    \]
    where we have used the estimate for $|\nabla_\theta J_\gamma|$ (the derivation of which is similar to that of \eqref{appendix: S1 holder_esti of J}) 
    \begin{align}
        |J_\gamma(\theta,\om) - J_\gamma(\theta',\om)|\leq C|\theta-\theta'| |\theta-\om|^{-2-\gamma},\quad |\om-\theta|>2|\theta-\theta'|.
    \end{align}
    For $B(\om)$, we have
    \[
    \int_{E_{\rm far}} |B(\omega)|\,d\omega
    \leq C |\phi(\theta)-\phi(\theta')|\int_{|\theta-\omega|>2\delta} |\theta-\omega|^{-1-\gamma}\,d\omega
    \leq C \|\phi\|_{C^{0,\beta}(\Sp^1)}\delta^\beta\int_{2\delta}^{1} r^{-1-\gamma}\,dr
    \leq C \|\phi\|_{C^{0,\beta}}\,\delta^{\beta-\gamma}.
    \]
    To summarize, we have obtained \eqref{appendix: S1 Holder_T esti}, and it follows that
    \begin{align}
        \|\tilde u_{0,\gamma}\|_{C^{0,\beta-\gamma}(\Sp^1)}\leq C\|\bar u_0\|_{C^{0,\beta}(\Sp^1)}.
    \end{align}
    The final result (the $C^{m,\beta-\gamma}$-bound) follows once we take derivatives in the angular direction to $u_0$ and note the fact that $(-\Delta)^{\frac{\gamma}{2}}$ is rotation-invariant. 
\end{proof}

\section{Appendix: Numerical implementation details}\label{Appendix B}
Let $B_R(0):= \{y\in\reall^2\mid|y|< R\}$ be our computational domain. By the far field asymptotics of $\Omega$, it is natural to impose that (adopting $(r,\theta)$ as the polar coordinate)
\[
    \Omega(y)\mid_{\pa B_R(0)} = R^{-2\alpha}\bar\om_0(\theta).
\]
As for the stream function profile $\Psi$, we consider the equation it satisfies in the far field:
\[
    \dr{\pa^2_r+\frac{1}{r}\pa_r+\frac{1}{r^2}\pa^2_\theta}\Psi = r^{-2\alpha}\bar\om_0(\theta).
\]
The decaying solution is given by
\[
    \Psi(y) = r^{2-2\alpha}\bar\Psi(\theta),
\]
where $\bar\Psi(\theta)$ is determined by
\[
    \dr{2-2\alpha}^2\bar\Psi(\theta)+\pa^2_\theta\bar\Psi(\theta) = \bar\om_0(\theta).
\]
It can be solved with Fourier expansion:
\begin{equation}
    \hat{\bar\Psi}_m = \frac{1}{\dr{2-2\alpha}^2-m^2}\hat{\bar\om}_{0,m},\quad \bar\Psi = \sum_m \hat{\bar\Psi}_me^{im\theta},\quad \bar\Omega = \sum_m \hat{\bar\om}_{0,m}e^{im\theta},
\end{equation}
assuming that $2-2\alpha$ is not an integer for convenience (otherwise, we need to take resonances into account). Then, we enforce the Neumann boundary conditions:
\[
    \frac{\pa\Psi}{\pa_{\vec n}}\Big|_{\pa B_R(0)} = \dr{2-2\alpha}R^{1-2\alpha}\bar\Psi(\theta).
\]
To summarize, numerically we are to solve the following system
\begin{equation}
\begin{cases}
    (\nabla^\perp\Psi-\frac{1}{2\alpha} y)\cdot\nabla\Omega = \Omega+\Delta\Omega,\quad y\in B_R(0),\quad \Omega\mid_{\pa B_{R}(0)} = R^{-2\alpha}\bar\om_0(\theta)\\[4pt]
    \Delta\Psi = \Omega,\quad y\in B_R(0),\quad \frac{\pa\Psi}{\pa_{\vec n}}\Big|_{\pa B_R(0)} = \dr{2-2\alpha}R^{1-2\alpha}\bar\Psi(\theta).
\end{cases}
\end{equation}
In practice, it is convenient to fix the size of the computational domain by rescaling the spatial variable:
\[
    \Omega(y) = \Omega_*(z),\quad \Psi(y) = \frac{R^2}{R^2_*}\Psi_*(z),\quad z:=\frac{R_* y}{R}.
\]
where $R_*>0$ is fixed and relatively small (e.g. $R_*=10$). Then, the numerical problem becomes
\begin{equation}\label{numerical formulation of the profile solve}
\begin{cases}
    (\nabla^\perp\Psi_*-\frac{1}{2\alpha} z)\cdot\nabla\Omega_* = \Omega_*+\frac{R^2_*}{R^2}\Delta\Omega_*,\quad z\in B_{R_*}(0),\quad \Omega_*\mid_{\pa B_{R_*}(0)} = R^{-2\alpha}\bar\om_0(\theta)\\[4pt]
    \Delta\Psi_* = \Omega_*,\quad z\in B_{R_*}(0),\quad \frac{\pa\Psi_*}{\pa_{\vec n}}\Big|_{\pa B_{R_*}(0)} = \dr{2-2\alpha}R_*R^{-2\alpha}\bar\Psi(\theta).
\end{cases}
\end{equation}
Once we obtained the profile $(\Omega_*,\Psi_*)$, the numerical eigen-problem for the linearized operator is
\begin{equation}\label{numerical formulation of the eigen problem}
    \begin{cases}
         \frac{R^2_*}{R^2}\Delta v+ v - (\nabla^\perp\Psi_*-\frac{1}{2\alpha} z)\cdot\nabla v - \nabla^\perp \phi\cdot\nabla\Omega_* = \lambda v,\quad z\in B_{R_*}(0), \quad v\mid_{\pa B_{R_*}(0)} = 0,\\
        \Delta\phi = v,\quad \frac{\pa\phi}{\pa \vec{n}}\Big |_{\pa B_{R_*}(0)} = 0.
    \end{cases}
\end{equation}
The computations are implemented using the FEniCSx/dolfinx finite element library \cite{Baratta2023DOLFINx}, with triangular meshes generated by Gmsh and linear algebra/eigensolvers provided through PETSc/SLEPc \cite{Hernandez2005SLEPc}. Both $\Omega_*$ and $\Psi_*$ are discretized by continuous piecewise linear Lagrange elements. In the main computation for boundary data $140\cos(3\theta)$, we use $R_*=10$, $R=150$, and a quasi-uniform mesh with $h_{\min}=h_{\max}=0.06$, giving about $10^5$ scalar degrees of freedom. The nonlinear profile problem \eqref{numerical formulation of the profile solve} is solved by continuation in the boundary amplitude from zero to the target value, using 21 equally spaced continuation steps; at each step, Newton iteration is run with tolerance $10^{-8}$, using the previous continuation state as the initial guess. After the profile has converged, the linearized eigenvalue problem \eqref{numerical formulation of the eigen problem} is assembled on the same finite element space and the first ten eigenvalues near the left edge of the spectrum are computed by a shift-invert Krylov--Schur method with shift $-1$, tolerance $10^{-8}$, and maximum iteration count 300. The auxiliary Neumann Poisson solves for $\Psi_*$ and for the eigenfunction stream function $\phi$ are treated by fixing the additive constant with one pinned degree of freedom. For the eigenvalue problem, the homogeneous Neumann equation for $\phi$ is solved in this gauge-fixed discrete Neumann space; the continuous compatibility condition is the zero-mean condition on the vorticity perturbation $v$, and the pinned formulation provides the corresponding discrete convention for inverting the Neumann operator.

\vspace{0.2in}
\noindent
{\bf Acknowledgments}. This research was in part supported by NSF Grants DMS-2205590 and DMS-2512878, the Choi Family Gift fund and Dr. Mike Yan Gift fund.
\bibliographystyle{plain} 
\bibliography{refs} 

\end{document}